\documentclass[A4, 12pt]{amsart}

\usepackage{amsmath}    
\usepackage{graphicx} 
\usepackage{verbatim} 
\usepackage{color} 
\usepackage{subfigure}  
\usepackage{hyperref}  
\usepackage{fullpage}
\usepackage{bbm}
\usepackage{dsfont} 
\usepackage{amssymb}
\usepackage{mathrsfs}
\usepackage{enumerate}
\usepackage[super]{nth}

\newtheorem{theorem}{Theorem}
\newtheorem{proposition}{Proposition}
\newtheorem{lemma}{Lemma}

\newtheorem{corollary}{Corollary}
\newtheorem{eg}{Example}

\newcommand{\Q}{\mathbb{Q}}
\newcommand{\Z}{\mathbb{Z}}

\newcommand{\C}{\mathbb{C}}
\newcommand{\N}{\mathbb{N}}
\newcommand{\F}{\mathbb{F}}
\newcommand{\E}{\mathbb{E}}
\renewcommand{\P}{\mathbb{P}}

\newcommand{\cC}{\mathcal{C}}
\newcommand{\cF}{\mathcal{F}}

\newcommand{\cond}{{\rm cond}}

\newcommand{\Gal}{\mathop{\rm Gal}}
\newcommand{\eqdef}{\mathop{=}^{\rm def}}

\newcommand{\eps}{\varepsilon}
\newcommand{\es}[1]{\begin{equation}\begin{split}#1\end{split}\end{equation}}
\newcommand{\est}[1]{\begin{equation*}\begin{split}#1\end{split}\end{equation*}}
\newcommand{\GL}{{\rm GL}}
\newcommand{\SL}{{\rm SL}}
\newcommand{\ssum}{\mathop{\sum\nolimits^{*_{m}}}}
\newcommand{\primesum}{\mathop{\sum\nolimits^{'}}}

\newcommand{\tr}{{\rm Tr}}
\newcommand{\Aut}{{\operatorname{Aut}}}
\newcommand{\legen}[2]{\left(\frac{#1}{#2}\right)}

\newcommand{\fr}{{\rm frob}}

\makeatletter
\let\@@pmod\pmod
\DeclareRobustCommand{\pmod}{\@ifstar\@pmods\@@pmod}
\def\@pmods#1{\mkern4mu({\operator@font mod}\mkern 6mu#1)}
\makeatother

\title[Elliptic curves and the trace formula]{Elliptic curves over a finite field and the trace formula}
\author{Nathan Kaplan \and Ian Petrow}

\address{Nathan Kaplan -- Department of Mathematics, University of California\\ Irvine, CA  92697\\ \texttt{nckaplan@math.uci.edu} }
\address{Ian Petrow -- Ecole Polytechnique F\'ed\'erale de Lausanne\\ Section des Math\'ematiques \\ 1015 Lausanne, Switzerland\\ \texttt{ian.petrow@epfl.ch} }
\curraddr{ETH Z\"urich - Departement Mathematik HG G 66.4 R\"amistrasse 101 \\  8092 Z\"urich \\
Switzerland \\ \texttt{ian.petrow@math.ethz.ch}}

\dedicatory{Dedicated to Professor B.J. Birch on his 85th birthday}


\begin{document}
\maketitle

\begin{abstract}
We prove formulas for power moments for point counts of elliptic curves over a finite field $k$ such that the groups of $k$-points of the curves contain a chosen subgroup.  These formulas express the moments in terms of traces of Hecke operators for certain congruence subgroups of $\SL_2(\Z)$.    
As our main technical input we prove an Eichler-Selberg trace formula for a family of congruence subgroups of $\SL_2(\Z)$ which include as special cases the groups $\Gamma_1(N)$ and $\Gamma(N)$.  Our formulas generalize results of Birch and Ihara (the case of the trivial subgroup, and the full modular group), and previous work of the authors (the subgroups $\Z/2\Z$ and $(\Z/2\Z)^2$ and congruence subgroups $\Gamma_0(2),\Gamma_0(4)$).  
We use these formulas to answer 
statistical questions about point counts for elliptic curves over a fixed finite field, generalizing results of Vl{\v{a}}du{\c{t}}, Gekeler, Howe, and others.
\end{abstract}

\section{Introduction}
Let $E$ be an elliptic curve defined over a finite field $\F_q$ with $q$ elements. In this paper we always consider elliptic curves up to $\F_q$-isomorphism and whenever speaking of an elliptic curve $E$ we always implicitly mean the isomorphism class of $E$.  
With this convention in mind, let $\cC=\{E/\F_q\}$.  The finite set $\cC$  
is a probability space where a singleton $\{E\}$ occurs with~probability
\est{ \P_q(\{E\}) = \frac{1}{q \# \Aut_{\F_q}(E)}.}  
Let $t_E\in \Z$ denote the trace of the Frobenius endomorphism associated to $E$.  We have $t_E=q+1 -\#E(\F_q)$ and by Hasse's Theorem $t^2_E \leq 4q$.  For a non-negative integer $R$, we consider \est{\E_q(t_E^{2R}) = \frac{1}{q}\sum_{E\in \cC} \frac{t_E^{2R}}{\#\Aut_{\F_q}(E)}.} Birch \cite[equation (4)]{Birch} gave the following explicit formulas for $\E_q(t_E^{2R})$. 
\begin{theorem}[Birch]\label{birch} 
For prime $p\geq 5$ we have 
\est{p\E_p(1) = & p \\ p\E_p(t_E^{2}) = & p^2-1 \\ p\E_p(t_E^{4}) = & 2p^3-3p-1 \\ p \E_p(t_E^{6}) = & 5p^4-9p^2-5p-1 \\ p\E_p(t_E^{8}) = & 14p^5-28p^3-20p^2 -7p-1 \\ p \E_p(t_E^{10}) = & 42p^6 -90p^4-75p^3 -35p^2-9p-1 -\tau(p),} 
where $\tau(p)$ is Ramanujan's $\tau$-function. 
\end{theorem}
\noindent Pairing a curve with its quadratic twist shows that for any $R \ge 0,\ \E_q\left(t_E^{2R+1}\right) = 0$.

To state the general formula for $\E_q(t_E^{2R})$ we introduce some more notation.  For $\Gamma\subset \SL_2(\Z)$ a congruence subgroup we write $S_k(\Gamma, \chi)$ for the $\C$-vector space of classical holomorphic weight $k$ cusp forms of nebentype character $\chi$ for the group $\Gamma$, and for $T$ a linear transformation of this vector space we write $\tr(T | S_k(\Gamma, \chi))$ for its trace.  If the nebentype character is trivial we may omit it from the notation.  Let $T_m$ be the $m$\textsuperscript{th} Hecke operator acting on such a space of cusp forms (see \cite[Ch.~3]{Shimura} or \cite[Ch.~5]{DiamondShurman}).  
Let 
\est{ a_{R,j} \eqdef \frac{2R-2j+1}{2R+1}\binom{2R+1}{j} =\binom{2R}{j} - \binom{2R}{j-1} .}
For $a,b \in \Z$ denote the indicator function of $a=b$ by $\delta(a,b)$, and for $c \in \N$ denote the indicator function of the congruence $a \equiv b \pmod*c$ by $\delta_c(a,b)$.
If $q=p^v$ is a prime power (we allow $v=0$ in which case $T_1$ is the identity operator) we define 
\est{\rho(q,k) \eqdef  -\tr(T_q | S_k(\SL_2(\Z))) + \frac{k-1}{12} q^{k/2-1} \delta_2(v,0)  -\frac{1}{2} \sum_{0 \leq i \leq v } \min(p^i,p^{v-i})^{k-1} + \sigma(q) \delta(k,2),}
and $\rho(p^{-1},k)=0$. In the prime field case, the following formula is \cite[equation (4)]{Birch}. The general finite field case is implicit in the work of Ihara \cite{Ihara}, see also \cite{KaplanPetrow1}.
\begin{theorem}[Birch, Ihara]\label{ihara}
For all $R\geq 0$ and $q=p^v$ with $p$ prime, $\E_q(t_{E}^{2R+1}) = 0$ and 
\begin{multline*}
\E_q(t_E^{2R}) =\sum_{j=0}^R a_{R,j}  q^{j-1} \Big( \rho(q,2R-2j+2)-p^{2R-2j+1}\rho(q/p^2,2R-2j+2)\Big) \\
+ \frac{p-1}{12q}(4q)^R \delta_2(v,0).
\end{multline*}   
In particular, as $q\rightarrow \infty$ we have 
\[
\E_q(t_E^{2R}) \sim C_R  q^{R},
\]
where $C_R = \frac{1}{R+1} \binom{2R}{R}$ is the $R$\textsuperscript{th} Catalan number. 
\end{theorem}
The constants $C_R$ match the moments of the Sato-Tate distribution, and by Carleman's condition since the moments do not grow too fast they determine the limiting probability distribution of the set $\{t_E/(2\sqrt{q})\ |\ E/\F_q\}\subset [-1,1]$
 as $q$ tends to infinity (see e.g. \cite[p. 126]{Koosis}).

In this paper we give a generalization of the theorems of Birch and Ihara where $\SL_2(\Z)$ is replaced with a congruence subgroup, and where we count only elliptic curves over a finite field whose group of rational points contains a subgroup isomorphic to a specified group.  Specifically, let $A$ denote a finite abelian group and let $\Phi_A$ be the function defined on $\cC$ by 
\[
\Phi_A(E) = \begin{cases} 1 & \text{if there exists an injective homomorphism } A \hookrightarrow E(\F_q) \\ 0 & \text{otherwise}. 
\end{cases}
\]  
The main result of this paper, Theorem \ref{MT}, is a generalization of Theorem \ref{ihara} to the expectations
\est{ \E_q(t^{R}\Phi_A) = \frac{1}{q}\sum_{\substack{E\in \cC \\ A \hookrightarrow E(\F_q)}} \frac{t_E^{R}}{\#\Aut_{\F_q}(E)}. }
The added flexibility of the function $\Phi_A$ opens up a host of applications of our Theorem \ref{MT}.  For example, we give an asymptotic formula for the average exponent (also called the first invariant factor) of $E(\F_q)$ over $\cC$.  We discuss this and other applications in Section \ref{applications}.

Before setting up the notation necessary to state Theorem \ref{MT} in full generality, we give two representative special cases.  

\begin{eg}\label{eg1}
Suppose that $q=p$ and $\ell$ are both prime with $\ell \neq p$.  Suppose that $A=\Z/\ell \Z$.  If $p\not \equiv 1 \pmod \ell$ then 
\begin{multline*} p\E_p(t^{2R}\Phi_{A}) = \frac{1}{\ell-1} C_R (p+1)p^{R} \\ - \sum_{j=0}^{R} a_{R,j} p^j  \bigg( \frac{\tr(T_p| S_{2R-2j+2}(\Gamma_1(\ell)))}{\ell-1} + 
\begin{cases} 1 & p  \equiv -1 \pmod*{\ell} \\ 1/2 & p \not \equiv -1 \pmod*{\ell}\end{cases} \bigg)
,\end{multline*} 
and if $p \equiv 1 \pmod \ell$ then
\begin{multline*}p\E_p(t^{2R}\Phi_{A}) =   \frac{\ell}{\ell^2-1}C_R(p+1)p^R -\frac{1}{\ell-1} \sum_{j=0}^{R} a_{R,j} p^j  \bigg( \tr(T_p|S_{2R-2j+2}( \Gamma_1(\ell))) \\  -\frac{1}{\ell+1}\tr(T_p| S_{2R-2j+2}(\Gamma(\ell)))  + 
\frac{1}{4}\left(3+(-1)^\ell\right)\bigg).\end{multline*} 
\end{eg}
The leading term here when $R=0$ gives the probability that an elliptic curve $E/\F_p$ has $\#E(\F_p)$ divisible by a prime $\ell$, a result originally due to Lenstra \cite[Proposition 1.14]{Lenstra}.  

\begin{eg}\label{eg2}
Suppose that $q=p$ and $ \ell$ are both prime with $\ell \neq p$.  Suppose that $A=\Z/\ell \Z \times \Z/ \ell \Z$.  If $p \not \equiv 1 \pmod \ell$ then $\E_p(t^{2R}\Phi_{A}) =0$, and if $p \equiv 1 \pmod \ell$ then 
\begin{multline*}p\E_p(t^{2R}\Phi_{A}) = \frac{1}{\ell(\ell^2-1)}C_R(p+1)p^{R}  -  \frac{1}{\ell(\ell^2-1)}\sum_{j=0}^{R} a_{R,j} p^j  \bigg( \tr (T_p |S_{2R-2j+2}(\Gamma(\ell))) \\  
+\frac{1}{4}(\ell^2-1) (3+(-1)^\ell)\bigg).
\end{multline*} 
\end{eg}
For $E/\F_q$ and $(n,q) = 1,\ E[n] \cong \Z/n\Z \times \Z/n\Z$ over $\overline{\F}_q$.  With respect to a $\Z/n\Z$-module basis of $E[n]$, the action of $\Gal(\overline{\F}_q/\F_q)$ gives a matrix $F \in \GL_2(\Z/n\Z)$ satisfying $\det(F) \equiv q \pmod{n}$ and $\tr(F) \equiv t_E \pmod{n}$.  We see that $\Z/n\Z \hookrightarrow E(\F_q)$ if and only if $1$ is an eigenvalue of $F$, and that $\Z/n\Z \times \Z/n\Z \hookrightarrow E(\F_q)$ if and only if $F$ is the $2\times 2$ identity matrix.  

The leading constants in Example \ref{eg1}, $1/(\ell-1)$ and $\ell/(\ell^2-1)$, have group-theoretic interpretations as the density in the coset $\{g \in \GL_2(\F_\ell) : \det(g) \equiv p \pmod*\ell\} \subset \GL_2(\F_\ell)$
of the matrices with $1$ as an eigenvalue, where the cases are distinguished by whether $p\equiv 1\pmod{\ell}$ or not.  Similarly, the leading constant $1/(\ell (\ell^2-1))$ in Example \ref{eg2} is equal to the density of the identity matrix in this coset (recall that we assume $p\equiv 1 \pmod*{\ell}$ in this case).  We explain a link to the Chebotarev density theorem in Section \ref{alternate}.

\subsection{Statement of the Main Result}
The set of functions  
$\{t^R\}_{R \ge 0}$ is not the most natural basis of continuous functions on the interval $[-1,1]$ in our situation.  Instead, we consider the space $L^2([-1,1],d\mu_\infty)$, where $\mu_\infty$ is the Sato-Tate measure on $[-1,1]$ given by 
\est{ d\mu_\infty = \frac{2}{\pi} \sqrt{1-t^2} \,dt.}  
This $L^2$ space admits a natural orthonormal basis of polynomials, called \emph{Chebyshev polynomials of the second kind}. For $j \geq 0$ these are defined as 
\est{ 
& U_0(t) =  1 \\ & U_1(t) = 2t \\ & U_{j+1}(t) =  2tU_j(t) - U_{j-1}(t).
} 

The Chebyshev polynomials are particularly natural from the monodromy point of view (see  
Section \ref{alternate}), as we now explain.  The underlying group from the monodromy perspective is ${\rm SU}_2$.  This group has a standard $2$-dimensional representation, which we denote by ${\rm Std}$, and has a unique finite dimensional irreducible representation ${\rm Sym}^j$ of dimension $j+1$ for each non-negative integer $j$ given by the $j$\textsuperscript{th} symmetric power of ${\rm Std}$.  Let $\chi_j$ denote the character of the representation ${\rm Sym}^j$. We have that $U_j( \cos \theta) = \chi_j(X_\theta)$ where $X_\theta$ is the conjugacy class in ${\rm SU}_2$ that has eigenvalues $\{e^{i \theta},e^{-i \theta}\}$. 

In this paper we define \emph{normalized} Chebyshev polynomials to be 
\es{\label{chebdef}
U_{k-2}(t,q) = q^{k/2-1} U_{k-2} \left(\frac{t}{2\sqrt{q}}\right)=\frac{\alpha^{k-1}-\overline{\alpha}^{k-1}}{\alpha-\overline{\alpha}} \in \Z[q,t],
}
where $\alpha, \overline{\alpha}$ are the two roots in $\C$ of $X^2-tX+q = 0$. Let 
\es{\label{chebcodef}
c_{R,j} = \begin{cases} a_{R/2,j} & \text{ if } R \text{ even} \\ a_{\frac{R-1}{2},j} + a_{\frac{R-1}{2},j-1} & \text{ if } R \text{ odd}\end{cases}
}
be the \emph{Chebyshev coefficients}.  By an induction argument we have 
\es{\label{cheb} t^{R} = \sum_{j=0}^{\lfloor R/2\rfloor} c_{R,j} q^j U_{R-2j}(t,q).} 
The formula \eqref{cheb} gives a dictionary between statements about the moments $\E_q(t_E^R\Phi_A)$ and the averages of Chebyshev polynomials $\E_q( U_{k-2}(t_E,q) \Phi_A)$, so it suffices to study the latter.  

The Chebyshev coefficients $c_{R,j}$ also have an interpretation in terms of the representation theory of ${\rm SU}_2$.   
We decompose ${\rm Std}^{\otimes k}$ into irreducibles: 
\est{ {\rm Std}^{\otimes k} = \bigoplus_{\substack{0 \leq j \leq k \\ j \equiv k \pmod*2}}\langle {\rm Sym}^j,{\rm Std}^{\otimes k}\rangle {\rm Sym}^j.} 
The multiplicity coefficients $\langle {\rm Sym}^j,{\rm Std}^{\otimes k}\rangle$ are, up to a change of variable, the Chebyshev coefficients defined in \eqref{chebcodef}.  Precisely, we have $c_{R,j} = \langle {\rm Sym}^{R-2j}, {\rm Std}^{\otimes R}\rangle$.

In Section \ref{traces} we prove a version of the Eichler-Selberg trace formula (see Theorem \ref{ESTF}) for the following congruence subgroups.
For positive integers $M \mid N$ let 
\begin{equation}\label{GamNM}
\Gamma(N,M) \eqdef \big\{ \left(\begin{smallmatrix} a& b \\ c & d \end{smallmatrix} \right)\in \SL_2(\Z) \text{ s.t. } a,d\equiv 1 \pmod N, \,\,\, c \equiv 0 \pmod {NM}\big\}.
\end{equation}
Note, for example, that $\Gamma(N,1)=\Gamma_1(N)$ and $\Gamma(N,N) \cong \Gamma(N)$ via conjugation by $\left( \begin{smallmatrix} N & 0 \\ 0 & 1 \end{smallmatrix}\right)$.  The Eichler-Selberg trace formula for $\SL_2(\Z)$
appears in Selberg's original paper on the trace formula \cite{SelbergTF}, the generalization to $\Gamma_0(N),\chi$ under the assumption that the index of the Hecke operator is relatively prime to $N$ was given by Hijikata \cite{Hijikata}, and the general case was achieved by Oesterl\'e \cite{Oest} (see the paper of Cohen \cite{Cohen} for a description).

Let $\psi(n) = [\SL_2(\Z) : \Gamma_0(n)] =  n\prod_{p \mid n}(1+1/p),\ \varphi(n) = |(\Z/n\Z)^\times | = n\prod_{p \mid n}(1-1/p)$, and $\phi(n)= n\prod_{p \mid n } (-\varphi(p)),$ where the products run over primes.  In particular, the function $\phi(n)$ has been chosen to be the Dirichlet convolution inverse to the function $\varphi(n^2)$.  For a positive integer $n$ and prime $p$ we write $v_p(n)$ for the $p$-adic valuation of $n$, and we write $(a,b)$ for the greatest common divisor of $a$ and $b$. 
For $d\in (\Z/n_1\Z)^\times$ let $\langle d \rangle$ denote the $d$\textsuperscript{th} diamond operator acting on $S_k(\Gamma(n_1,n_2))$. The operators $\langle d \rangle$ and $T_q$ with $(q,N)=1$ are normal and pairwise commute. For the definitions of $T_q$ and $\langle d\rangle$ see Section \ref{traces}.

We define a function $T_{n_1,\lambda}(q,d),$ which is motivated by the trace formula, as follows.  We write $q=p^v$ where $p$ is prime and $v$ is a non-negative integer.   
For $\lambda \mid (d^2q-1,n_1)$ (which is well-defined even though $d$ is not an integer) let  
\est{T_{n_1,\lambda}(q,d) = \frac{\psi(n_1^2/\lambda^2) \varphi(n_1/\lambda)}{\psi(n_1^2)} \left( -T_{\text{trace}} + T_{\text{id}} -  T_\text{hyp} + T_{\text{dual}}\right),}
with 
\est{T_\text{trace} = \frac{1}{\varphi(n_1)}\tr( T_q  \langle d \rangle | S_k(\Gamma(n_1,\lambda))),}
\est{T_\text{id} = \frac{k-1}{24}q^{k/2-1} \psi(n_1\lambda) \Big(\delta_{n_1}(q^{1/2} , d^{-1})+ (-1)^k \delta_{n_1}(q^{1/2} , -d^{-1})\Big),}
\begin{multline*}T_\text{hyp} = \frac{1}{4} \sum_{i=0}^v \min(p^i,p^{v-i})^{k-1}   \sum_{\substack{\tau \mid {n_1}\lambda \\  g \mid p^i-p^{v-i} }}  \frac{\varphi(g)\varphi\left({n_1}(\lambda,g)/g\right)} {\varphi({n_1})} \\ 
 \times \Big( \delta_{{n_1}(\lambda,g)/g}(y_i, d^{-1}) + (-1)^k\delta_{{n_1}(\lambda,g)/g}(y_i, -d^{-1}) \Big),\end{multline*}
\est{T_{\text{dual}} =  \frac{\sigma(q)}{\varphi({n_1})} \delta(k,2),}
 and where in the expressions above: 
 \begin{itemize}
 \item if $q$ is not a square then $\delta_{n_1}(q^{1/2} , \pm d^{-1}) = 0$,
\item $g = (\tau,{n_1}\lambda/ \tau)$,
\item $y_i$ is the unique element of $(\Z/({n_1}\lambda/g)\Z)^\times$ such that $y_i \equiv p^i \pmod{\tau}$ and $y_i \equiv p^{v-i} \pmod{{n_1}\lambda/\tau}$,
\end{itemize}
We also define $T_{{n_1},\lambda}(p^{-1},d)=0$. 
 
The main result of this paper is the following.  For a finite abelian group $A$, let $n_1=n_1(A)$ and $n_2=n_2(A)$ be its first and second invariant factors, respectively. That is to say, $n_1$ is the largest order of a cyclic subgroup of $A$.
\begin{theorem}\label{MT} 
Let $A$ be a finite abelian group of rank at most $2$.  Suppose that $(q,|A|)=1$ and $k \geq 2$.  If $q\equiv 1 \pmod {n_2(A)}$ we have 
\begin{multline*}
\E_q(U_{k-2}(t_E,q)\Phi_A) =  \frac{1}{q\varphi(n_1/n_2)} \sum_{\nu \mid \frac{(q-1,n_1)}{n_2}} \phi(\nu) \left(T_{n_1,n_2 \nu}\left(q,1\right) - p^{k-1}T_{n_1,n_2 \nu}\left(q/p^2,p\right)\right)
\\ + q^{k/2-1} \frac{(p-1)(k-1)}{24q}\left( 
\delta_{n_1}(q^{1/2}, 1)+  (-1)^{k}
\delta_{n_1}(q^{1/2}, -1)\right) 
\end{multline*}
and if $q\not \equiv 1 \pmod {n_2(A)}$ then $\E_q(U_{k-2}(t,q)\Phi_A)=0$.
 \end{theorem}
The following special case gives the flavor of the general formula.  
\begin{eg}
Suppose that $q=p$ and $\ell \neq p$ are primes.  Suppose that $A$ is a finite abelian group with rank at most $2$. 
Also suppose that $p \equiv 1 \pmod*{n_2(A)}$ and $(p-1,n_1(A))=n_2(A)$. Then
\begin{multline*}
\E_p(U_{k-2}(t_E,p)\Phi_A) = \frac{\psi(n_1^2/n_2^2)}{p\psi(n_1^2)\varphi(n_1) } \bigg( (p+1)\delta(k,2)- \tr(T_p | S_k(\Gamma(n_1, n_2)))  \\ 
 -\frac{1}{4}\left(1+ (-1)^k \delta(n_1,2) + \delta_{n_1}(p,1) + (-1)^k \delta_{n_1}(p,-1)\right) \varphi(n_1)\sum_{\substack{\tau \mid n_1n_2 \\ 
 \left(\tau , \frac{n_1n_2}{\tau}\right)\mid n_2}}  \varphi\left( \left(\tau , \frac{n_1n_2}{\tau}\right)\right)\bigg).
\end{multline*}
\end{eg}
Remarks:
\begin{enumerate}
\item 
The formula in Theorem \ref{MT} looks complicated but it is quite usable.  As $q\to \infty$ there are very few terms on the right hand side of this formula as compared to the left hand side.  Of these terms, only the ``trace'' term is mysterious, as the ``identity'', ``hyperbolic'', and ``dual'' terms are all given in terms of straightforward arithmetic functions.  

One knows quite a lot about the spaces of classical cusp forms $S_k(\Gamma(n_1,n_2))$, hence about the traces  
appearing in Theorem \ref{MT}. 
This gives a lot of information that is inaccessible from the starting formula on the left hand side.  As a first example, one can use Deligne's bound on individual Hecke eigenvalues and an estimate for the dimension of $S_k(\Gamma(n_1,n_2))$ to obtain estimates as $q \to \infty$ on the left hand side that are significantly better than the trivial upper bound of $\ll q^{k/2-1}$. We explain this and several other 
corollaries of Theorem \ref{MT} in greater detail in Section \ref{applications}.
\item 
As previously remarked, the special case $A= 1$ (the trivial group) of Theorem \ref{MT} goes back to work of Birch \cite{Birch} and Ihara \cite{Ihara}.  The cases $A = \Z/2\Z$ and $A = \Z/2\Z \times \Z / 2 \Z$ were previous work of the authors in \cite{KaplanPetrow1}, and the case of $A= \Z/n\Z \times \Z/n \Z$ over the prime field $\F_p$ in weight $k=2$ was given by Kowalski  
in terms of $\F_p$-points on the modular curve $Y(n)$ \cite[Section 6.2]{KowalskiAPEC}.  
\item Theorem \ref{MT} requires the hypothesis that $(q,|A|)=1$.  We are cautiously optimistic that the methods of this paper could be adapted to treat the case $p \mid |A|$ as well.  However, doing so would require working out in parallel to the already lengthy computations in Sections \ref{curves2classnumbers}, \ref{traces}, and \ref{assembly} the special case where $A$ is a $p$-group.  To avoid dividing our argument into additional cases we suppose that $p \nmid |A|$ in this paper.
\item 
As was noted in  
\cite[Section 6.2]{KowalskiAPEC}, the hyperbolic terms $T_{\text{hyp}}$ in Theorem \ref{MT} may have an interpretation in terms of the  
cusps on the modular curve $X(n_1,n_2)$ corresponding to the group $\Gamma(n_1,n_2)$.  For example, in Example \ref{eg2} the ``hyperbolic terms'' are 
\[
\frac{1}{4} \left(3+(-1)^\ell\right)\frac{\ell^2 - 1}{\ell(\ell^2-1)},
\] 
which is the number of cusps of $\Gamma(\ell)$ (see e.g.~\cite[Section 3.8]{DiamondShurman}), divided by $|\SL_2(\F_\ell)|$.
\item Theorem \ref{MT} could presumably be obtained working directly with modular curves, using Deligne's equidistribution theorem and the Chebotarev density theorem (see, e.g.~Katz-Sarnak \cite[Chapter 9]{KatzSarnak}).  For example, Howe obtains the main term for a similar counting problem in this way \cite{Howe}.  Our results generalize Howe's in that Theorem \ref{MT} is an explicit (not only asymptotic) formula in terms of traces of Hecke operators, and also in that we give formulas for all Chebyshev polynomials (not only $U_0$).  We give a few more details on the geometric approach in 
Section \ref{alternate}.  In 
Section \ref{apps1} and in Theorem \ref{MTn1} we give applications of these formulas for $k\geq 3$, and in 
Section \ref{apps3} we discuss applications of the explicit formulas for the ``error terms''.
\item 
On the other hand, our approach has some considerable advantages over the monodromy approach:  
\begin{itemize} 
\item The natural geometric setting is probably about algebraic stacks (note the automorphisms in the probability function $\P_q$) and the analytic approach here  hides the complications involved in that theory.  Dealing with the primes 2 and 3 geometrically might be quite involved, but our approach is not much more difficult for these primes.  
\item Any geometric proof that could reproduce the exact formulas in Theorem \ref{MT} would necessarily be quite delicate and complicated.
\item It is by no means clear that a geometric proof dealing only with the asymptotic order of magnitude would be as strong quantitatively in its applications, i.e.~that the error term in Corollary \ref{cor} below would be as uniform in $n_1$ and $n_2$.
\end{itemize}
\end{enumerate}

\subsection{Alternate Approach}\label{alternate}
In this  
section we sketch an alternate approach to some of our results using geometric techniques and monodromy computations.  This approach naturally explains some features of the computations that follow in the rest of the paper, and extends (at least to asymptotic order) to variants of the problem where it is not clear a trace formula approach would work.  

Consider any one-parameter family of elliptic curves $(E_\lambda)$ defined over $\F_q$ and consider those $\lambda$ such that $E_\lambda(\F_q)$ contains a subgroup isomorphic to $A$. The following proposition describes the special case where $(E_\lambda)$ is the Legendre family.  It was suggested by an anonymous referee.  
\begin{proposition}\label{alternateapproachprop}
Let $\mathcal{E} \to \mathbb{A}^1-\{0,1\}$ be the family of elliptic curves over $\Z[1/2]$ given by the Legendre family 
\[
E_\lambda : y^2 = x(x-1)(x-\lambda).
\]
For a finite field $\F_q$ with $(q,2)=1$ and any $\lambda \in \F_q -\{0,1\}$, let $t(\lambda)$ be the corresponding trace of Frobenius.  Let $A$ be a fixed finite abelian group of rank at most $2$ and with odd order.  As $q\to \infty$ though any sequence with $(q,2|A|)=1$ and $q \equiv 1 \pmod*{n_2(A)}$ the finite sets 
\[
\left\{\frac{t(\lambda)}{\sqrt{q}} : \lambda \in \F_q -\{0,1\} \,\text{ and } \,A \hookrightarrow E_\lambda(\F_q)\right\}
\]
become equidistributed with respect to the Sato-Tate measure.
\end{proposition}
\begin{proof}[Proof Sketch]
Let $E$ be an elliptic curve defined over $\F_q$. Let $A$ be a fixed finite abelian group of rank at most $2$, and let $n_1=n_1(A)$ and $n_2(A) = n_2$ be its first and second invariant factors.  Let $n=n_1n_2$.  The condition that $A \hookrightarrow E(\F_q)$ can be detected via the $\Gal(\overline{\F}_q/\F_q)$-action on $E[n]$.  Specifically, this action determines a well-defined conjugacy class $(\fr_{q}) \subset \GL_2(\Z/n\Z)$, and $A \hookrightarrow E(\F_q)$ if and only if $(\fr_q)$ has one eigenvalue that is $1$ modulo $n_1$ and $(\fr_{q})$ is trivial modulo $n_2$.  We say a conjugacy class of $\GL_2(\Z/n\Z)$ is of \emph{type $A$} if these conditions are satisfied.

Let $K=\F_q(t)$, $G_\text{ar}=\Gal(K^{\rm sep}/K)$ and $G_\text{geo}=\Gal(K^{\rm sep}/K.\overline{\F}_q)$ be the arithmetic and geometric Galois groups of $K$.  To show that the distribution of $t(\lambda)$ is independent of the Galois action on $n$-torsion points, we fix a prime $\ell \nmid q$ and consider two Galois representations: 
\begin{enumerate}
\item $$\rho: G_\text{ar} \to \GL( V),$$ given by the $\ell$-adic sheaf $\mathcal{F} =R^1 \pi_* \Q_\ell(1/2)$ lisse outside of $\{0,1,\infty \}$ (see e.g. \cite[Section 9.1.11]{KatzSarnak}) 
where $\pi :  \mathcal{E} \to \P^1$, and $V$ is the 2-dimensional $\Q_\ell$-vector space given intrinsically by the stalk $\mathcal{F}_{\overline{\eta}}$ for $\overline{\eta}$ a geometric generic point of $\P^1$. 
\item $$\rho_n:  G_\text{ar} \to \GL_2(\Z/n\Z)$$ given by the Galois action on $n$-torsion.
\end{enumerate}

Let $U\subset \P^1$ be the affine open of $\P^1$ avoiding $\{0,1,\infty\}$.  For each $\lambda \in U(\F_q)$ we have inertia and decomposition groups $I_\lambda \trianglelefteq D_\lambda \leq G_\text{ar}$ and an element $\fr_{\lambda,q} \in D_\lambda/I_\lambda \simeq \Gal(\overline{\F}_q/\F_q)$, which is the inverse image of the geometric Frobenius. 
Therefore for each $\lambda \in U(\F_q)$ we have a well-defined conjugacy class $\fr_{\lambda,q}$ in a quotient of $G_\text{ar}$ (the conjugacy class $\fr_{\lambda,q}$ in $G_\text{ar}$ itself depends on a choice of lift modulo $I_\lambda$).  We have in particular for all $\lambda \in U(\F_q)$ that $\tr \rho (\fr_{\lambda,q}) = -t(\lambda)/\sqrt{q},$ and that $\rho_n(\fr_{\lambda,q}) = (\fr_q)$, the conjugacy class of $\GL_2(\Z/n\Z)$ defined via the curve $E_\lambda$ in the first paragraph of this proof.

By Lemma \ref{equilem} (see Section \ref{apps1}), our goal is to show that for each $j \geq 1$ we have 
$$\frac{\sum_{\lambda \in U(\F_q)} U_j(t(\lambda)/2\sqrt{q}) \Phi_A(E_\lambda) }{\sum_{\lambda \in U(\F_q)} \Phi_A(E_\lambda)} \sim 0 $$ 
as $q \rightarrow \infty$.  To do this, we use the Lefschetz trace formula and the Riemann hypothesis of Deligne to estimate the quantity in the numerator.   

We begin by calculating the monodromy groups of $\rho$ and $\rho_n$.  By theorems of Deligne and because of the normalization by the Tate twist we have (see \cite[Section 10.1.16]{KatzSarnak}) that 
$$\overline{\iota \rho(G_{\text{geo}})} = \overline{\iota \rho(G_{\text{ar}})} = \SL_2(\C)$$
where $\iota:\Q_\ell \to\C$ is a fixed complex embedding and the bar denotes the Zariski closure.  We also let $G_n = \rho_n(G_\text{ar})$ and $G_{n,0} = \rho_n(G_{\text{geo}}) \trianglelefteq G_n$.  The geometric monodromy group $H$ of $\rho \times \rho_n$ is contained in $\SL_2 \times G_{n,0}$.  We claim it is equal to this group.   
The connected component of the identity $H^0$ is of the same dimension as $\SL_2(\C)$ so equals $\SL_2(\C) \times \{1\}$.  Since $H$ also surjects onto $G_{n,0}$ we conclude that $H= \SL_2(\C) \times G_{n,0}$.

Let $\Lambda$ be any irreducible representation of $\SL_2(\C)$ and $\pi$ any irreducible representation of $G_n$.  If $\Lambda \otimes \pi$ is not trivial on the geometric monodromy group of $\rho \times \rho_n$ we apply the Grothendieck-Lefschetz trace formula and the Riemann hypothesis of Deligne (see e.g.~\cite[9.2.6(2) and 9.2.6(3)]{KatzSarnak}) to deduce 
\es{\label{rh} 
\sum_{\lambda \in U(\F_q)} \tr (\Lambda (\rho (\fr_{\lambda,q}))) \tr (\pi (\rho_n(\fr_{\lambda,q}))) \ll_{\Lambda,n} \sqrt{q}.
}   
Here the implied constant depends on $n$ and $\Lambda$ but it can be shown using the Grothendieck-Ogg-Safarevich formula that this constant is $\ll |G_n| \dim (\Lambda) \dim(\pi) \ll |G_n|^{3/2} \dim (\Lambda),$ with an absolute implied constant.

Choosing $\Lambda$ to be the $j$\textsuperscript{th} symmetric power representation, we get $\tr (\Lambda (\rho(\fr_{\lambda,q}))) = U_j(t(\lambda)/2\sqrt{q})$.  By the first paragraph of the proof, we expand $\Phi_A(E_\lambda)$ spectrally into finite dimensional irreducible representations of $G_n$:
\[
\Phi_A(E_\lambda) = \begin{cases} 1 & \text{ if } \rho_n(\fr_{\lambda,q}) \text{ is of type } A \\ 0 & \text{ else}\end{cases}= \sum_{\pi}  c(A,\pi) \tr (\pi(\rho_n(\fr_{\lambda,q}))).
\]
We then have by \eqref{rh} that  
\[
\sum_{\lambda \in U(\F_q)} U_j(t(\lambda)/2\sqrt{q}) \Phi_A(E_\lambda)  \ll_{A,j} q^{-1/2}.
\]

The Chebotarev density theorem over finite fields (see \cite[Sections 9.7.10 and 9.7.11]{KatzSarnak})
 implies that 
\es{\label{chebotarev} \sum_{\lambda \in U(\F_q)} \Phi_{A}(E_\lambda) = (q-2) \frac{|\{g \in G_n : \overline{g} \text{ has type } A\}|}{|G_{n,0}|}+ O_{n} ( \sqrt{q}),} 
where $\overline{g}$ is the canonical projection $G_n \to G_{n}/G_{n,0}$.  To show equidistribution with respect to the Sato-Tate measure, it therefore suffices to show that the numerator in \eqref{chebotarev} is non-zero.    

A result of Igusa \cite[Theorem 3]{Igusa} states that for $n$ odd we have $G_{n,0}= \SL_2(\Z/n\Z)$.  Therefore the cosets $G_n/G_{n,0}$ correspond to determinants of elements, and as $\det(\fr_{\lambda,q}) = q^{-1}$ we have that the matrix 
\[
\left(\begin{array}{cc} q & 0 \\ 0 & 1 \end{array}\right) \in G_n \subseteq \GL_{2}(\Z/n\Z)
\]
has one eigenvalue equal to $1$ and is the identity modulo $n_2$, by our assumption that $q\equiv 1\pmod*{n_2(A)}$.  Therefore, the main term of \eqref{chebotarev} does not vanish, which completes the proof.

\end{proof}
\subsection{Sketch of the Proof of Theorem \ref{MT}}\label{sketchproof}
For $A$ an abelian group and $t \in \Z$ let $\mathcal{C}(A)= \{E/\F_q: A\hookrightarrow E(\F_q) \}$ and $\mathcal{C}(A,t) = \{ E/\F_q : A\hookrightarrow E(\F_q) \text{ and } t_E=t \}$.  These sets are empty unless $A$ is finite of rank at most $2$ and $t^2 \leq 4q$.  

\textbf{Step 1.} The first step in the proof of Theorem \ref{MT} is to fiber $\cC(A)$ over isogeny classes, which are parameterized by $t\in \Z$ with $t^2\leq 4q$.  We have 
\begin{eqnarray}\label{step1}
\E_q(U_{k-2}(t_E,q) \Phi_A) & = & \frac{1}{q}
\sum_{\substack{E/\F_q\\ A\hookrightarrow E(\F_q)}}  \frac{U_{k-2}(t_E,q)}{\#\Aut_{\F_q}(E)} = \sum_{t^2 \le 4q} U_{k-2}(t,q) 
 \left(\frac{1}{q}\sum_{\substack{t_E = t \\ A \hookrightarrow E(\F_q)}} \frac{1}{\#\Aut_{\F_q}(E)}\right) \nonumber \\
 & = & \sum_{t^2 \le 4q} U_{k-2}(t,q) \P_q(\cC(A,t)).
\end{eqnarray}
\textbf{Step 2.} In Section \ref{curves2classnumbers}, the sizes of the fibers $\P_q(\cC(A,t))$ are expressed in terms of sums of ideal class numbers of orders in imaginary quadratic fields.
When $A = \Z/n \Z \times \Z/n\Z$ such results go back to Deuring \cite{Deu}, Lenstra \cite{Lenstra}, Schoof \cite{Schoof}, and Waterhouse \cite{Waterhouse}, and are stated as Lemmas \ref{S46withweights}, \ref{Schoof1} and \ref{Schoof2} below.  We extend these results to general rank at most 2 finite abelian groups in Proposition \ref{CurveCount}.  To state these results precisely, we introduce in \eqref{cndef} certain class numbers $H_{n_1,n_2}(t,q,d)$ where $n_1=n_1(A)$ and $n_2=n_2(A)$.  

\textbf{Step 3.} In Section \ref{traces} we prove a version of the Eichler-Selberg trace formula for the groups $\Gamma(n_1,n_2)$ by summing the formula for $\Gamma_0(n),\chi$ over nebentype characters $\chi$.  The computation is very explicit and delicate, but in the end it turns out that the elliptic term involves a similar sum over the same class numbers $H_{n_1,n_2}(t,q,d)$.  The full formula is given in Theorem \ref{ESTF}.

\textbf{Step 4.} The sum over isogeny classes $t$ in \eqref{step1} is not exactly the same as the elliptic terms in the trace formula (Theorem \ref{ESTF}), but they are close enough that we can compare them explicitly in Section \ref{assembly}.  After some manipulation, the trace formula captures all of the elliptic curves in \eqref{step1} where the ring of endomorphisms \emph{over the base field} is an order in a imaginary quadratic field, and not only the ordinary curves.  The isogeny classes corresponding to $t^2=4q$ are the only ones where the endomorphism rings over the base field are orders in quaternion algebras, and these have to be considered separately in the proof of Theorem \ref{MT}.  We remark also that the comparison in Section \ref{assembly} is particularly simple in the case that $\F_q$ is a prime field with $p\geq 5$.  

\subsection{Index of Notation}

In this section we give an index of notation used throughout the paper.  If definitions are easy to state, we give them here, but for more involved notation we give only a location for where a full definition can be found.

\begin{center}
 \begin{tabular}{@{} | p{3cm} |  p{10cm}  |  p{1.6cm} |} 
\hline
Notation & Definition & Location \\
\hline\hline
$q$ & A prime power & Page 1\\
$\F_q$ & A finite field with $q$ elements & Page 1\\
$p$ & The characteristic of $\F_q$ & Page 2\\
$\ell$ & A prime $\ell \neq p$ & Page 3\\
$t$ & An integer satisfying $t^2 \leq 4q$ & Page 9\\
$A$ & Finite abelian group of rank at most $2$ & Page 3\\
$\prod_p$, $\prod_\ell$ & Products over primes & Page 5\\
$\sigma(n)$ & Sum of divisors function: $\sigma(n) = \sum_{d\mid n} d$ & Page 2 \\
$\varphi(n)$ & Euler phi function: $\varphi(n) = n\prod_{p \mid n}(1-1/p)$ & Page 5 \\
$\psi(n)$ & $n\prod_{p \mid n}(1+1/p)$ & Page 5 \\
$\phi(n)$ & $n\prod_{p \mid n } (-\varphi(p))$ &  Page 5\\
$\mu(n)$ & M\"obius function & Page 19\\
$v_p(n)$ & $p$-adic valuation of $n$ &  Page 5\\
$(a,b)$ & Greatest common divisor of $a$ and $b$ &  Page 5 \\
$\tau(n)$ & Number of divisors of $n$ &  Page 15\\
$\omega(n)$ & Number of distinct prime factors of $n$ & Page 24\\
$\lambda(n)$ & $(-1)^{\omega(n)}$ & Page 24 \\
$\lfloor x \rfloor$ & The greatest integer $\leq x$ & Page 14\\
$\delta(a,b)$ & Indicator function of $a = b$  & Page 2\\
$\delta_c(a,b)$ & Indicator function of $a \equiv b \pmod{c}$ & Page 2\\
$D(t;n)$ & $\delta_n(dq+d^{-1},t)$ & Page 23\\
$O_x$, $\ll_x$ & Big O notation. The constant may depend on $x$ & Page 13\\
$a_{R,j}$ & $\binom{2R}{j} - \binom{2R}{j-1}$ & Page 2\\
$c_{R,j}$ & Chebyshev coefficients &  Page 4 \\
$C_R$ & Catalan number: $a_{R,R}$ & Page 2\\
$U_j(t)$ & Chebyshev polynomials of the second kind & Page 4\\
$U_{k-2}(t,q)$ & Normalized Chebyshev polynomials & Page 4\\
$n_1(A),\ n_2(A)$ & First and second invariant factors of $A$ & Page 6\\
$n_1(E),\ n_2(E)$ & First and second invariant factors of $E(\F_q)$ & Page 14\\
$E/\F_q$ & Isomorphism class of the elliptic curve $E$ over $\F_q$ & Page 1\\
$t_E$ & Trace of the Frobenius endomorphism associated to $E$ & Page 1 \\
$\cC$ & Set of all isomorphism classes of elliptic curves over $\F_q$ & Page 1 \\
$\cC(A)$ & Set of all $E/\F_q$ where $\Phi_A(E) = 1$ & Page 9\\
$\cC(A,t)$ & $\{E/\F_q \colon \Phi_A(E) = 1 \text{ and } t_E = t\}$ & Page 9 \\
$\Phi_A$ & Indicator function on $\cC$ of $A \hookrightarrow E(\F_q)$ & Page 2\\
$\P_q(*)$ & Probability measure on $\cC$ & Page 1 \\
$\E_q(*)$ & Expectation of a random variable on $\cC$ & Page 1\\
$v(n_1,n_2)$ & Constant for main term from Corollary 1 & Page 13\\
$h(d)$ & Class number of the quadratic order of discriminant $d$ & Page 21 \\
$h_w(d)$ & Class number weighted by size of unit group & Page 21 \\
$H(\Delta)$ & Hurwitz-Kronecker class number & Page 21\\
$H_{n_1,n_2}(t,q,d)$ & Modified class numbers & Page 24\\
$H^*_{n_1,n_2}(t,q,d)$ & Class numbers for supersingular contribution & Page 26\\
$\legen{a}{n}$ & Kronecker symbol & Page 21\\
\hline
\end{tabular}
\end{center}

\begin{center}
 \begin{tabular}{@{} | p{3cm} |  p{10cm}  |  p{1.6cm} |} 
\hline
Notation & Definition & Location \\
\hline\hline
$\Gamma(N,M)$ &  Congruence subgroup of \eqref{GamNM} & Page 5 \\
$T_m$ & $m$\textsuperscript{th} Hecke operating acting on a space of cusp forms & Page 28\\
$\langle d \rangle$ & $d$\textsuperscript{th} diamond operator acting on $S_k(\Gamma(n_1, n_2))$ & Page 28\\
$S_k(\Gamma, \chi)$ & Space of weight $k$ cusp forms of character $\chi$ for $\Gamma$ & Page 2 \\
$\tr(T|S_k(\Gamma,\chi))$ & Trace of $T$ acting on the vector space $S_k(\Gamma,\chi)$ & Page 29\\
$T_{n_1,\lambda}(q,d)$ & Function motivated by the trace formula & Page 5\\
$\rho(q,k)$ & A special case of $T_{n_1,\lambda}(q,d)$ & Page 2\\
$\omega_{A}(q,d),\ \omega_{A}^*(q,d)$ & Functions defined in terms of class numbers & Page 27 \\
$\Sigma_{n_1, n_2}(q,d)$ & A sum of class numbers & Page 53\\
$T_{\text{trace}}$ & Part of $T_{n_1,\lambda}(q,d)$ & Page 5\\
$T_{\text{id}}$ & Part of $T_{n_1,\lambda}(q,d)$ & Page 5\\
$T_{\text{hyp}}$ & Part of $T_{n_1,\lambda}(q,d)$ & Page 5\\
$T_{\text{dual}}$ & Part of $T_{n_1,\lambda}(q,d)$ & Page 5\\
$T^{(*)}_{\chi}$  & Parts of Theorem 10, where $*=i,e,h,$ or $d$ & Page 31\\
$\nu || n_1$ & Full divisor: for all primes $\ell \mid \nu,\ v_{\ell}(\nu) = v_{\ell}(n_1)$ & Page 23\\
$a \mid b^\infty$ & $b$ is divisible by each prime dividing $a$ & Page 23\\
$\mu \prec \nu$ & A partial order on integers, see Section 3 & Page 23\\
$D_{\nu,\mu}(t)$ & A function related to $D(t;n)$, see Section 3 & Page 23\\
$\mu_{\chi}(t,m,q)$ & See Theorem 10 & Page 30\\
$\sum^{*m}$ & See Theorem 10 & Page 30\\
$\Delta$ & $t^2-4q$ & Page 32\\
$\Delta^*$  & $\Delta/4$ & Page 36\\
$m$ & An integer such that $m^2 \mid \Delta$ & Page 32\\
$\alpha$ & An integer such that $0 \leq \alpha \leq \beta \leq \gamma$ & Page 32\\
$\gamma$ & An integer such that $0 \leq \alpha \leq \beta \leq \gamma$ & Page 32\\
$\beta$ & The integer such that $(d^2q-1,\ell^\gamma)= \ell^\beta$ & Page 32\\
$\kappa$ & $v_\ell(m)$ & Page 32\\
$\nu$ & $v_\ell(\Delta)$ & Page 32\\
$W(d)$ & A sum involving $\sum^{*m}$, see Section 5.2 & Page 32\\ 
$C(t,q,d)$ & Used in the proof of Theorem 9, see Section 5.2 & Page 32\\ 
$c_{\kappa}(t,q,d)$ & Used in the proof of Theorem 9, see Section 5.5 & Page 34 \\ 
$V(\tau,d)$ & Used in the proof of Theorem 9, see Section 5.3 & Page 33\\
$S(a,n)$ & Number of solutions of $x^2-a\equiv 0 \pmod{n}$ & Page 36\\
$C_{K,N,M}(t,q,d)$ & Used in the proof of Theorem 9, see Section 5.6 & Page 51\\
\hline
\end{tabular}
\end{center}

\section{Applications}\label{applications}
 
 \subsection{Points on Elliptic Curves over Finite Field Extensions}
Our first application emphasizes expectations of functions of $t_E$.  We give a generalization of the work of Brock and Granville \cite{BrockGranville} on ``quadratic excess'', i.e. that there are extra points in quadratic extensions on curves over finite fields.  For a geometric explanation of this phenomenon, see the paper of Katz \cite{Katz}.

Let $E$ be an elliptic curve defined over a finite field $\F_q$, and $\F_{q^r}$ be the degree $r$ extension field.  Then the number of $\F_{q^r}$-points on $E$ is given by $E(\F_{q^r}) = q^r+1 - (\alpha^r+ \overline{\alpha}^r)$, where $\alpha, \overline{\alpha}$ are the two roots of $X^2-t_E X + q$.  In particular, $\alpha \overline{\alpha}  = q$.  We have by \eqref{cheb} that 
\[
(\alpha^r + \overline{\alpha}^r) = (\alpha^r + \overline{\alpha}^r) \frac{(\alpha-\overline{\alpha})}{(\alpha-\overline{\alpha})} = \begin{cases} U_{r}(t,q) -q U_{r-2}(t,q) & \text{ if } r \geq 2 \\ U_1(t,q)& \text{ if } r=1.\end{cases}
\]
We use this to compute the average number of points on elliptic curves over extension fields.  If $r \geq 2$ we have 
\[
\E_q(\# E(\F_{q^r}) \Phi_A) = (q^r+1)\E_q(\Phi_A) - \E_q(U_r(t_E,q)\Phi_A)+ q\E_q(U_{r-2}(t_E,q)\Phi_A),
\] 
and we will see in below 
Section \ref{apps1} that 
\[
\E_q(U_{k-2}(t_E,q)\Phi_A) = 
\begin{cases} 
\E_q(\Phi_A)+ O_{A,\eps}(q^{-1/2+ \eps}) & \text{ if } k=2 \\ O_{A,\eps}\left(q^{\frac{k-3}{2}+\eps}\right) & \text{ if } k\geq 3.  
\end{cases}
\] 
Therefore we have 
\[
\E_q(\# E(\F_{q^r}) \Phi_A) = 
\begin{cases} 
(q^2+q)\E_q(\Phi_A) + O_{A,\eps}(q^{1/2+\eps}) & \text{ if } r=2 \\ q^r \E_q(\Phi_A)  + O_{A,\eps}\left(q^{\frac{r-1}{2}+ \eps}\right) & \text{ if } r\geq 3,
\end{cases}
\]
recovering the Brock-Granville quadratic excess for the family of elliptic curves with {$A \hookrightarrow E(\F_q)$}.
 \subsection{Families of Curves over $\F_q$ and the Sato-Tate Distribution}\label{apps1}
Our second application is the Sato-Tate equidistribution of traces of the Frobenius endomorphism for several families of elliptic curves over $\F_q$.  Let $\mathcal{F} \subseteq \cC$ be a subset of elliptic curves over $\F_q$ and let $\Phi(\cF)$ be the indicator function of $\cF$, e.g., 
\[
\Phi(\cC(A)) = \Phi_A.
\]
Similarly, we often drop the $\cC$ from the notation in the families that we consider below.  We study the equidistribution of $t_E$ for $E\in \cF$ via the following lemma, which is immediate from the definition found in  
\cite[\S1]{SerreVertical}.
\begin{lemma}\label{equilem}
The traces of the Frobenius $t_E$ for $E\in \cF$ are  
equidistributed with respect to the Sato-Tate measure 
if for all $j \geq 1$ we have 
\[ 
\lim_{q\to\infty} \frac{\E_q\left(U_{j}(t_E,q) \Phi(\cF) \right)}{q^{j/2}\E_q( \Phi(\cF))} = 0.
\]
\end{lemma}
For $n_1,n_2$ two natural numbers with $n_2 \mid n_1$ and $q \equiv 1 \pmod*{n_2}$ let 
\[
v(n_1,n_2) = \frac{n_1}{\psi(n_1) \varphi(n_1)n_2^2} \prod_{\ell \mid \frac{n_1}{(q-1,n_1)}} \left(1+\ell^{-1-2v_\ell\left(\frac{(q-1,n_1)}{n_2}\right)} \right).
\]
We have the trivial estimate $\E_q(U_{k-2}(t_E,q) \Phi_A)\ll  q^{k/2-1}$.  The following is the main corollary of Theorem \ref{MT}.
 \begin{corollary}\label{cor}
Let $A$ be a finite abelian group of rank at most $2$ with $(q,|A|)=1$ and $k\geq 2$ an integer. When $q\equiv 1 \pmod{n_2(A)}$ we have
\est{\E_q(U_{k-2}(t_E,q) \Phi_A) =  v\left(n_1(A),n_2(A)\right)  \left( \delta(k,2) + O_{\eps}(kn_2(A) n_1(A)^{2+\eps}q^{\frac{k-3}{2}+\eps})\right).
}
In particular, the traces of the Frobenius $t_E$ for $E\in \cC(A)$ become 
equidistributed with respect to the Sato-Tate measure 
as $q\to \infty$ through prime powers $q\equiv 1\pmod*{n_2(A)}$.  The equidistribution is uniform in $A$ as soon as $q\gg n_2(A)^2 n_1(A)^{4+\delta}$ for some $\delta>0$.  
 \end{corollary}

Corollary \ref{cor} follows from Theorem \ref{MT} and Deligne's bound on the Hecke eigenvalues of modular forms.  For a similar calculation, see the proof of Theorem \ref{SigMRQ}. One could obtain an asymptotic estimate from Theorem \ref{MT} without Deligne's work by using a bound of Selberg \cite{SelbergSCS} on Hecke eigenvalues as Birch does in~\cite{Birch}, but one would get a weaker error term. 

By inclusion-exclusion arguments we can also show the distribution of $t_E$ tends to the Sato-Tate distribution for several other families of elliptic curves.  We give some examples here.  

\begin{enumerate}
\item Howe computes the probability that $N \mid \#E(\F_q)$ as $q\rightarrow \infty$ \cite[Theorem 1.1]{Howe}.  Let $\cC(N \mid \#E(\F_q))\subset \cC$ denote the set of curves for which $N \mid \#E(\F_q)$, and let $\Phi(N \mid \#E(\F_q))$ be its indicator function. For a prime $\ell$, let $A_\ell(a, b)$ denote the group $\Z/\ell^a\Z \times \Z/\ell^b\Z$.  We see that
\es{\label{PhiN}
\Phi(N \mid \#E(\F_q)) = \prod_{\ell \mid N} \left(\Phi_{A_\ell(v_\ell(N),0)} + \sum_{k=1}^{\big\lfloor \frac{v_\ell(N)}{2}\big\rfloor}\bigg( \Phi_{A_{\ell}(v_\ell(N)-k+1,k)} -\Phi_{A_{\ell}(v_\ell(N)-k,k)}\bigg)\right).
}
Expanding $U_{k-2}(t_E,q)\Phi(N \mid \#E(\F_q)) $ using \eqref{PhiN} and applying Corollary \ref{cor}  to each term shows that for $q\rightarrow \infty$ sufficiently fast with respect to $N$, the $t_E$ for $E \in \cC(N \mid \#E(\F_q))$ 
become equidistributed with respect to the Sato-Tate measure.

\item Let $\cC(\ell-\text{part} (\alpha, \beta)) \subset \cC$ be the set of curves such that the $\ell$-primary part of $E(\F_q)$ is isomorphic to $\Z/\ell^\alpha \Z \times \Z/\ell^\beta \Z$.  Gekeler computes $\P_q(\cC(\ell-\text{part} (\alpha, \beta)))$ for $q\to \infty$ through primes \cite[Formula (3.9)]{Gekeler}.  Let $\Phi(\ell-\text{part} (\alpha, \beta))$ be the indicator function of $\cC(\ell-\text{part} (\alpha, \beta))$.  
 Then 
\begin{multline}\label{gek1}
\Phi(\ell-\text{part}(\alpha, \beta)) = 
\Phi(A_\ell(\alpha,\beta))  - \Phi(A_\ell(\alpha+1,\beta))- \Phi(A_\ell(\alpha,\beta +1)) \\ + \Phi(A_\ell(\alpha+1,\beta+1))
\end{multline}
when $\beta < \alpha$, and 
\es{\label{gek2}
\Phi(\ell-\text{part}(\alpha, \alpha)) = 
\Phi(A_\ell(\alpha,\alpha))  - \Phi(A_\ell(\alpha+1,\alpha))
}
when $\alpha = \beta$.  Applying Corollary \ref{cor} to \eqref{gek1} and \eqref{gek2}, we recover formulas of Gekeler for all finite fields $\F_q$ \cite{Gekeler}.  Expanding $U_{k-2}(t_E,q)\Phi(\ell-\text{part}(\alpha, \beta))$ and applying Corollary \ref{cor} to each term shows that the distribution of $t_E$ over $E \in \cC(\ell-\text{part} (\alpha, \beta))$ becomes equidistributed with respect to the Sato-Tate measure as 
$q\rightarrow \infty$ through prime powers $q\equiv 1\pmod*{\ell^\beta}$.  Applying Theorem \ref{MT} to \eqref{gek1} and \eqref{gek2} gives explicit formulas in terms of traces of Hecke operators for these counts.

\item Let $\Phi(n_2=m)$ be the indicator function of the family of elliptic curves $\cC(n_2 = m) = \{E/\F_q: n_2(E(\F_q)) = m \}$. In particular, $\cC(n_2=1)$ is the set of isomorphism classes of curves with cyclic group structure over $\F_q$.  In Theorem \ref{SigMRQ} we give asymptotic formulas for $\E_q(U_{k-2}(t_E,q)\Phi(n_2=m))$, which show by Lemma \ref{equilem} that the $t_E$ for curves $E \in \cC(n_2=m)$ 
become equidistributed with respect to the Sato-Tate measure 
as soon as $q \gg m^{6+\delta}$ for any $\delta>0$.
\end{enumerate}

\subsection{Averages for Invariant Factors of $E(\F_q)$}\label{apps2}

\begin{theorem}\label{MTn1} 
Let $c(q)$ be defined by 
\est{c(q) \eqdef \prod_{\ell^\alpha || q-1} \left(1- \frac{1}{\ell^2(\ell+1) } \frac{1-\ell^{-4\alpha}}{1-\ell^{-4}}\right).} 
We have that 
\[
\E_q(n_1(E)) = c(q) q + O_\eps(q^{1/2+\eps}).
\]
Note that for all $q$
\est{ .8758<  \prod_{\ell} \left(1 - \frac{\ell^2}{(\ell^4-1)(\ell+1)}\right) < c(q) \leq 1.
} 
\end{theorem}
\noindent Asymptotic formulas for the higher moments $\E_q(n_1(E)^R)$ are also accessible by our methods.  
\begin{theorem}\label{MTn2}
Let $b(q)$ be defined by 
\est{ 
b(q) \eqdef \prod_{\ell^\alpha || q-1} \left(1+ \frac{1}{\ell(\ell+1)} \frac{1-\ell^{-2\alpha}}{1-\ell^{-2}}\right).
} 
We have that 
\[
\E_q(n_2(E)) = b(q)  + \frac{pq^{-1/2}}{12} \mathds{1}_{q=\square}+ O_\eps(q^{-1/3+\eps}),
\] 
where $\mathds{1}_{q=\square} = 1$ if $q$ is a square and is $0$ otherwise.
Note that for all $q$ 
\est{1 
\leq b(q) < \prod_{\ell} \left(1+ \frac{\ell}{(\ell^2-1)(\ell+1)}\right) < 1.45004.
} 
\end{theorem}

Theorems \ref{MTn1} and \ref{MTn2} are both consequences of the following result.  
\begin{theorem}\label{SigMRQ} 
If $m \nmid q-1$ then $\E_q\left(U_{k-2}(t_E,q)\Phi(n_2 = m) \right)=0$.  If $m \mid q-1$ then 
\begin{multline*}
 \E_q\left(U_{k-2}(t_E,q)\Phi(n_2 = m) \right) = \frac{(q+1) \delta(k,2)}{q\psi(m^2)\varphi(m)} \prod_{\substack{ 
\ell \mid \frac{q-1}{m} \\  
\ell \nmid m}} \left(1 - \frac{1}{\ell(\ell^2-1)}\right) \prod_{\substack{
\ell \mid \frac{q-1}{m} \\  
\ell \mid m}} \left( 1- \frac{1}{\ell^3}\right) \\ + O( kq^{\frac{k-3}{2}} \tau(q-1)\log q),
\end{multline*}
where $\tau(n)$ is the number of divisors of $n$, and $5$ is an admissible constant in the $O$ notation.  In particular, we have that the $t_E$ for $E\in \cC(n_2=m)$ become equidistributed with respect to the Sato-Tate measure as $q\to \infty$ through prime powers $q\equiv 1\pmod*{m}$.  The equidistribution is uniform in $m$ when $q\gg m^{6+\delta}$ for some $\delta>0$.
\end{theorem}

We prove Theorems \ref{MTn1} and \ref{MTn2} in 
Section \ref{MTn1MTn2proofs} and \ref{SigMRQ} in 
Section \ref{SigMRQproof}.
Theorem \ref{SigMRQ} with $k=2$ and $m=1$ gives the number of $\F_q$-isomorphism classes of elliptic curves over $\F_q$ with cyclic group structure, recovering a result of Vl\v adu\c t \cite{Vladut}.
Cojocaru has shared with us a preprint \cite{Cojocaru1} in which an asymptotic formula for $\E_q\left(\Phi(n_2 = m)\right)$ is derived from the results of Howe \cite{Howe}. 
 
The method of proof of Theorem \ref{MTn2} also allows one to compute averages of several types of arithmetic functions of $n_2(E)$ over $\cC$, such as the divisor function, $\E_q\left(\tau(n_2(E))\right)$, or a Dirichlet character, $\E_q\left(\chi(n_2(E))\right)$.  Such applications were suggested to us by Cojocaru at the Arizona Winter School 2016, see the related work of her project group \cite{AWSCojocaru2016}. On the other hand, computing averages of arithmetic functions of $n_1(E)$ over $\cC$ seems more challenging.

\subsection{Other applications}\label{apps3}
We briefly sketch some applications of our work that we do not pursue further in this paper.

Gekeler \cite{Gekeler} studies elliptic curves not over one prime field $\F_p$, but rather takes averages over all fields $\F_p$ with $p \le X$, where $X>1$. In this situation he applies 
results of Howe for individual primes \cite{Howe}, to prove that
\es{\label{gek3}
\frac{\left|\{E/\F_p \mid p \leq X,\ E(\F_p)[\ell^\infty] \cong \Z/\ell^{\alpha}\Z \times \Z/\ell^\beta\Z\}\right|}{\left|\{E/\F_p \mid p \leq X\}\right|} 
= g^{(\ell)}(\alpha,\beta) +O_{\alpha,\beta,\ell}(X^{-1/2}),
} 
for an explicit constant $g^{(\ell)}(\alpha,\beta)$ given in \cite[Equation (2.3)]{Gekeler}.  Applying our Theorem \ref{MT} in place of the estimate for $w(m,n)$ found on page 245 of \cite{Howe} gives an explicit expression for the error in term in \eqref{gek3} in terms of eigenvalues of Hecke operators.  

Following the same steps as Section 3 of \cite{Gekeler} we estimate \eqref{gek3} but now in Gekeler's step (3.11) we may exploit cancellation among the eigenvalues of $T_p$ as $p$ varies over $p \leq X$ to give a better error term (using e.g.~Theorem 5.40 (or assuming GRH, Theorem 5.15) of \cite{IK}).  We may thus improve the error in Gekeler's result to $O_{\ell,\alpha,\beta}(X^{-1/2}\exp(-C\sqrt{\log X})),$ or under GRH to $O_{\ell, \alpha,\beta}(X^{-1}(\log X)^2).$  Interestingly, a similar calculation shows that 
\es{\label{gek4}
\frac{\left|\{E/\F_{p^2} \mid p^2 \leq X, E(\F_{p^2})[\ell^\infty] \cong \Z/\ell^{\alpha}\Z \times \Z/\ell^\beta\Z\}\right|}{|\{E/\F_{p^2} \mid p^2 \leq X\}|}} 
has a lower-order main term of size asymptotic to $c X^{-1/2}$ for some $c=c_{\ell,\alpha,\beta}$ depending on $\ell, \alpha, \beta$ which comes from the hyperbolic and supersingular terms of Theorem \ref{MT}.  

Next we mention briefly two standard applications of the Eichler-Selberg trace formula for $S_k(\Gamma_0(N),\chi)$ that should generalize to $S_k(\Gamma(N,M))$ using Theorem \ref{ESTF}.  

One can give simple formulas for the dimension of the space of cusp forms $\dim S_k( \Gamma_0(N),\chi)$ by studying the trace formula when $q=1$: the Hecke operator $T_1$ is just the identity on a space of cusp forms, and so   
$\tr(T_1|S_k( \Gamma_0(N),\chi)) =\dim S_k( \Gamma_0(N),\chi).$ 
See for example \cite[Cor 8]{Ross}, where Ross carefully derives this formula.  Theorem \ref{ESTF} can be used to give a similar formula for $\dim  S_k(\Gamma(N,M))$.  The dimension of the space of cusp forms for $\Gamma(N,M)$ may also be computed via the Riemann-Roch theorem, an approach worked out in detail by Quer \cite{Quer}.

Another interesting application of the Eichler-Selberg trace formula is the ``vertical'' equidistribution of eigenvalues of Hecke operators acting on $S_k(\Gamma_0(N),\chi)$.  It was proved independently around the same time by Conrey, Duke and Farmer \cite{ConreyDukeFarmer} and Serre \cite{SerreVertical} that for $p$ a fixed prime, as $k,N\rightarrow \infty$ through even weights $k$ and levels $N$ such that $p\nmid N$ that the eigenvalues of $T_p$ become equidistributed with respect to the measure 
\est{ d\mu_p  = \frac{p+1}{\pi} \frac{(1-x^2/4)^{1/2}}{(p^{1/2}+p^{-1/2})^2-x^2} \,dx} 
on $[-2,2]$ (which is not the Sato-Tate measure).  
Our Theorem \ref{ESTF} should yield a similar equidistribution result for Hecke eigenvalues of $T_p$ acting on the spaces $S_k(\Gamma(N,M))$.

We mention one more amusing application of the explicit formulas we prove, in particular an application of the simple case of Example \ref{eg1}.  Let $\ell,p$ be two primes with $\ell >(\sqrt{p}+1)^2$, which implies that $\ell \neq p$ and $p \not\equiv \pm1 \pmod{\ell}$.  For $\ell$ and $p$ in this range, the Hasse bound implies that an elliptic curve over $\F_p$ cannot have a point of order $\ell$ and so
$\cC(\Z/\ell\Z)$ is empty.
Then for any two primes $\ell,p$ satisfying $\ell>(\sqrt{p}+1)^2$, the result of Example \ref{eg1} implies that
\est{
\tr(T_p|S_2( \Gamma_1(\ell)))  = p+1-\frac{\ell-1}{2}.
} 
Another interesting choice is to take $\ell$ and $p$ to be primes such that $(\sqrt{p}-1)^2 < \ell < (\sqrt{p}+1)^2$ and $p-\ell \neq -1,0,1.$  In this case we get a formula for a Hurwitz-Kronecker class number (for the definition of $H(\Delta)$ see \eqref{HKcn}): 
\est{
H\left((p+1-\ell)^2-4p\right) = 
-\frac{\tr(T_p|S_2( \Gamma_1(\ell)))}{\ell-1} + \frac{p+1}{\ell-1} -\frac{1}{2}.
}

These types of examples show that one should only expect equidistribution with respect to the Sato-Tate measure when $p \to \infty$ much faster than the conductor.  It would be interesting to study with what uniformity one can expect Sato-Tate equidistribution when the conductor and $p$ tend to infinity simultaneously.

\subsection{Proofs of Theorems \ref{MTn1} and \ref{MTn2}}\label{MTn1MTn2proofs}
 
 We begin by using Theorem \ref{SigMRQ} to prove Theorems \ref{MTn1} and \ref{MTn2}. Throughout this section we suppose that $\F_q$ is a finite field of characteristic $p$, that $q = p^u$ for some integer $u \ge 1$.
 
 \begin{proof}[Proof of Theorem \ref{MTn1}]
We split up the sum defining $\E_q(n_1(E))$ based on the value of $n_2(E)$:
\es{\label{Tformula}
\E_q(n_1(E))
= 
\sum_{1 \leq m \leq \sqrt{q}+1}
\left(\frac{q+1}{m}  \E_q(\Phi(n_2 = m)) - \frac{1}{m} \E_q(\Phi(n_2 = m) t_E)\right).
}
If $m\mid q-1$ then applying Theorem \ref{SigMRQ} with $k=3$ we have $\E_q(\Phi(n_2 = m) t_E) \ll_\eps q^{\eps}$. Applying Theorem \ref{SigMRQ} with $k=2$  
gives 
\es{\label{antiexp1} 
\E_q(\Phi(n_2 = m))
=  \frac{1}{\psi(m^2)\varphi(m)} \prod_{\substack{v \text{ prime} \\ v \mid \frac{q-1}{m} \\ v \nmid m}} \left(1 - \frac{1}{v(v^2-1)}\right) \prod_{\substack{v \text{ prime} \\ v \mid \frac{q-1}{m} \\ v \mid m}} \left( 1- \frac{1}{v^3}\right) + O_\eps\left( q^{-\frac{1}{2}+\eps}\right).
}  
The error term here is uniform in $m$. If $m\nmid q-1$ then $\cC(n_2 = m)$ is empty and $\E_q(\Phi(n_2 = m)) = \E_q(\Phi(n_2 = m) t_E) = 0$.

We substitute \eqref{antiexp1} into \eqref{Tformula} and also see that 
\est{ 
\sum_{\substack{1 \leq m \leq \sqrt{q}+1\\ m \mid q-1 }}
\frac{1}{m} \E_q(\Phi(n_2 = m) t_E)
\ll_\eps q^{\eps},
}
so it suffices to estimate 
\es{\label{mt1}(q+1)
\sum_{\substack{1 \leq m \leq \sqrt{q}+1\\ m \mid q-1 }} 
\frac{\E_q(\Phi(n_2 = m))}{m}.
}  
The error term from \eqref{antiexp1} makes a contribution of size $O_\eps(q^{1/2+\eps})$ to \eqref{mt1}.  The main term from \eqref{antiexp1} gives a main term of $(q+1)$ times 
\est{
c(q) \eqdef \sum_{ m \mid q-1}\frac{1}{m\psi(m^2)\varphi(m)} \prod_{\substack{v \text{ prime} \\ v \mid \frac{q-1}{m} \\ v \nmid m}} \left(1 - \frac{1}{v(v^2-1)}\right) \prod_{\substack{v \text{ prime} \\ v \mid \frac{q-1}{m} \\ v \mid m}} \left( 1- \frac{1}{v^3}\right), } 
where we have extended the sum over $m$ to all $m\mid q-1$ at a cost of a small error term.   
Note that $c(q)$ is a multiplicative function of $q-1$, so let $c(q)=f(q-1)$ where $f$ is multiplicative.  For a prime power $\ell^\alpha$ we calculate 
\est{f(\ell^\alpha) 
= 1- \frac{1}{\ell^2(\ell+1) } \left( 1+ \frac{1}{\ell^4-1} - \frac{1}{\ell^{4\alpha-4} (\ell^4-1)}\right). }
\end{proof}

\begin{proof}[Proof of Theorem \ref{MTn2}]  
We have  
\begin{equation}\label{MTn2line1}
\E_q(n_2(E)) = 
\sum_{\substack{1\leq m \leq \sqrt{q}+1 \\ m \mid q-1}} m \E_q(\Phi(n_2 = m)).
\end{equation}
The calculation of $\E_{q}(\Phi(n_2 = m))$ in \eqref{antiexp1} gives a useful bound when $m \ll q^{1/6 + \eps}$.  When $m$ is larger, we need a trivial estimate.  This estimate will involve bounds for Hurwitz-Kronecker class numbers $H(\Delta)$, which will be defined in Section \ref{curves2classnumbers}.

For $m\geq 3$, we have by Lemma \ref{S46withweights} in Section \ref{curves2classnumbers} below and Lemma 4.8 of \cite{Schoof} that 
\est{
\left| 
q \E_q(\Phi(n_2 = m)) -\frac{p-1}{24} \delta(m, \sqrt{q}-1)
 - \frac{p-1}{24} \delta(m, \sqrt{q}+1) \right| & 
\leq \sum_{\substack{t^2<4q \\ 
t \equiv q+1 \pmod*{m^2}}} 
H\left(\frac{t^2-4q}{m^2}\right),
}
where the terms $\delta(m,\sqrt{q} \pm 1)$ are $0$ if $q$ is not a square, and the Hurwitz-Kronecker class number $H(\Delta)$ is defined in \eqref{HKcn}.  To bound this expression we note that 
\est{H(\Delta) \ll |\Delta|^{1/2+\eps}} 
by the class number formula and the upper bound $L(1,\chi) \ll \log |\Delta|$ for the associated $L$-functions. Then 
\es{\label{trivest} \sum_{\substack{t^2< 4q \\ t \equiv q+1 \pmod*{m^2}}} H\left(\frac{t^2-4q}{m^2}\right)  \ll_\eps \frac{q^{1/2+\eps}}{m}\left( 1+ \frac{q^{1/2}}{m^2}\right).}  
We split \eqref{MTn2line1} at $m=q^{1/6+\eps}$ and use \eqref{antiexp1} when $m \leq q^{1/6+\eps}$ and \eqref{trivest} when $m > q^{1/6+\eps}$:
\begin{multline*} 
\E_q(n_2(E))
= \sum_{m\mid q-1}   \frac{m}{\psi(m^2)\varphi(m)} \prod_{\substack{v \text{ prime} \\ v \mid \frac{q-1}{m} \\ v \nmid m}} \left(1 - \frac{1}{v(v^2-1)}\right) \prod_{\substack{v \text{ prime} \\ v \mid \frac{q-1}{m} \\  v \mid m}} \left( 1- \frac{1}{v^3}\right)  +\frac{p }{12\sqrt{q}}
\mathds{1}_{q=\square}\\ + O_\eps \left( q^{-1/3+\eps}\right),
\end{multline*}
where $\mathds{1}_{q = \square} = 1$ if $q$ is a square and is $0$ otherwise.
Setting 
\est{g(q-1) = \sum_{m\mid q-1}   \frac{m}{\psi(m^2)\varphi(m)} \prod_{\substack{v \text{ prime} \\ v \mid \frac{q-1}{m} 
\\ v \nmid m}} \left(1 - \frac{1}{v(v^2-1)}\right) \prod_{\substack{v \text{ prime} \\ v \mid \frac{q-1}{m}  \\ v \mid m}} \left( 1- \frac{1}{v^3}\right),} 
we see that $g(q-1) = b(q)$.  It suffices to compute $g$ on prime powers.  We calculate
\est{g(\ell^\alpha) =1+ \frac{1}{\ell(\ell+1)} \frac{1-\ell^{-2\alpha}}{1-\ell^{-2}} .}
\end{proof}

\subsection{Proof of Theorem \ref{SigMRQ}}\label{SigMRQproof}
Let $A(md,md)$ denote the group $\Z/md\Z \times\Z/md\Z$. Note that $\cC(A(md,md))$ is empty if $(md,q) >1$ since the set of isomorphism classes over $\overline{\F}_q,\ \{E/\overline{\F}_q:  A(md,md) \hookrightarrow E(\overline{\F}_q)\}$ is empty.  Also, $\cC(A(md,md))$ is empty if $md \nmid q-1$ by the first paragraph of the proof of Proposition \ref{alternateapproachprop}, or by a Weil pairing argument.  

In this proof we write $U(t,q)=U_{k-2}(t,q)$.  We have that 
\es{\label{thm6pf1}
\E_q(U(t_E,q)\Phi(n_2=m)) =  \sum_{d\geq 1} \mu(d) \E_q\left(U(t_E,q)\Phi_{A(md,md)}\right).
} 
We assume that $m \mid q-1$ for the remainder of this argument, since otherwise $\E_q(U(t_E,q)\Phi(n_2=m)) =0$. When $m \mid q-1$, the sum on the right hand side of \eqref{thm6pf1} is finite.

Applying Theorem \ref{MT} to the right hand side of \eqref{thm6pf1} shows that 
\begin{multline}\label{thm6pf2}
\E_q(U(t_E,q)\Phi(n_2=m)) = \frac{1}{q} \sum_{d \mid \frac{q-1}{m}}\mu(d) \left( T_{md,md}(q,1) - p^{k-1} T_{md,md}(q/p^2,p)\right) \\ + q^{k/2-1} \frac{(p-1)(k-1)}{24q} \sum_{d \mid \frac{q-1}{m}} \mu(d) \left( \delta_{md}(q^{1/2},1) + (-1)^k  \delta_{md}(q^{1/2},-1) \right).
\end{multline}
By M\"obius inversion, the second line of \eqref{thm6pf2} is equal to 
\es{\label{thm6ss}  q^{k/2-1} \frac{(p-1)(k-1)}{24q} \left( \delta( q,(m+1)^2) + (-1)^k \delta(q,(m-1)^2)\right)  \leq   \frac{(k-1)}{12}q^{\frac{k-3}{2}}.}

From the first line of \eqref{thm6pf2} and the definitions in the introduction we have that \est{ \frac{1}{q} \sum_{d \mid \frac{q-1}{m}}\mu(d) T_{md,md}(q,1) = D-T+I - H,} where \est{D= \frac{1}{q} \sum_{d \mid \frac{q-1}{m}}\mu(d)  \frac{\sigma(q) \delta(k,2)}{\psi((md)^2) \varphi(md)},}
\est{ T =  \frac{1}{q} \sum_{d \mid \frac{q-1}{m}}\mu(d)  \frac{\tr(T_q | S_k(\Gamma(md))) }{\psi((md)^2)\varphi(md)},}
\est{ I = \frac{1}{q} \sum_{d \mid \frac{q-1}{m}}\mu(d) \frac{(k-1)}{24} q^{k/2-1} \left( \delta_{md}(q^{1/2},1) + (-1)^k  \delta_{md}(q^{1/2},-1) \right) ,}
and \begin{multline*}H = \frac{1}{4q} \sum_{d \mid \frac{q-1}{m}}  \frac{\mu(d)}{\psi((md)^2) }\sum_{i=0}^v \min(p^i,p^{v-i})^{k-1}   \primesum_{\substack{\tau \mid (md)^2 \\ 
 g \mid p^i - p^{u-i}}} \varphi(g)  \left( 
 \delta_{md}(y_i,1)
 + (-1)^k 
 \delta_{md}(y_i,-1)
 \right),
 \end{multline*}
where the $'$ on the sum means that $g = (\tau,(md)^2/\tau)$, and $y_i$ is the unique element of $(\Z/(m^2d^2/g)\Z)^\times$ such that $y_i \equiv p^i \pmod{\tau}$ and $y_i \equiv p^{u-i} \pmod{m^2d^2/(g \tau)}$. 

We estimate each of the terms $D,T,I$, and $H$.  A short calculation shows that 
\es{\label{thm6d}D = \frac{\sigma(q)}{q}  \frac{1}{\psi(m^2)\varphi(m)} \prod_{\substack{ \ell \mid \frac{q-1}{m} \\ \ell \nmid m}} \left(1 - \frac{1}{\ell(\ell^2-1)}\right) 
\prod_{\substack{ \ell \mid \frac{q-1}{m}  \\ 
\ell \mid m}} \left( 1- \frac{1}{\ell^3}\right)  \delta(k,2).
} 
We apply Deligne's bound on Hecke eigenvalues to $T$ to get
\[
|T| \leq 2 q^{\frac{k-3}{2}} \log q \sum_{d \mid \frac{q-1}{m}} \frac{\dim S_k(\Gamma(md))}{\psi((md)^2) \varphi(md)}.
\] 
We have (e.g.~\cite[\S3.9]{DiamondShurman}) that 
\es{ \dim S_k(\Gamma(N)) \leq \frac{kN^3}{12}.} 
Applying this bound we find 
\es{\label{thm6t} |T| \leq \frac{\zeta(2)}{6}k q^{\frac{k-3}{2}}\tau(q-1) ( \log q ) .}
M\"obius inversion shows that \es{\label{thm6i} |I |= \frac{q^\frac{k-4}{2}(k-1)}{24} \left| \delta( q,(m+1)^2) + (-1)^k \delta(q,(m-1)^2)\right|  \leq \frac{(k-1)}{12}q^{\frac{k-4}{2}}.}
Lastly, we have \es{\label{thm6h} |H| \leq \frac{1}{q} \sum_{i=0}^v \min(p^i,p^{v-i})^{k-1} \sum_{d \mid \frac{q-1}{m}}  \frac{1}{\psi((md)^2)} \sum_{\tau \mid md} \varphi(\tau) \leq 2 q^{\frac{k-3}{2}} \tau(q-1)\log q.}

Very similar estimates hold for the term $$ \frac{p^{k-1}}{q} \sum_{d \mid \frac{q-1}{m}}\mu(d)  T_{md,md}(q/p^2,p)$$ of \eqref{thm6pf2} so we omit the calculation.  Drawing together \eqref{thm6pf2}, \eqref{thm6ss}, \eqref{thm6d}, \eqref{thm6t}, \eqref{thm6i}, \eqref{thm6h} we conclude the expression in the statement of Theorem \ref{SigMRQ}.

\section{Counting Curves Containing a Prescribed Subgroup}\label{curves2classnumbers}
Recall the definitions of the probability measure $\P_q$ and of the sets $\cC,\ \cC(A)$, and $\cC(A,t)$, from the introduction.
In Section \ref{sketchproof} we explained that our main goal is to give a formula for 
\es{\label{firststep}
\E_q\left( U_{k-2}(t_E,q) \Phi_A\right) = \frac{1}{q} \sum_{\substack{E/\F_q \\ A \hookrightarrow E(\F_q)}} \frac{U_{k-2}(t_E,q)}{\# \Aut_{\F_q}(E)} = \sum_{t^2\leq 4q} U_{k-2}(t,q) \P_q(\cC(A,t)),} 
where $U_{k-2}(t,q)$ are the normalized Chebyshev polynomials defined in \eqref{chebdef}.

In this section we give formulas for $\P_q(\cC(A,t))$ in terms of class numbers of orders in imaginary quadratic fields.  In the special case that $A \cong \Z/n\Z \times \Z/n\Z$ such results are due to Deuring \cite{Deu}, Lenstra \cite{Lenstra}, Schoof \cite{Schoof}, and Waterhouse \cite{Waterhouse}.  The following is a weighted version of Theorem 4.6 of \cite{Schoof}.  We begin with some definitions.

For $d<0$ with $d \equiv 0,1\pmod{4}$, let $h(d)$ denote the class number of the unique quadratic order of discriminant $d$.  Let 
\es{h_w(d)\eqdef 
\begin{cases} h(d)/3, &\text{ if } d = -3, \\ 
h(d)/2, & \text{ if } d=-4, \\ 
h(d) & \text{ if } d < 0,\ d \equiv 0,1 \pmod*{4}, \text{ and } d\neq -3,-4 \\
0 & \text{otherwise} \end{cases}
} 
and for $\Delta \equiv 0,1\pmod 4$ let 
\es{
\label{HKcn}H(\Delta) \eqdef \sum_{d^2 \mid \Delta} h_w\left(\frac{\Delta}{d^2}\right)
} 
be the Hurwitz-Kronecker class number. For $a\in \Z$ and $n$ a positive integer, the Kronecker symbol $\Big(\frac{a}{n}\Big)$ is defined to be the completely multiplicative function in $n$ such that if $p$ is an odd prime $\legen{a}{p}$ is the quadratic residue symbol and if $p=2$ 
\es{\left(\frac{a}{2}\right) \eqdef \begin{cases} 0 & \text{ if } 2 \mid a, \\ 1 & \text{ if } a \equiv \pm 1 \pmod 8, \\ -1 & \text{ if } a \equiv \pm5 \pmod 8 .\end{cases} }

\begin{lemma}\label{S46withweights}
Let $t\in \Z$.  Suppose $q = p^v$ where $p$ is prime and $v\geq 1$.  Then if $q$ is not a square
\begin{alignat*}{3}
\P_q(\cC(1,t))  = &  \frac{1}{2q}H(t^2-4q) \quad && \text{ if } t^2 < 4q \text{ and } p\nmid t,\\
 = & \frac{1}{2q}H(-4p)\quad  && \text{ if } t=0, \\
  = & \frac{1}{4q} \quad  && \text{ if } t^2=2q \text{ and } p =2, \\
 = & \frac{1}{6q} && \text{ if } t^2=3q \text{ and } p =3, 
\end{alignat*}
 and if $q$ is a square
\begin{alignat*}{3}
\P_q(\cC(1,t))  = &  \frac{1}{2q}H(t^2-4q)\quad  && \text{ if } t^2 < 4q \text{ and } p\nmid t,\\
 = & \frac{1}{4q}\left(1 - \legen{-4}{p}\right) \quad   && \text{ if } t=0, \\
 = & \frac{1}{6q}\left(1 - \legen{-3}{p}\right) \quad  && \text{ if } t^2 = q, \\
  = & \frac{p-1}{24q}\quad   && \text{ if } t^2 = 4q,
\end{alignat*}
 and $\P_q(\cC(1,t)) = 0$ in all other cases.
\end{lemma}
Next we state separately the case $n=2$.  Lemma \ref{Schoof1} is essentially Lemma 4.8 of \cite{Schoof}.  We write $A_{2,2} = \Z/2\Z \times \Z/2\Z$.
\begin{lemma}\label{Schoof1}
Let $q = p^v$ where $p \neq 2$ is prime. Suppose that $t\in \Z$ satisfies $t^2 \le 4q$.  If $q\equiv 3 \pmod*4$ then
\begin{alignat*}{3}
\P_q(\cC(A_{2,2},t)) = & \frac{1}{2q}H\left(\frac{t^2 - 4q}{4}\right) \quad && \text{ if } p\nmid t \text{ and } t\equiv q+1 \pmod*{4}, \\ 
= & \frac{h_w(-p)}{2q} \quad && \text{ if }t=0, \\ 
= & \P_q(\cC(1,t)) \quad && \text{ if } t^2=4q,
\end{alignat*}
and if $q \equiv 1\pmod*4$ 
\begin{alignat*}{3}
\P_q(\cC(A_{2,2},t)) = & \frac{1}{2q}H\left(\frac{t^2 - 4q}{4}\right) \quad && \text{ if } p\nmid t \text{ and } t\equiv q+1 \pmod*{4}, \\ 
= & \P_q(\cC(1,t)) \quad && \text{ if } t^2=4q,
\end{alignat*}
and $\P_q(\cC(1,t)) = 0$ in all other cases.
\end{lemma}
We also highlight the case where $A= A_{n,n} = \Z/n\Z \times \Z/n\Z$ with $n>2$. This is a weighted version of Theorem 4.9 of \cite{Schoof}. 
\begin{lemma}\label{Schoof2}
Let $q = p^v$ where $p$ is prime and $n>2$ be a positive integer with $p\nmid n$. Suppose that $t\in \Z$ satisfies $t^2 \le 4q$. If $q$ is not a square then 
\begin{alignat*}{3}
\P_q(\cC(A_{n,n},t)) = &\frac{1}{2q}H\left(\frac{t^2 - 4q}{n^2}\right) \delta_{n}(q,1) \quad && \text{ if } p\nmid t, \text{ and }t\equiv q+1 \pmod{n^2},\\
=& 0 \quad && \text{ otherwise,}\end{alignat*} 
and if $q$ is a square then 
\begin{alignat*}{3}
\P_q(\cC(A_{n,n},t)) & = &\frac{1}{2q}H\left(\frac{t^2 - 4q}{n^2}\right) \delta_{n}(q,1) \quad && \text{ if } p\nmid t, \text{ and }t\equiv q+1 \pmod{n^2},\\
& = & \P_q(\cC(1,2\sqrt{q}))\delta_{n}(\sqrt{q},1) \quad &&\text{ if } t=2\sqrt{q}, \\
& = & \P_q(\cC(1,2\sqrt{q}))\delta_{n}(\sqrt{q},-1) \quad &&\text{ if } t=-2\sqrt{q}, \\
 & = & 0 \quad && \text{ otherwise,}
\end{alignat*} 
and $\P_q(\cC(1,t)) = 0$ in all other cases.
\end{lemma}
We use an inclusion-exclusion argument to express $\cC(A,t)$ in terms of the sets $\cC(A_{n,n},t)$ and congruence conditions on $t$ and $q$.  This extends Lemmas \ref{S46withweights}, \ref{Schoof1}, and \ref{Schoof2} to a general finite abelian group $A$ of rank at most $2$.  See Theorems \ref{CurvesToCn} and \ref{SupersingularCount}.

We introduce a function $D(t;n)$ that plays a large role throughout the rest of the paper. Fix $n_1$ and $n_2$ with $n_2 \mid n_1$ and $d \in \left( \Z/n_1 \Z\right)^\times$.  If $n' \mid n_1$ and $d^2q \equiv 1 \pmod {n'}$ then by Lemma \ref{DtnWellDef} below $dq+d^{-1}$ is a well-defined residue class modulo $n_1n'$.  
For any $n \mid n_1n'$ we let 
\es{\label{Dtn}D(t;n) \eqdef \delta_n(dq+d^{-1},t) = \begin{cases} 1 & \text{if } n \mid dq+d^{-1}-t \\ 0 & \text{otherwise}.
\end{cases}} We must check that $D(t;n)$ is well-defined (\`a priori $dq+d^{-1}$ only makes sense modulo $n_1$), which we will do in Lemma \ref{DtnWellDef}.  This lemma will be used extensively throughout Section \ref{traces} with various choices of parameters.

For $m \leq n$ and $c \in \Z/\ell^{m}\Z$ we call the pre-image of $c$ in $\Z/\ell^{n}\Z$ under the canonical projection $\Z/\ell^{n}\Z \to \Z/\ell^{m}\Z$ the set of \emph{lifts} modulo $\ell^n$.   Choosing a particular $c_0$ modulo $\ell^n$ that is a lift of $c$, we can describe the set of all lifts of $c$ by $c_0 + j \ell^m \pmod{\ell^{n}}$ where $0 \le j < \ell^{n-m}$.

\begin{lemma}\label{DtnWellDef} 
Let $0\leq B \leq C$, $D \in (\Z/\ell^{C}\Z)^\times$, and $D^2q\equiv 1 \pmod {\ell^B}$.  For any $i$ satisfying $C \leq i \leq C+B$ we have that $Dq+D^{-1}\pmod {\ell^i}$ is the same residue class for any lift of $D$ to a residue class modulo $\ell^i$, and we say it is ``well-defined''.
\end{lemma}

\begin{proof}
We write the set of lifts of $D \pmod{\ell^C}$ to residue classes modulo $\ell^i$ as $D+D'\ell^{C}$ with $0\le D' < \ell^{i-C}$.  Then 
\est{(D+D'\ell^{C})q+(D+D'\ell^{C})^{-1} \equiv & Dq+D^{-1} +(q-D^{-2})D' \ell^C \\ \equiv & Dq+D^{-1} \pmod {\ell^i}} 
since $C\leq i \leq C+B\leq 2C$, completing the proof.
\end{proof}

When $d =1$, the condition that $D(t;n) = 0$ is the same as $n \nmid q+1-t$, in which case $\cC(\Z/n\Z, t)$ is empty.  Note that $D(t;n)$ is multiplicative for fixed $t$, i.e.~if $(n,m)=1$ then $D(t;n)D(t;m)=D(t;nm)$.  

We say that $\nu \in \N$ is a full divisor of $n_1$ and write $\nu || n_1$ if for all primes $\ell \mid \nu$ we have $v_\ell(\nu) = v_\ell(n_1)$.  For $\mu, \nu \in \N$ we write $\mu \prec \nu$ if:
\begin{enumerate} 
\item The integer $\mu$ is divisible by all the primes dividing $\nu$ and no others, i.e.~$\mu \mid \nu^\infty$ and $\nu \mid \mu^\infty$, and  
\item 
For all primes $\ell \mid \nu$ we have $v_\ell(\mu) \leq v_\ell\left( n_1/n_2\right) - 1$. 
\end{enumerate} 
Finally, we define the function $D_{\nu,\mu}(t)$ (which also depends on $q, d, n_1,n_2$) to be 
\es{
D_{\nu,\mu}(t) \eqdef \prod_{\ell \mid \nu} \left( D(t;\ell^{v_\ell(n_1n_2 \mu)-1})-D(t;\ell^{v_\ell(n_1n_2\mu)})\right).
} 
When $d=1,\ D_{\nu,\mu}(t) = 1$ if for each prime $\ell$ dividing $\mu,\ v_{\ell}(q+1-t) = v_{\ell}(n_1 n_2 \mu) - 1$.

Now we define the following class numbers:
\begin{multline}\label{cndef} H_{n_1,n_2}(t,q,d) \eqdef \frac{1}{2} H\left( \frac{t^2-4q}{n_2^2}\right) \delta_{n_2}(d^2q,1) D(t;n_1n_2)\\ + \sum_{\substack{m || n_1 \\ m\geq 2}} \sum_{\mu \prec m }\lambda(m) \frac{1}{2}H\left(\frac{t^2-4q}{(n_2\mu)^2} \right)\delta_{n_2\mu }(d^2q,1)D_{n_1 ,\mu}(t),
\end{multline} 
where $\lambda(m) = (-1)^{\omega(m)}$ and $\omega(m)$ denotes the number of distinct prime factors of $m$ (so $\lambda(m)$ is almost the Liouville function).

Note in particular that when $n_1 = \ell^e$, and $n_2= \ell^\delta$ we have  \begin{multline}\label{HNMdef} 
H_{\ell^e,\ell^\delta}(t,q,d)  =  \frac{1}{2} H\left(\frac{t^2-4q}{\ell^{2\delta}}\right)\delta_{\ell^{\delta}}(d^2q,1)D(t;\ell^{e+\delta})  \\ 
 - \sum_{k = 1}^{e-\delta-1} \frac{1}{2} H\left(\frac{t^2-4q}{\ell^{2(\delta+k)}}\right)\delta_{\ell^{\delta+k}}(d^2q,1)  \left(D(t;\ell^{e+\delta+k-1})  - D(t;\ell^{e+\delta+k})\right).
 \end{multline} 

We can divide the set of isomorphism classes of elliptic curve into those that are ordinary and those that are supersingular. 
We have 
\[
\P_q(\cC(A,t)) = \begin{cases} 
\P_q(\cC(A,t), E \text{ ordinary}) & \text{ if } p \nmid t \\  \P_q(\cC(A,t), E \text{ supersingular}) & \text{ if } p \mid t, 
\end{cases}
\]
and we deal with these two cases separately.  

We now state one of the main results of this section, a formula for the number of isomorphism classes $E/\F_q$ such that $E(\F_q)$ contains a subgroup isomorphic to $A$, has $t_E = t$, and is ordinary.  We return to the supersingular case at the end of this section.
\begin{theorem}\label{CurvesToCn} 
For a finite abelian group $A$ of rank at most $2$ we denote by $n_1(A)$ and $n_2(A)$ the first and second invariant factors of $A$, respectively.  We have that 
\[
\P_q(\cC(A,t),\,E \text{ ordinary}) = 
\begin{cases} \frac{1}{q} H_{n_1(A),n_2(A)}(t,q,1) & \text{ if } p \nmid t \text{ and }t^2<4q \\ 0 & \text{ otherwise}.
\end{cases}
\] 
\end{theorem}
Note that $\P_q(\cC(A,t))=0$ unless $q\equiv 1 \pmod*{n_2(A)}$ by the definitions of these class numbers.  
The crux of the proof of Theorem \ref{CurvesToCn} is the following proposition.
\begin{proposition}\label{CurveCount}
Let $A_{n_1,n_2} = \Z/n_1\Z \times \Z/n_2\Z$. We have 
\begin{equation}\label{CurveCountEq}
\P_q\left(\cC(A_{n_1,n_2},t)\right) = \P_q\left(\cC(A_{n_2,n_2},t)\right) D(t;n_1n_2) + \sum_{\substack{d || n_1 \\ d \geq 2}} \sum_{e \prec d} \lambda(d) \P_q\left(\cC(A_{n_2e,n_2e},t)\right) D_{n_1,e}(t),
\end{equation} 
where $\lambda(n) = (-1)^{\omega(n)}$ and $d=1$ in the definition of $D(t; n)$.  
\end{proposition}

\begin{proof}[Proof of Proposition \ref{CurveCount}]
We prove the proposition by induction on the number of prime factors $\ell$ of $n_1$ for which $v_\ell(n_1) \neq v_\ell(n_2)$.  

First consider the base case in which the number of such prime factors is $0$, so $n_2 = n_1$.  Then \eqref{CurveCountEq} holds since the sum over $e \prec d$ on the right hand side is empty, and the factor $D(t; n_1 n_2)$ is redundant since Lemmas \ref{Schoof1} and \ref{Schoof2} imply that $\cC(A_{n_1,n_2},t)$ is empty unless $n_1 n_2 \mid q+1-t$.

Let $A$ be a finite abelian group of rank at most two.  We suppose that \eqref{CurveCountEq} holds when $n_1$ has at most $\omega(|A|)$ prime factors $\mathfrak{p}$ for which $v_\mathfrak{p}(n_1) \neq v_\mathfrak{p}(n_2)$, and show it also holds for $n_1$ having at most $\omega(|A|)+1$ such prime factors.  Let $\ell$ be a prime such that $\ell \nmid |A|$. For the rest of this proof we write $A_{e,d}= \Z/\ell^e\Z \times \Z/\ell^d \Z$ where $0 \leq d \leq e$.  We are interested in the probability of the set of elliptic curves with $A \times A_{e,d} \hookrightarrow E(\F_q)$ and $t_E=t$. We note that $A \times A_{e,d} \hookrightarrow E(\F_q)$ if and only if $A \hookrightarrow E(\F_q)$ and $ A_{e,d} \hookrightarrow E(\F_q)$.

Let 
\est{
X \eqdef \cC(A\times A_{e,d}) = \{E/\F_q : A \times A_{e,d} \hookrightarrow E(\F_q)\}.
}
We define a set of isomorphism classes that contains $X$:
\est{
X_0 \eqdef \{E/\F_q :  A \times A_{d,d} \hookrightarrow E(\F_q) \text{ and } v_{\ell}(\#E(\F_q)) \ge d+e\}.
} 
Let $X_1$ denote the difference of these sets, i.e. $X = X_0 \setminus X_1$.

If $E \in X_0$ and $v_{\ell}(\#E(\F_q))  \ge 2e-1$ then $A_{e,d} \hookrightarrow E(\F_q)$, which implies $E \in \cC(A \times A_{e,d})$.  Therefore, each $E \in X_1$ satisfies $v_{\ell}(\#E(\F_q)) = d+e+k-1$ for some $k$ satisfying $1 \le k \le e-d-1$.  Let $X_1^k$ denote the subset of $X_1$ with $v_{\ell}(\#E(\F_q))  = d+e+k-1$, so that $X_1$ is a disjoint union of these sets.

If $E \in X_1^k$, then since $A_{e,d} \not\hookrightarrow E(\F_q)$ we have $A_{d+k, d+k} \hookrightarrow E(\F_q)$.  Conversely, if $E \in X_0$ satisfies $A_{d+k, d+k} \hookrightarrow E(\F_q)$ and $v_{\ell}(\#E(\F_q)) = d+e+k-1$, then $A_{e,d} \not\hookrightarrow E(\F_q)$ and $E \in X_1^k$.  We conclude that
\es{ \label{X1kEq}
X_1^k  = \{E/\F_q : A_{d+k,d+k} \hookrightarrow E(\F_q) \text{ and } v_\ell\left(\#E(\F_q)\right)=e+d+k-1\}.
}

Fixing the value of $t$ fixes $v_\ell(\#E(\F_q)) = v_\ell(q+1-t)$, so 
\eqref{X1kEq} 
implies that
\est{ \P_q(\cC(A \times A_{e,d},t)) & =  \P_q(X_0,\, t_E=t) - \P_q(X_1,\,t_E=t) \\ 
& =  \P_q(X_0, \,t_E=t) - \sum_{k=1}^{e-d-1} \P_q(X_1^k,\,t_E=t).
}
Since $v_{\ell}(\#E(\F_q)) \ge d+e$ if and only if $D(t;\ell^{d+e}) = 1$, we have
\est{ 
& \P_q(X_0, \,t_E=t) - \sum_{k=1}^{e-d-1} \P_q(X_1^k,\,t_E=t)\\
& =  \P_q(\cC(A\times A_{d,d},t)) D(t;\ell^{d+e})  - \sum_{\substack{1 \leq k \leq e-d-1 \\ v_\ell(q+1-t) = e+d+k-1}} \P_q(\cC(A\times A_{d+k,d+k},t)). 
} 

The induction hypothesis is now applicable to each probability on the right hand side of this expression, from which we deduce the proposotion for $n_1$ having $\omega(|A|)+1$ prime factors for which $v_\mathfrak{p}(n_1) \neq v_\mathfrak{p}(n_2)$.  
\end{proof}

\begin{proof}[Proof of Theorem \ref{CurvesToCn}]
We apply Lemmas \ref{S46withweights}, \ref{Schoof1}, and \ref{Schoof2} to \eqref{CurveCountEq}.  Rewriting this expression using definition of $H_{n_1,n_2}(t,q,d)$ completes the proof. 
\end{proof}

In order to prove a result analogous to Theorem \ref{CurvesToCn} for supersingular curves, we define a version of the class numbers $H_{n_1,n_2}(t,q,d)$ that collects all of the contributions from supersingular curves.  This involves an analysis of many special cases.  

We define a function $H^*_{n_1,n_2}(t,q,d)$ as follows.  If $n_2>2$ we define \es{\label{H*0}
H^*_{n_1,n_2}(t,q,d) \eqdef 0,} 
and if $n_2=2$ we define 
\es{\label{H*2}H^*_{n_1,2}(t,q,d) \eqdef \begin{cases} \frac{1}{2} h_w(-p)D(t;2n_1)& \text{ if } t=0 \\ 0 & \text{ otherwise.} 
\end{cases}} 
If $n_2 =1$ and $q=p^v$ with $v$ even we define 
\es{\label{H*1even} 
H^*_{n_1,1}(t,q,d) \eqdef 
\begin{cases} 
\frac{1}{4} \left(1-\legen{-4}{p}\right)D(t;n_1) & \text{ if } t=0, \\   
\frac{1}{6} \left(1-\legen{-3}{p}\right)  D(t;n_1) & \text{ if } t^2=q, \\ 
0 & \text{ otherwise.}
\end{cases}}  
If $n_2=1$ and $q=p^v$ with $v$ odd then we define 
\begin{multline}\label{H*1odd}
H^*_{n_1,1}(t,q,d) \\ \eqdef 
\begin{cases}  
\frac{1}{2} H(-4p) D(t;n_1)- \delta_{4}(n_1,0) \frac{1}{2}h_w(-p) \left( D(t;n_1) - D(t;2n_1)\right) & \text{ if } t=0, \\ 
\frac{1}{4} D(t;n_1)& \text{ if } t^2=2q \text{ and }p=2, \\
 \frac{1}{6} D(t;n_1) & \text{ if } t^2=3q \text{ and } p=3, \\ 
 0 & \text{ otherwise.} 
 \end{cases} 
 \end{multline} 
 
 \begin{theorem}\label{SupersingularCount}
 For a finite abelian group $A$ of rank at most $2$ we denote by $n_1(A)$ and $n_2(A)$ the first and second invariant factors of $A$, respectively.  We have that
\[
\P_q(\cC(A,t),\,E \text{ supersingular})= 
\begin{cases} \frac{1}{q} H^*_{n_1(A),n_2(A)}(t,q,1) & \text{ if } t^2<4q \\ \frac{p-1}{24q} \delta_{n_1(A)}(\sqrt{q},\pm1) & \text{ if } t^2= 4q.
\end{cases}
\]
\end{theorem}

\begin{proof}
We recall that $E/\F_q$ is supersingular if and only if $p \mid t_E$.  Lemma \ref{S46withweights} implies that such a curve must have $t_E^2 \in \{0,q,2q,3q,4q\}$, where $t_E^2 = 2q$ arises only when $p=2$ and $t_E^2 = 3q$ arises only when $p =3$.  Lemmas \ref{Schoof1} and \ref{Schoof2} imply that there are not many different group structures to consider.  We divide our argument into cases based on the value of $n_2$.

\vspace{12pt}
\noindent {\bf \emph{The case $t^2 < 4q$ and $n_2(A) > 2$.}}
\vspace{12 pt}

Lemma \ref{Schoof2} implies that for any $e > 2,\ \P_q\left(\cC(A_{e, e},t), E \text{ supersingular} \right) = 0$.  This implies that $\P_q(\cC(A,t),\,E \text{ supersingular}) = 0$, which matches the definition of $H^*_{n_1,n_2}(t,q,1)$ given by \eqref{H*0}.

\vspace{12pt}
\noindent {\bf \emph{The case $t^2 < 4q$ and $n_2(A) = 2$.}}
\vspace{12 pt}

Lemma \ref{Schoof1} implies that in this case $\P_q\left(\cC(A_{2,2},t), E \text{ supersingular} \right) = 0$ unless $t = 0$ and $q \equiv 3\pmod{4}$.  Lemma \ref{Schoof2} implies that for any $e \ge 2,\ \P_q\left(\cC(A_{2e,2e},t), E \text{ supersingular}\right) = 0$.  Therefore,
\[
\P_q\left(\cC(A_{n_1,2}, t),\ E \text{ supersingular}\right) = \P_q\left(\cC(A_{2,2},t),\ E \text{ supersingular}\right).
\]
The expression in Lemma \ref{Schoof1} matches the definition of $H^*_{n_1,2}(t,q,1)$ given by \eqref{H*2}.   

\vspace{12pt}
\noindent {\bf \emph{The case $t^2 < 4q$ and $n_2(A) = 1$.}}
\vspace{12 pt}

Lemma \ref{Schoof2} implies that for any $e \ge 3,\ \P_q\left(\cC(A_{e,e},t), E \text{ supersingular} \right)= 0$.  As above, Lemma \ref{Schoof1} implies that $\P_q\left(\cC(A_{2,2},t), E \text{ supersingular} \right) = 0$ unless $t = 0$ and $q \equiv 3\pmod{4}$.  If $2 \prec d$ then $d$ is a power of $2$ and $\lambda(d) = -1$. By Proposition \ref{CurveCount} we have 
\begin{eqnarray*}
& & \P_q\left(\cC(A_{n_1,1},t),\ E \text{ supersingular} \right) \\
& = &  \P_q\left(\cC(1,t),\ E \text{ supersingular}\right) D(t;n_1) - \delta_4(n_1,0)\P_q\left(\cC(A_{2,2},t),\ E \text{ supersingular}\right) D_{n_1,2}(t).
\end{eqnarray*}
Applying Lemmas \ref{S46withweights} and \ref{Schoof1} to this expression we check that it matches the remaining cases in the definition of $H^*_{n_1,1}(t,q,1)$.  

\vspace{12pt}
\noindent {\bf \emph{The case $t^2 = 4q$.}}
\vspace{12 pt}

In this case any $E/\F_q$ with $t_E=t = \pm 2\sqrt{q}$ has 
\[
E(\F_q) \cong \Z/(\sqrt{q} \mp 1)\Z \times \Z/(\sqrt{q} \mp 1)\Z
\] 
by  
\cite[Lemma 4.8(ii)]{Schoof}.  Therefore 
\[
\P_q(\cC(A_{n_1,n_2} , \pm 2\sqrt{q})) = \P_q(\cC(1,\pm 2\sqrt{q})) \delta_{n_1}(\pm \sqrt{q},1).
\] 
Applying Lemma \ref{S46withweights} to this expression concludes the proof of the theorem.
\end{proof}

We summarize the main results of this section in the following proposition.  We define two functions $\omega_{A}(q,d)$ and $\omega_{A}^*(q,d)$ that will be used extensively in Section \ref{assembly}.  Let 
\es{\label{omegadef} 
\omega_{A}(q,d) \eqdef \sum_{\substack{t^2<4q \\ p \nmid t}} U_{k-2}(t,q) H_{n_1(A),n_2(A)}(t,q,d),
}
and 
\es{\label{omega*def}
\omega^*_{A}(q,d) \eqdef  \sum_{t^2 <4q} U_{k-2}(t,q) H^*_{n_1(A),n_2(A)}(t,q,d).
} 
Note also $U_{k-2}(\pm 1) = (k-1)(\pm 1)^{k}.$ By \eqref{firststep} and Theorems \ref{CurvesToCn} and \ref{SupersingularCount} we have the following.
\begin{proposition}\label{s3mp}
We have \begin{multline*}\E_q(U_{k-2}(t,q) \Phi_A) = \frac{1}{q}\omega_{A}(q,d) + \frac{1}{q}\omega^*_{A}(q,1) \\ +  q^{k/2-1}\frac{(p-1)(k-1)}{24q} \left( \delta_{n_1(A)} (\sqrt{q}, 1) + (-1)^k \delta_{n_1(A)} (\sqrt{q}, -1) \right).
\end{multline*}
\end{proposition}  
\noindent Note that the right hand side above is $0$ if $q \not \equiv 1 \pmod*{n_2(A)}$.  Indeed, the congruence $q \equiv 1 \pmod*{n_2}$ is a necessary condition for $E(\F_q)$ to have a subgroup isomorphic to $\Z/n_2\Z \times \Z / n_2 \Z$ by, say, the Weil pairing on $E$.

\section{Trace Formulas}\label{traces}

In this section we derive an Eichler-Selberg trace formula for the congruence subgroups $\Gamma(N,M)$, which were defined in equation (\ref{GamNM}).  Our proof starts from the corresponding trace formula for $\Gamma_0(MN)$ with nebentype character.  
Let $S_k(\Gamma(N,M))$ be the space of weight $k \ge 2$ cusp forms for $\Gamma(N,M)$.  (Here we follow Sections 5.1 and 5.2 of the book of Diamond and Shurman \cite{DiamondShurman}.)  
We have an exact sequence \est{ 1 \rightarrow \Gamma(N,M)\rightarrow \Gamma_0(NM) \rightarrow (\Z/N\Z)^\times \rightarrow 1 .}  
The congruence subgroup $\Gamma_0(NM)$ acts on $S_k(\Gamma(N,M))$ via the slash operator with $\Gamma(N,M)$ acting trivially, so this action is via the quotient $(\Z/N\Z)^\times$.  Thus for each $d\in (\Z/N\Z)^{\times}$ we have the diamond operator \est{\langle d \rangle : S_k(\Gamma(N,M)) \rightarrow S_k(\Gamma(N,M))} given by \est{\langle d \rangle f = f_{|_{\gamma}}, \text{ for any } \gamma = \left(\begin{array}{cc} a & b \\ c & \delta \end{array} \right) \in \Gamma_0(NM) \text{ with } \delta \equiv d \pmod N. }  For a Dirichlet character $\chi$ mod $N$ (considered as a imprimitive character mod $NM$) the space of cusp forms of $\Gamma_0(NM)$ with nebentype character $\chi$ is the $\chi$-eigenspace of the diamond operators: \est{ S_k(\Gamma_0(NM),\chi) = \{ f\in S_k(\Gamma(N,M)) \text{ s.t. } \langle d \rangle f = \chi(d) f \text{ for all } d \in (\Z/N\Z)^\times \} . }   

We also define Hecke operators for this group.  For $p$ prime, let 
\est{
\alpha = \left( \begin{array}{cc} 1 & 0 \\ 0 & p \end{array}\right).
} 
Take a coset decomposition for the double coset \est{ \Gamma(N,M) \alpha \Gamma(N,M) = \bigsqcup_j \Gamma(N,M)\beta_j,} where $\beta_j \in \GL_2^+(\Q).$  Then 
\est{T_p  : S_k(\Gamma(N,M)) \rightarrow S_k(\Gamma(N,M))
} 
is defined by \est{ T_pf = \sum_{j} f_{|_{\beta_j}}.}  So long as $(p^v,N)=1$ we define the Hecke operators for prime powers inductively via \est{T_{p^v} = T_pT_{p^{v-1}} -p^{k-1}\langle p \rangle T_{p^{v-2}}.}

One can check that the diamond operators and Hecke operators commute with each other and that $T_p$ respects the decomposition \est{S_k(\Gamma(N,M)) = \bigoplus_{\chi \pmod* N} S_k(\Gamma_0(NM),\chi).}  
Comparing Propositions 5.2.2 and 5.3.1 of \cite{DiamondShurman} and Proposition 3.46 in the book of Knightly and Li \cite{KnightlyLi} shows that the Hecke operators defined above coincide with those of \cite{DiamondShurman}.
Thus if $(dq,NM)=1$ we have that 
\es{\label{GammaNM}\tr( \langle d \rangle T_q  | S_k(\Gamma(N,M))) = \sum_{\chi \pmod* N} \chi(d)\tr( T_q | S_k(\Gamma_0(NM), \chi)).} We use \eqref{GammaNM} to prove a trace formula for Hecke operators on the groups $\Gamma(N,M)$.

\begin{theorem}[Eichler-Selberg Trace Formula for $\Gamma(N,M)$]\label{ESTF} Let $q, M, N$ be positive integers, $d \in (\Z/N\Z)^\times$, and $k\geq 2$ an integer.  Let $T_q$ be the $q\textsuperscript{th}$ Hecke operator acting on $S_k(\Gamma(N,M))$ and $\langle d \rangle$ be the $d$-diamond operator acting on $S_k(\Gamma(N,M))$.  

Suppose that $M \mid N$, $(N,q)=1$, and $d^2q \equiv 1 \pmod*M$.  Let $L=(d^2q - 1,N)$. We have that
 \est{\frac{\tr( \langle d \rangle T_q| S_k(\Gamma(N,M)))}{\varphi(N)} = & \frac{k-1}{24}q^{k/2-1} \psi(NM) \left(\delta_N(q^{1/2} d, 1) + (-1)^k \delta_N(q^{1/2} d, -1)\right)  \\ 
 & -\frac{\psi(N^2)}{\psi(N^2/M^2)}\sum_{\Lambda \mid (L/M)} \frac{\varphi(\Lambda^2)\varphi(N/(M\Lambda))}{\varphi(N/M)} \sum_{t^2<4q 
 } U_{k-2}(t,q)H_{N,\Lambda M}(t,q,d) \\  
 &-\frac{1}{4} \sum_{b \mid q} \min(b,q/b)^{k-1}   \sum_{\substack{\tau \mid NM \\  g \mid (b-q/b) }}  \frac{\varphi(g)\varphi(N(M,g)/g)} {\varphi(N)}  \\ & \hspace{24pt} \times \Big( \delta_{N(M,g)/g}( y_\tau d, 1) + (-1)^k \delta_{N(M,g)/g}( y_\tau d, -1)\Big) \\  
 &+ \frac{\sigma(q)}{\varphi(N)}\delta(k,2),}  
where \begin{itemize}
\item $U_{k-2}(t,q)$ is a normalized Chebyshev polynomial (see \eqref{chebdef}), 
\item $g = (\tau,NM/\tau),$
\item $y_\tau$ is the unique element of $(\Z/(NM/g)\Z)^\times$ such that $y_\tau \equiv b \pmod{\tau}$ and $y_\tau \equiv q/b \pmod{NM/\tau}$.
\end{itemize}
\end{theorem}

We will prove Theorem \ref{ESTF} using the trace formula for $\Gamma_0(N)$ with nebentypus.  We refer to \cite{KnightlyLi} for the proof and also to \cite{LiWebpage} for the statement.
\begin{theorem}[Eichler-Selberg Trace Formula for $\Gamma_0(N),\chi$]\label{ESTFKL}Let $k\geq 2$ and $\chi(-1)=(-1)^k$. We have
\es{\label{TF} 
\tr (T_q | S_k(\Gamma_0(N),\chi))= & \frac{k-1}{12} \psi(N)\chi(q^{1/2})q^{k/2-1} \\ 
& -\frac{1}{2} \sum_{t^2<4q} U_{k-2}(t,q) \sum_{m^2 \mid (t^2-4q)}h_w\left(\frac{t^2-4q}{m^2}\right) \mu_\chi(t,m,q) \\ 
& - \frac{1}{2} \sum_{b \mid q} \min(b,q/b)^{k-1} \sum_{\tau} \varphi((\tau,N/\tau))\chi(y_\tau) \\ 
& + \delta(k,2) \mathds{1}_{\chi = 1} \sum_{\substack{ c \mid q \\ (N,q/c)=1}}c } 
where: 
\begin{itemize}  
\item $\chi(q^{1/2})=0$ if $q$ is not a perfect square,
\item $U_{k-2}(t,q)$ is a normalized Chebyshev polynomial (see \eqref{chebdef}),
\item 
\es{\mu_\chi(t,m,q) = \frac{\psi(N)}{\psi(N/(N,m))}\ssum_{c \in (\Z/N\Z)^\times} \chi(c), 
} 
where the $*_m$ indicates that $c$ runs through all elements of $(\Z/N\Z)^\times$ that lift to solutions of $c^2-tc+q\equiv 0 \pmod{N (N,m)}$,  
\item $\tau$ runs through all positive divisors of $N$ such that $(\tau,N/\tau)$ divides both $N/{\rm cond}(\chi)$ and $b-q/b$ where ${\rm cond}(\chi)$ is the conductor of $\chi$,
\item $y_\tau$ is the unique element of $\Z/(N/(\tau,N/\tau))\Z$ such that $y_\tau \equiv d \pmod \tau$ and $y_\tau \equiv q/d \pmod {N/\tau}$,
\item and $\mathds{1}_{\chi =1}$ is $1$ if $\chi$ is the trivial character and is $0$ otherwise.
\end{itemize}  
\end{theorem}

\begin{proof}[Proof of Theorem \ref{ESTF}]
We define the following four terms, which correspond to the terms appearing on the right hand side of Theorem \ref{ESTFKL} applied to the group $\Gamma_0(MN)$:
\begin{enumerate}
\item \est{T^{(i)}_\chi \eqdef \frac{k-1}{12} \psi(MN)\chi(q^{1/2})q^{k/2-1},}
\item \est{T^{(e)}_\chi  \eqdef \frac{1}{2} \sum_{t^2<4q} U_{k-2}(t,q) T^{(e)}_\chi (t),}
where \est{ T^{(e)}_\chi (t) \eqdef \sum_{m^2 \mid (t^2-4q)}h_w\left(\frac{t^2-4q}{m^2}\right) \mu_\chi(t,m,q) ,}
\item \est{ T^{(h)}_\chi \eqdef \frac{1}{2} \sum_{b \mid q} \min(b,q/b)^{k-1} T^{(h)}_\chi (b),}
where 
\est{T^{(h)}_\chi(b) \eqdef \sum_{\tau} \varphi((\tau,MN/\tau))\chi(y),}
\item 
\est{T^{(d)}_\chi \eqdef \delta(k,2)  \mathds{1}_{\chi = 1} \sum_{\substack{ c \mid q \\ (MN,q/c)=1}}c.}
\end{enumerate}
Then by \eqref{GammaNM}, Theorem \ref{ESTFKL}, and the fact that $S_k(\Gamma_0(MN),\chi)$ is $\{0\}$ unless $\chi(-1)=(-1)^k$ we have that 
\es{\label{tracesstep1}\frac{1}{\varphi(N)}\tr(\langle d \rangle T_q | S_k(\Gamma(N,M))) = T^{(i)} -T^{(e)}-T^{(h)} + T^{(d)}} where 
\begin{enumerate}
\item \est{T^{(i)} \eqdef \frac{1}{\varphi(N)} \sum_{\substack{\chi \pmod* N \\ \chi(-1) = (-1)^k}} \chi(d) T_\chi^{(i)} ,} 
\item \est{T^{(e)}  \eqdef  \frac{1}{2} \sum_{t^2<4q} U_{k-2}(t,q) T^{(e)}(t) ,} where 
\est{T^{(e)}(t) \eqdef \frac{1}{\varphi(N)}\sum_{\substack{\chi \pmod* N \\ \chi(-1) = (-1)^k}} \chi(d) T^{(e)}_\chi(t),}
\item \est{T^{(h)}  \eqdef  \frac{1}{2} \sum_{b \mid q} \min(b,q/b)^{k-1} T^{(h)}(b),} where \est{ T^{(h)}(b) \eqdef \frac{1}{\varphi(N)} \sum_{\substack{\chi \pmod* N \\ \chi(-1) = (-1)^k}} \chi(d) T^{(h)}_\chi(b) ,} 
\item \est{T^{(d)} \eqdef \frac{1}{\varphi(N)} \sum_{\substack{\chi \pmod* N \\ \chi(-1) = (-1)^k}} \chi(d) T^{(d)}_\chi .}
\end{enumerate}
We compute each of the $T^{(*)}$ in a separate section and check that they match what is claimed in Theorem \ref{ESTF}.
\subsection{Identity Term}
We have the orthogonality relation
\es{\label{ortho}\frac{1}{\varphi(N)} \sum_{\substack{\chi \pmod*N\\ \chi(-1)=(-1)^k}}\chi(d) = \frac{1}{2} \left(\delta_N(d,1) + (-1)^k\delta_N(d,-1)\right),} 
which shows that 
\es{T^{(i)} = & \frac{k-1}{12} q^{k/2-1} \psi(MN) \frac{1}{\varphi(N)} \sum_{\substack{\chi \pmod* N \\ \chi(-1) = (-1)^k}} \chi(dq^{1/2})  \\
= & \frac{k-1}{24}q^{k/2-1} \psi(MN) \left(\delta_N(q^{1/2}d, 1) + (-1)^k \delta_N(q^{1/2}d, -1)\right).} 
This matches the claimed identity term of Theorem \ref{ESTF}.

\subsection{Elliptic Term}\label{ETerm1}
The elliptic term $T^{(e)}$ is by far the most difficult to compute of the four terms in the trace formula.  We start with the case that $N$ is a power of a prime $\ell$ with $\ell \nmid q$. The goal of this section is to reduce the proof of this prime power case to the proof of Proposition \ref{MP}, which we state below.  We prove Proposition \ref{MP} in Section \ref{ProofMP}.  Finally, in Section \ref{MPGeneral} we consider the case where $N$ is not necessarily a prime power.

Let $0 \leq \alpha \leq \beta \leq \gamma$ be non-negative integers such that $(d^2q-1,\ell^\gamma)=\ell^\beta$.  Let $\Delta = t^2-4q$.  Taking $N=\ell^\gamma$ and $M = \ell^\alpha$ in the definition of $T^{(e)}(t)$ we have that 
\es{\label{Tet1} 
T^{(e)}(t) = \sum_{m^2 \mid \Delta} h_w\left(\frac{\Delta}{m^2}\right) \frac{1}{\varphi(\ell^\gamma)} \sum_{\substack{\chi \pmod* {\ell^\gamma} \\ \chi(-1) = (-1)^k}}  \chi(d) \mu_\chi(t,m,q).
} 
Also let $\kappa = v_\ell(m)$ and $\nu = v_\ell(\Delta)$ so that $\kappa \leq \lfloor \nu/2 \rfloor$.  We then have $(m,\ell^{\alpha+\gamma}) = \ell^{\min(\kappa,\alpha+\gamma)}$, which occurs often below.  Let 
\es{\label{wts}W(d) \eqdef \ssum_{c\in (\Z/\ell^{\alpha+\gamma}\Z)^{\times}} \delta_{\ell^{\gamma}}(c,d^{-1})} 
where the $*_m$ notation was explained in the statement of Theorem \ref{ESTFKL}. Note that $W(d)$ not only depends on $d$, but also on $\ell, \gamma$, and $\alpha$.

Using the orthogonality relation \eqref{ortho} we have 
\es{\label{orthogonalityresult} \frac{1}{\varphi(\ell^\gamma)} \sum_{\substack{\chi \pmod* {\ell^\gamma} \\ 
\chi(-1)=(-1)^k}} \chi(d) \mu_\chi(t,m,q) 
= & \frac{\psi(\ell^{\alpha+\gamma})}{\psi(\ell^{\alpha+\gamma}/\ell^{\min(\kappa,\alpha+\gamma)})}\ssum_{c \in (\Z/\ell^{\alpha+\gamma} \Z)^\times}\frac{1}{\varphi(\ell^\gamma)} \sum_{\substack{\chi \pmod* {\ell^\gamma} \\ 
\chi(-1)=(-1)^k}} \chi(dc)  \\
= & \frac{\psi(\ell^{\alpha+\gamma})}{\psi(\ell^{\alpha+\gamma-\min(\kappa,\alpha+\gamma)})}\frac{1}{2} \left(W(d) + (-1)^kW(-d)\right) .}  
We also set \es{C(t,q,d) \eqdef    \sum_{m^2 \mid (t^2-4q)} h_w\left(\frac{\Delta}{m^2}\right) \frac{\psi(\ell^{\alpha+\gamma})}{\psi(\ell^{\alpha+\gamma-\min(\kappa,\alpha+\gamma)})} W(d)} so that by \eqref{Tet1} and \eqref{orthogonalityresult} we have \est{T^{(e)}(t) = \frac{1}{2}\left( C(t,q,d) + (-1)^kC(t,q,-d)\right).} 
Note that $C(t,q,d)$ not only depends on $t,q,d$ but also on $\ell,\alpha,\beta,\gamma$.

Recall the definition of $H_{n_1,n_2}(t,q,d)$ from \eqref{HNMdef}. The proof of the following proposition is the subject of Section \ref{ProofMP}. 
\begin{proposition}\label{MP}
We have \est{C(t,q,d) =  2\frac{\psi(\ell^{2\gamma})}{\psi(\ell^{2\gamma-2\alpha})}\sum_{j=0}^{\beta-\alpha} \frac{\varphi(\ell^{2j})\varphi(\ell^{\gamma-\alpha-j})}{\varphi(\ell^{\gamma-\alpha})}H_{\ell^{\gamma},\ell^{\alpha+j}}(t,q,d).}
\end{proposition} 
Since $U_{k-2}(t,q)$ is an even (resp.~odd) function of $t$ when $k$ is even (resp.~odd) and $H_{n_1,n_2}(t,q,-d)= H_{n_1,n_2}(-t,q,d)$ we have \es{ \sum_{t^2<4q} U_{k-2}(t,q) C(t,q,d) = (-1)^k  \sum_{t^2<4q} U_{k-2}(t,q) C(t,q,-d).} By Proposition \ref{MP}, \eqref{tracesstep1}, and the definitions that follow 
\est{
T^{(e)} = \frac{\psi(\ell^{2\gamma})}{\psi(\ell^{2\gamma-2\alpha})}\sum_{j=0}^{\beta-\alpha} \frac{\varphi(\ell^{2j})\varphi(\ell^{\gamma-\alpha-j})}{\varphi(\ell^{\gamma-\alpha})}\sum_{t^2<4q} U_{k-2}(t,q)H_{\ell^{\gamma},\ell^{\alpha+j}}(t,q,d),
} 
matching exactly what is claimed in Theorem \ref{ESTF}.

\subsection{Hyperbolic Term}
Let $\tau \mid MN$ and $g= (\tau,NM/\tau)$.  We have that 
\est{T^{(h)}(b) = \frac{1}{\varphi(N)} \sum_{\substack{\chi \pmod*N \\ \chi(-1)=(-1)^k}} \chi(d) \sum_{\substack{\tau \mid MN \\ g \mid (b-q/b) \\ g \mid (MN/\cond(\chi))}} \varphi(g) \chi(y_\tau)} 
where 
$y_\tau$ is defined to be the unique element of $(\Z/(MN/g)\Z)^\times$ that satisfies $y_\tau \equiv b \pmod*\tau$ and $y_\tau \equiv q/b \pmod*{MN/\tau}$.  Swapping the order of summation we have  
\est{T^{(h)}(b) = \sum_{\substack{\tau \mid MN \\ g \mid (b-q/b) }} \frac{\varphi(g)}{\varphi(N)} \sum_{\substack{c \mid N \\ g \mid (MN/c)}} \sum_{\substack{\cond(\chi) = c \\ \chi(-1)=(-1)^k}} \chi(d y_\tau).}  
Recall the following almost-orthogonality relation for sums of characters of a given conductor (see e.g.~\cite[Section 2]{HB}) 
\es{\sum_{\substack{\cond(\chi) = c \\ \chi(-1)=(-1)^k}} \chi(a)= \frac{1}{2} \sum_{\delta \mid (a-1,c)} \varphi(\delta)\mu\left(\frac{c}{\delta}\right) + (-1)^k \frac{1}{2} \sum_{\delta \mid (a+1,c)} \varphi(\delta)\mu\left(\frac{c}{\delta}\right).} Let \est{V(\tau,d) \eqdef \sum_{\substack{c \mid N \\ g \mid (MN/c)}}\sum_{\delta \mid (dy_\tau-1,c)} \varphi(\delta)\mu\left(\frac{c}{\delta}\right) } 
so that we have 
\est{T^{(h)}(b) = \frac{1}{2} \sum_{\substack{\tau \mid MN \\ g \mid (b-q/b) }} \frac{\varphi(g)}{\varphi(N)} \left(V(\tau,d) + (-1)^k V(\tau,-d)\right).} 
By commutativity and associativity of Dirichlet convolution we have 
\es{\label{vform2} V(\tau,d) = \sum_{\substack{ a_1 \mid N \\ a_1 \mid dy_\tau -1 }} \varphi(a_1) \sum_{\substack{a_2 \mid (N/a_1) \\ g \mid a_2 M}} \mu \left( \frac{N}{a_1a_2}\right).} 
For any positive integers $n,g,M$ we have 
\es{\sum_{\substack{ a\mid n \\ g \mid aM}} \mu\left(\frac{n}{a}\right) = \sum_{\substack{a \mid n \\  
(g/(g,M)) \mid a }} \mu\left(\frac{n}{a}\right) = 
\delta_{(g/(g,M))}(n,0)
\sum_{ a \mid (ng/(g,M)) } \mu \left( \frac{ng}{a(g,M)}\right)  = \delta\left(n,\frac{g}{(g,M)}\right),} 
and applying this to \eqref{vform2} \es{ V(\tau,d) = \varphi\left( \frac{N(g,M)}{g}\right) \delta_{N(g,M)/g} \left(dy_\tau,1\right).} 
Therefore we have 
\est{ T^{(h)}(b) = \frac{1}{2} \sum_{\substack{\tau \mid MN \\ g \mid (b-q/b) }} \frac{\varphi(g)}{\varphi(N)} \varphi\left( \frac{N(g,M)}{g}\right)  \left( \delta_{N(g,M)/g} (dy_\tau,1) + (-1)^k \delta_{N(g,M)/g} (dy_\tau,-1)\right),} 
as claimed in the statement of Theorem \ref{ESTF}.

\subsection{Dual Term}
We have immediately that 
\est{ T^{(d)} = \frac{\sigma(q)}{\varphi(N)}\delta(k,2)} 
as all terms in the sum over $\chi$ vanish except the identity character, and $(N,q)=1$ by hypothesis. 
\end{proof}

\subsection{The Proof of Proposition \ref{MP}}\label{ProofMP}

In this section we compute $C(t,q,d)$, which is defined as a sum over divisors $m^2$ of $\Delta$.  We split up this sum based on the value of $v_{\ell}(m)$.  Let 
\es{\label{ckdef} c_\kappa(t,q,d) \eqdef \sum_{\substack{m^2 \mid \Delta \\ v_\ell(m) = \kappa}} h_w\left( \frac{\Delta}{m^2}\right) \frac{\psi(\ell^{\alpha+\gamma})}{\psi(\ell^{\alpha+\gamma- \min(\kappa, \alpha+\gamma)})} W(d)} 
so that 
\es{ C(t,q,d) = \sum_{\kappa \geq 0 } c_\kappa(t,q,d).}

Our goal is to rewrite these $c_\kappa(t,q,d)$ so that we can express $C(t,q,d)$ as a sum involving the class numbers $H_{\ell^\gamma,\ell^{\alpha+j}}(t,q)$ defined in Section \ref{curves2classnumbers}.  We start with some preliminary lemmas.

\subsubsection{Preliminary Lemmas} We prove several lemmas that will be used in the proof of Proposition \ref{MP}.  Our first goal in simplifying $c_\kappa(t,q,d)$ is to compute $W(d)$. We do this in Lemma \ref{weightscalc}.
Lemmas \ref{sublem2}, \ref{resclbound}, and \ref{anysolnisbinv} are technical and will be used in the proof of Lemma \ref{weightscalc}.  

Lemma \ref{sublem2} allows us to convert the sums over solutions to the Hecke polynomial $c^2-tc+q$ implicit in the $*_m$ notation in \eqref{wts} into the indicator function of $t$ lying in a certain residue class.  We will apply this lemma for for several different choices of $B,C,$ and $D$. 
\begin{lemma}\label{sublem2} 
Let $0\leq B \leq C\leq i \leq B+C$, $D \in (\Z/\ell^{C}\Z)^\times$, $D^2q\equiv 1 \pmod {\ell^B}$, and $t\in \Z$.  
The following statements are equivalent: 
\begin{enumerate}  
\item There exists a lift $c$ modulo $\ell^i$ of $D^{-1} \pmod*{\ell^C}$ that satisfies $c^2-tc+q \equiv 0 \pmod {\ell^i}$. 
\item The integer $t$ satisfies $ t \equiv Dq+D^{-1} \pmod{\ell^i}$, which is well-defined by Lemma \ref{DtnWellDef}.
\end{enumerate}
If the above equivalent statements hold, then every one of the $\ell^{i-C}$ lifts of $D^{-1} \pmod {\ell^{C}}$ to residue classes $c$ modulo $\ell^i$ satisfies $c^2 - tc + q \equiv 0 \pmod{\ell^i}$.
\end{lemma}

\begin{proof}  
Suppose that $c= D^{-1} + c_0\ell^C$ is a solution to $c^2-tc+q\equiv 0 \pmod {\ell^i}$ for some $c_0$ satisfying $0 \le c_0 < \ell^{i-C}$. Then we solve for $t$:
\est{
t \equiv & (D^{-1} +c_0\ell^{C})^{-1} (q+ (D^{-1}+c_0\ell^{C})^2) \pmod{\ell^i}\\ 
\equiv & Dq+D^{-1}\pmod {\ell^i}.}

Now assume that $t\equiv Dq+D^{-1}\pmod {\ell^i}$.  We parametrize lifts of $D^{-1}$ to $\Z/\ell^{i}\Z$ by $D^{-1}+ c'\ell^C$ with $0 \le c' < \ell^{i-C}$.  Then \est{ (D^{-1}+c'\ell^C)^2 - (Dq+D^{-1})(D^{-1}+c'\ell^C) + q \equiv & (D^{-1}-Dq)c'\ell^C  \equiv  0 \pmod{\ell^i}} 
for any value of $c'$. Thus, all $\ell^{i-C}$ lifts of $D^{-1} \pmod{\ell^C}$ are solutions to this quadratic polynomial modulo $\ell^i$.  
\end{proof}
The second technical lemma, Lemma \ref{resclbound}, will be used both in the proof of Lemma \ref{weightscalc} and also repeatedly in the proof of Proposition \ref{MP}.  Recall $\nu = v_\ell(\Delta) = v_{\ell}(t^2-4q)$.  
\begin{lemma}\label{resclbound} \begin{enumerate} \item If $\nu<2\beta$ then $t\not \equiv dq+d^{-1}\pmod{\ell^{\nu+1}}$.
\item If $\ell=2$ and $\nu < 2\beta$ then $t\not \equiv dq+d^{-1}\pmod{2^{\nu}}.$ 
\item If $\beta<\gamma$ and $\nu \neq 2\beta$ then $t\not \equiv dq+d^{-1}\pmod{\ell^{2\beta+1}}.$ 
\end{enumerate}
\end{lemma}

\begin{proof}
\begin{enumerate}
\item If $\nu<2\beta$ then $dq+d^{-1}\pmod{\ell^{\nu+1}}$ is well-defined by Lemma \ref{DtnWellDef}.  If $t\equiv dq+d^{-1}\pmod{\ell^{\nu+1}}$ then $\Delta \equiv (dq-d^{-1})^2 \pmod{\ell^{\nu+1}}$.  By definition $(d^2q-1, \ell^\gamma) = \ell^\beta$, so $v_\ell((d^2q-1)^2) \ge  2\beta \ge \nu+1$.  Therefore, $\Delta \equiv 0 \pmod{\ell^{\nu +1}}$, a contradiction with the definition of $\nu$.

\item If $\ell=2$ and $t \equiv dq+d^{-1}\pmod{2^{\nu}}$ but not mod $2^{\nu+1}$ then we can write $t \equiv dq+d^{-1}+2^\nu \pmod{2^{\nu+1}}$.  We get $\Delta \equiv (dq-d^{-1})^2 \equiv 0 \pmod{2^{\nu+1}}$ nonetheless, and get a contradiction with the definition of $\nu$.  
\item The residue class $dq+d^{-1}\pmod{\ell^{2\beta+1}}$ is well-defined by Lemma \ref{DtnWellDef} because $\beta<\gamma$. If $t\equiv dq+d^{-1}\pmod{\ell^{2\beta+1}}$ then $\Delta \equiv (dq-d^{-1})^2 \pmod{\ell^{2\beta+1}}$ and thus $\nu=2\beta$ since $(d^2q-1,\ell^\gamma)=\ell^\beta$.   
\end{enumerate}
\end{proof}

The third technical lemma, Lemma \ref{anysolnisbinv}, concerns the situation where $\beta$ is as large as possible, i.e.~$\beta=\gamma$, and says that any solution modulo $\ell^i$ is a lift of the distinguished $d^{-1}$ modulo $\ell^\gamma$.
\begin{lemma}\label{anysolnisbinv} Let $ \gamma\geq 1$, $d\in (\Z/\ell^\gamma\Z)^\times$, $\ell\nmid q$, and suppose that $d^2q\equiv 1 \pmod {\ell^{\gamma}}$ and $t\equiv dq+d^{-1} \pmod {\ell^{2\gamma}}$, which is well-defined by Lemma \ref{DtnWellDef}. For any $i \geq 2 \gamma$, if $c$ is a solution to $c^2-tc+q\equiv 0 \pmod {\ell^{i}}$  then $c \equiv d^{-1}\pmod {\ell^\gamma}$.
\end{lemma}
\begin{proof}
Suppose first that $\ell\neq 2$.  If $c_0$ satisfies $c_0^2-tc_0+q\equiv 0 \pmod {\ell^{i}}$ then 
\est{c_0^2 -(dq+d^{-1})c_0 +q \equiv 0 \pmod {\ell^{2\gamma}}} because $i\geq 2\gamma$. By completing the square we find 
\est{ (2c_0 - ( dq+d^{-1}))^2\equiv  (dq-d^{-1})^2  \equiv  0 \pmod {\ell^{2\gamma}}.} 
We conclude $2c_0 \equiv dq +d^{-1}\pmod{\ell^\gamma}$. Since $dq \equiv d^{-1}\pmod{\ell^\gamma}$, we have $c_0 \equiv d^{-1} \pmod{\ell^\gamma}$.

Now consider the case $\ell=2$.  Since $2 \nmid q$ we have that $d,d^{-1},$ and $q$ are all odd. Hence $t\equiv dq+d^{-1}\pmod {2^{2\gamma}}$ is even.  Then $2^{-1}(dq+d^{-1})$ is defined modulo $2^{2\gamma-1}$, and $2^{-2}(dq+d^{-1})^2$ is defined modulo $2^{2\gamma}$.  By completing the square we have 
\est{\left(c_0- 2^{-1}(dq+d^{-1})\right)^2 \equiv  2^{-2}(dq-d^{-1})^2  \equiv 0 \pmod {2^{2\gamma}}.}   
Thus $c_0 \equiv  2^{-1} (dq+d^{-1})  \equiv  d^{-1} \pmod{2^\gamma}$.
\end{proof}

With Lemmas \ref{sublem2}, \ref{resclbound}, and \ref{anysolnisbinv} we may now evaluate $W(d)$. Recall the notation $\ssum$ is a sum over elements $c \in (\Z/\ell^{\alpha+\gamma}\Z)^{\times}$ that satisfy $c^2 - tq + q \equiv 0 \pmod{\ell^{\alpha+\gamma}}$ and lift to solutions of this polynomial modulo $\ell^{\alpha+\gamma+\min(\kappa,\alpha+\gamma)}$.  In particular, it depends on the value of $m$.  Also recall the definition $D(t;n)=\delta_n(t,dq+d^{-1})$ from \eqref{Dtn}.
Let $S(a,n)$ be the number of solutions to $x^2-a \equiv 0 \pmod n$.  Later, in Lemma \ref{sqrts}, we give an explicit evaluation for $S(a,n)$.  

\begin{lemma}\label{weightscalc}
\begin{enumerate}
\item If $\beta <\gamma$ then 
\est{\ssum_{c\in (\Z/\ell^{\alpha+\gamma}\Z)^{\times}} \delta_{\ell^{\gamma}}(c,d^{-1})  = 
\begin{cases} 
D(t;\ell^{\gamma+ \alpha+\min(\kappa,\alpha+\gamma)})  \ell^{\alpha} & \text{ if } \min(\kappa,\alpha+\gamma) \leq \beta-\alpha \\
D(t;\ell^{\beta+\gamma})  \ell^{\beta-\min(\kappa,\alpha+\gamma)} & \text{ if } \beta-\alpha \leq \min(\kappa,\alpha+\gamma).
\end{cases}}  

\item If $ \beta = \gamma $ then 
\est{ \ssum_{c\in (\Z/\ell^{\alpha+\gamma}\Z)^{\times}} \delta_{\ell^{\gamma}}(c,d^{-1})= 
\begin{cases}D(t;\ell^{\gamma+ \alpha+\min(\kappa,\alpha+\gamma)})  \ell^{\alpha} & \text{ if }\min(\kappa,\alpha+\gamma) \leq \gamma-\alpha \\
D(t;\ell^{2\gamma})  \frac{S\left(\Delta^*,\ell^{\alpha+\gamma+\min(\kappa,\alpha+\gamma)}\right)}{\ell^{\min(\kappa,\alpha+\gamma)}} & \text{ if } \gamma-\alpha \leq \min(\kappa,\alpha+\gamma),
\end{cases}}
where
\est{\Delta^*= \begin{cases} t^2-4q & \ell\neq 2 \\ \frac{t^2-4q}{4} & \ell=2. 
\end{cases}} 
\end{enumerate}
\end{lemma}
Note that in the case $\min(\kappa,\alpha+\gamma)=\beta-\alpha$ both ``if'' statements are true, and the evaluations for $W(d)$ coincide.   Note also that if $\ell = 2$ then $W(d)$ is only supported on $t\in 2\Z$ so that $\Delta^*\in \Z$, or else $W(d)$ is $0$.

\begin{proof}
When $\min(\kappa, \alpha+\gamma) \leq \beta-\alpha,$ the conclusion of Lemma \ref{weightscalc} is immediate from Lemma \ref{sublem2} taking $B=\beta$, $C=\gamma$, $D=d$ and $i=\alpha + \gamma + \min(\kappa,\alpha+\gamma)$.  For the rest of the proof we suppose that $\beta-\alpha \leq \min(\kappa, \alpha+ \gamma)$. 

\begin{enumerate}
\item Assume that $\beta<\gamma.$  Let $k=\alpha + \gamma+ \min(\kappa,\alpha+ \gamma)$.  We have already treated the cases $\alpha + \gamma \leq k \leq \beta + \gamma$, so we may assume that $\beta+\gamma < k$.  By Lemma \ref{sublem2} with $B=\beta$, $C=\gamma$, $D=d$ and $i=\alpha+\gamma$ we have \est{W(d) = \ssum_{c\in (\Z/\ell^{\alpha+\gamma}\Z)^{\times} } \delta_{\ell^\gamma} (c,d^{-1}) = D(t;\ell^{\alpha+\gamma}) \ssum_{c\in (\Z/\ell^{\alpha+\gamma}\Z)^{\times} } \delta_{\ell^\gamma} (c,d^{-1}) .}
By Lemma \ref{resclbound} (iii) we have that $k\leq \alpha + \beta+\gamma$ when the congruence condition above is satisfied, so we may apply Lemma \ref{sublem2} with $B = \beta$, $C=\alpha + \gamma$, $D=c^{-1}$ and $i=k$. Thus 
\est{W(d) = \sum_{c\in (\Z/\ell^{\alpha+\gamma}\Z)^{\times}  } \delta_{\ell^k}(t,c^{-1}q+c)  \delta_{\ell^\gamma} (c,d^{-1}) .} 

We can simplify this further.  Choose any $c_0 \in (\Z/\ell^{\alpha+\gamma}\Z)^{\times} $ such that $c_0 \equiv d^{-1}\pmod {\ell^\gamma}$.  We parameterize the $c \in (\Z/\ell^{\alpha+\gamma}\Z)^{\times} $ such that $c \equiv d^{-1}\pmod {\ell^\gamma}$ by $c=c_0 + i\ell^\gamma$, where $0 \le i < \ell^{\alpha}$.  Lemma \ref{DtnWellDef} implies that the residue class $c^{-1} q + c  \pmod{\ell^k}$ is well-defined.  Since 
\[
(c_0 + i\ell^\gamma)^{-1} \equiv c_0^{-1} - c_0^{-2} i \ell^{\gamma} \pmod{\ell^k},
\] 
we have
\begin{multline}
\{ c^{-1}q+c \pmod*{\ell^k} : c \in (\Z/\ell^{\alpha+\gamma}\Z)^{\times} \text{ and } c \equiv d^{-1}\pmod*{\ell^\gamma}\} \\ 
= \{ c_0^{-1} q + c_0 - (c_0^{-2} q -1)i \ell^\gamma \pmod*{\ell^k} : 0 \le i < \ell^{\alpha}\}.
\end{multline}

We have that $c_0 \equiv d^{-1}\pmod {\ell^\gamma},\ \beta< \gamma,\ k > \beta+\gamma$, and $d^2 q - 1 \equiv 0 \pmod{\ell^\beta}$, but ${d^2 q - 1} \not\equiv 0 \pmod{\ell^{\beta+1}}$.
Therefore, there exists $z$ with $0 \le z < \ell^{k-(\gamma+\beta)}$ such that 
\[
-(c_0^{-2} q -1) i \ell^\gamma \equiv z i \ell^{\beta+\gamma} \not\equiv 0 \pmod{\ell^k}.
\]

This implies 
\est{ W(d)  =  \sum_{i \in \Z/\ell^\alpha\Z} \delta_{\ell^k}(t, c_0^{-1}q + c_0+ i z \ell^{\beta + \gamma}).} 
As we vary over all $i$ satisfying $0 \le i < \ell^\alpha$, we see that $c_0^{-1}q + c_0 + zi\ell^{\beta+ \gamma} \pmod*{\ell^k}$ represents each residue classes modulo $\ell^k$ that is a lift of $c_0^{-1}q+c_0 \equiv dq+d^{-1}\pmod*{\ell^{\beta+\gamma}}$ exactly $\ell^{\alpha-(k-\gamma-\beta)}$ times. Thus 
\begin{eqnarray*}
W(d) & = & \ell^{\alpha - (k-\beta - \gamma)} \sum_{i \in \Z/\ell^{k-\beta-\gamma}\Z} \delta_{\ell^k}(t,c_0^{-1}q+c_0 + iz\ell^{\beta+\gamma}) \\ 
 & = & \ell^{\beta - \min(\kappa,\alpha+\gamma)} D(t;\ell^{\beta+\gamma}).
 \end{eqnarray*}

\item Assume that $\beta=\gamma$. Let $k = \alpha + \gamma + \min(\kappa, \alpha+ \gamma)$, so $\alpha+ \gamma \leq k \leq 2\alpha + 2\gamma$.  In fact we may assume that $2\gamma \leq k \leq 2 \alpha + 2 \gamma$ as the other cases have already been treated.  By Lemma \ref{sublem2} with $B=\gamma$, $C=\gamma$, and $D=d$ we have 
\est{ W(d) = \ssum_{c\in (\Z/\ell^{\alpha+\gamma}\Z)^{\times}} \delta_{\ell^{\gamma}}(c,d^{-1})= D(t;\ell^{2\gamma})\ssum_{c\in (\Z/\ell^{\alpha+\gamma}\Z)^{\times}} \delta_{\ell^{\gamma}}(c,d^{-1}).} 
We claim that for each $k$ satisfying $2\gamma \leq k \leq 2\alpha + 2\gamma$  and for each $c \in (\Z/\ell^{\alpha+\gamma}\Z)^\times$ such that $c\equiv d^{-1}\pmod*{\ell^\gamma}$ there are either $\ell^{k-(\alpha+\gamma)}$ lifts $\tilde{c}$ of $c$ satisfying $\tilde{c}^2-t\tilde{c}+q\equiv 0 \pmod{\ell^k}$ or $0$ such lifts.  

To see this, for each $c \in (\Z/\ell^{\alpha+\gamma}\Z)^\times$ such that $c\equiv d^{-1}\pmod*{\ell^\gamma}$ we apply Lemma \ref{sublem2} with $B=\gamma$, $C = \alpha + \gamma$ and $D=c^{-1}\pmod*{\ell^{\alpha+\gamma}}$.  For each lift $\tilde{c}\pmod*{\ell^{\alpha+2\gamma}}$ of $c$ thus produced, we apply Lemma \ref{sublem2} again with $B=\gamma$, $C=\alpha+ 2\gamma$, and $D=\tilde{c}^{-1}\pmod*{\ell^{\alpha+2\gamma}}$.  We get either exactly $\ell^{k-(\alpha+\gamma)}$ lifts of $c$, or none, for each $k$ satisfying $2\gamma \leq k \leq 2\alpha + 2\gamma\leq \alpha + 3\gamma$.  
  
Completing the square shows that there are $S(\Delta^*,\ell^k)$ total solutions to $c^2-tc+q \equiv 0 \pmod*{\ell^k}$.  Note that if $\ell =2$ we may assume $t$ is even because of the factor $\delta_{\ell^{2\gamma}} (t, dq+d^{-1})$, so we have that $\Delta^*\in \Z$, and completing the square makes sense.  By Lemma \ref{anysolnisbinv}, all solutions to $c^2-tc+q \equiv 0 \pmod*{\ell^k}$ with $k\geq 2\gamma$ have $c \equiv d^{-1} \pmod*{\ell^\gamma}$.  By the previous claim, this implies that exactly 
\est{\frac{S(\Delta^*,\ell^k)}{\ell^{k-(\alpha+\gamma)}}} 
of the $c \in (\Z/\ell^{\alpha+\gamma}\Z)^\times$ with $c\equiv d^{-1} \pmod*{\ell^\gamma}$ have lifts to solutions modulo $\ell^k$.  
\end{enumerate}
\end{proof}

The next lemma gives an evaluation of $S(a,n)$.  It is a special case of more general results in several variables going back at least to Jordan \cite{Jordan}.
\begin{lemma}\label{sqrts}
For $p$ an odd prime we have \est{S(a,p) = 1+\legen{a}{p}.}  
For an odd prime power $p^\epsilon$ with $\epsilon\geq 2$ we have 
\est{S(a,p^\epsilon) = \begin{cases} S(a,p) & \text{ if } v_p(a) = 0\\ 0 & \text{ if } v_p(a)=1 \\ p S(a/p^2,p^{\epsilon-2}) & \text{ if } v_p(a)\geq 2 . \end{cases}}  If $p=2$ we have\est{ S(a,2) & = 1 \\ S(a,4) & = \begin{cases} 2 & \text{ if } a \equiv 0,1 \pmod 4 \\ 0 & \text{ if } a \equiv 2,3 \pmod 4\end{cases} } 
and if 
$\epsilon\geq 3$ then 
\est{ S(a,2^\epsilon) = \begin{cases} 4 \delta_8(a,1) & \text{ if } v_2(a)=0 \\ 0 & \text{ if } v_2(a)=1 \\ 2 S(a/4,2^{\epsilon-2}) & \text{ if } v_2(a) \geq 2.\end{cases}}  
\end{lemma}
\begin{proof}[Proof Sketch]
Clear factors of $p$ common to $a$ and $p^\epsilon$, and apply Hensel's lemma.
\end{proof}

The next lemma is a standard result relating class numbers of orders to the class numbers of the maximal orders containing them.  For a proof, see for example \cite[Corollary 7.28]{primesoftheform}.
\begin{lemma}\label{cor728}
For $d<0$, $d\equiv 0,1\pmod 4$ and $f \in \N$ we have \es{h_w(f^2d) = h_w(d) f \prod_{p \mid f} \left(1 - \legen{d}{p}\frac{1}{p}\right).}
In particular, if $f$ is a prime power dividing $d$ then \es{h_w(f^2d) = h_w(d) f } and if $(d,f)=1$ and $d$ is a square modulo $f$ (resp.~$4f$ if $2\mid f$) then \es{h_w(f^2d) = h_w(d) \varphi(f).} 
\end{lemma}
The last lemma gives identities between class numbers that we will use later. 
\begin{lemma}\label{H} 
Let $0 \leq \alpha \leq \beta\leq \gamma$ and $\ell \nmid q$ prime.  Assume $(d^2q-1,\ell^{\gamma})=\ell^{\beta}$.  
If $\beta < \gamma$, then
\begin{multline*} 
2\sum_{j=0}^{\beta-\alpha} \frac{\varphi(\ell^{2j})\varphi(\ell^{\gamma-\alpha-j})}{\varphi(\ell^{\gamma-\alpha})}H_{\ell^{\gamma},\ell^{\alpha+j}}(t,q,d) \\ 
= H\left( \frac{\Delta}{\ell^{2\alpha}} \right) D(t;\ell^{\alpha+\gamma}) 
+ \sum_{j=1}^{\beta-\alpha} H\left( \frac{ \Delta}{\ell^{2(\alpha+j)}} \right) \left( \ell^{j}D(t;\ell^{\alpha+\gamma+j}) -\ell^{j-1}D(t;\ell^{\alpha+\gamma+j-1})\right) \end{multline*}
and if $\beta = \gamma$ we have 
\begin{multline*} 
2\sum_{j=0}^{\beta-\alpha} \frac{\varphi(\ell^{2j})\varphi(\ell^{\gamma-\alpha-j})}{\varphi(\ell^{\gamma-\alpha})}H_{\ell^{\gamma},\ell^{\alpha+j}}(t,q,d)  =H\left( \frac{\Delta}{\ell^{2\alpha}} \right)D(t;\ell^{\alpha+\gamma}) \\
+ \sum_{j=1}^{\gamma-\alpha-1} H\left( \frac{ \Delta}{\ell^{2(\alpha+j)}} \right) \left( \ell^{j}D(t;\ell^{\alpha+\gamma+j}) -\ell^{j-1} D(t;\ell^{\alpha+\gamma+j-1})\right)  + \ell^{\gamma-\alpha}D(t;\ell^{2\gamma})H\left( \frac{\Delta}{\ell^{2\gamma}}\right) .\end{multline*}
\end{lemma}
\begin{proof}
The expression in the lemma is by definition 
\begin{multline*}
\sum_{j=0}^{\beta-\alpha}  \frac{\varphi(\ell^{2j})\varphi(\ell^{\gamma-\alpha-j})}{\varphi(\ell^{\gamma-\alpha})} 
\Bigg[H\left(\frac{\Delta}{\ell^{2(\alpha+j)}}\right)\delta_{\ell^{\alpha+j}}(1,d^2q)D(t;\ell^{\alpha+\gamma+j})   \\  
- \sum_{k = 1}^{\gamma-(\alpha+j)-1}  H\left(\frac{\Delta}{\ell^{2(\alpha+j+k)}}\right)\delta_{\ell^{\alpha+k+j}}(1,d^2q)\left(D(t;\ell^{\alpha+\gamma+j-k-1})   -D(t;\ell^{\alpha+\gamma+j+k}) \right)\Bigg]. 
\end{multline*}
Note that $\delta_{\ell^{\alpha+j}}(1,d^2q) = 1$ since this is equivalent to $j \leq \beta-\alpha$.  The 
term $\delta_{\ell^{\alpha+k+j}}(1,d^2q)$ is equal to $1$ if and only if 
$j+k \leq \beta-\alpha$.  Swapping order of summation and writing $j+k=i$ this is \begin{multline*} 
\sum_{j=0}^{\beta-\alpha} \frac{\varphi(\ell^{2j})\varphi(\ell^{\gamma-\alpha-j})}{\varphi(\ell^{\gamma-\alpha})} H\left(\frac{\Delta}{\ell^{2(\alpha+j)}}\right)D(t;\ell^{\alpha+\gamma+j}) \\ - \sum_{i=1}^{\min(\gamma-\alpha-1,\beta-\alpha)} \left(\sum_{j=0}^{i-1} \frac{\varphi(\ell^{2j})\varphi(\ell^{\gamma-\alpha-j})}{\varphi(\ell^{\gamma-\alpha})}\right) H\left(\frac{\Delta}{\ell^{2(\alpha+i)}}\right)\left(D(t;\ell^{\alpha+\gamma+i-1})   - D(t;\ell^{\alpha+\gamma+i}) \right).\end{multline*}
It is straightforward to show that
\est{
\sum_{j=0}^{i-1} \frac{\varphi(\ell^{2j})\varphi(\ell^{\gamma-\alpha-j})}{\varphi(\ell^{\gamma-\alpha})} = 
\begin{cases} \ell^{i-1} & \text{ if } i-1< \gamma- \alpha \\ 
\psi(\ell^{\gamma-\alpha}) &  \text{ if } i-1=\gamma-\alpha . 
\end{cases}}
We now consider two cases.  If $\beta < \gamma$ we have $\beta - \alpha \leq \gamma-\alpha-1$, so combining terms gives
\est{ 
H\left( \frac{\Delta}{\ell^{2\alpha}} \right)D(t;\ell^{\alpha+\gamma})
+ \sum_{j=1}^{\beta-\alpha} H\left( \frac{ \Delta}{\ell^{2(\alpha+j)}} \right) \left( \ell^{j}D(t;\ell^{\alpha+\gamma+j}) -\ell^{j-1} D(t;\ell^{\alpha+\gamma+j-1})\right).
}  
On the other hand if $\beta = \gamma$ the expression is 
\begin{multline*}
H\left( \frac{\Delta}{\ell^{2\alpha}} \right) D(t;\ell^{\alpha+\gamma}) 
+ \sum_{j=1}^{\gamma-\alpha-1} H\left( \frac{ \Delta}{\ell^{2(\alpha+j)}} \right) \left( \ell^{j}D(t;\ell^{\alpha+\gamma+j}) -\ell^{j-1}D(t;\ell^{\alpha+\gamma+j-1})\right) \\ 
+ \ell^{\gamma-\alpha} D(t;\ell^{2\gamma})H\left( \frac{\Delta}{\ell^{2\gamma}}\right).
\end{multline*}
\end{proof}

\subsubsection{Evaluating $C(t,q,d)$}
The proof of Proposition \ref{MP} breaks into three main cases:  $(\beta< \gamma)$, $(\beta = \gamma \text{ and } \nu<2(\alpha+\gamma))$, and $(\beta = \gamma \text{ and } \nu\geq 2(\alpha+\gamma))$.

\vspace{12pt}
\noindent {\bf \emph{The case $\beta <\gamma$.}}
\vspace{12 pt}

We further split into three cases according to the value of $\kappa$.  The ranges $\kappa<\beta-\alpha$, $\beta-\alpha\leq \kappa \leq \beta$, and $\kappa >\beta$ will each be treated differently.

\vspace{12 pt}
\noindent {\bf \emph{The case $\beta<\gamma$ and $\kappa >\beta$}.}\\
\vspace{12 pt}  

We have $\beta-\alpha \leq \beta < \kappa$ and $\beta-\alpha \leq \alpha + \gamma$ so that the second case of Lemma \ref{weightscalc} (i) applies. For such $\kappa$ we have 
\es{  
c_\kappa(t,q,d) = \sum_{\substack{m^2 \mid \Delta \\ v_\ell(m) = \kappa}} h_w\left( \frac{\Delta}{m^2}\right) \frac{\psi(\ell^{\alpha+\gamma})}{\psi(\ell^{\alpha+\gamma- \min(\kappa, \alpha+\gamma)})} \ell^{\beta-\min(\kappa, \alpha+\gamma)} D(t;\ell^{\beta+\gamma}).
} 
Since $2\beta < 2\kappa \leq \nu$ and $2\beta+1\leq \beta+\gamma$, Lemma \ref{resclbound} (iii) applies, and so $D(t;\ell^{\beta+\gamma}) = 0$.
We conclude that $c_\kappa(t,q,d)=0$ for all $\kappa>\beta$.  
 
\vspace{12 pt}
\noindent {\bf \emph{The case $\beta < \gamma$ and $\beta-\alpha \leq \kappa \leq \beta$}.}
\vspace{12 pt}

We have $\min(\kappa , \alpha+\gamma)= \kappa$ since $\kappa \leq \beta <\gamma \leq \alpha+\gamma$.  The second case of Lemma \ref{weightscalc} (i)  applies and we again have  
\est{  c_\kappa(t,q,d) = \sum_{\substack{m^2 \mid \Delta \\ v_\ell(m) = \kappa}} h_w\left( \frac{\Delta}{m^2}\right) \frac{\psi(\ell^{\alpha+\gamma})}{\psi(\ell^{\alpha+\gamma- \kappa})} \ell^{\beta-\kappa} D(t;\ell^{\beta+\gamma}).} 
Suppose that $\nu \neq 2\beta$.  Because $2\beta+1 \leq \beta+\gamma$, Lemma \ref{resclbound} (iii) again applies and we see that $D(t;\ell^{\beta+\gamma}) = 0$.  
Therefore for any $\beta-\alpha \leq \kappa \leq \beta$ with $\beta<\gamma$ we have 
\es{  
c_\kappa(t,q,d) = \delta(\nu, 2\beta) \sum_{\substack{m^2 \mid \Delta \\ v_\ell(m) = \kappa}} h_w\left( \frac{\Delta}{m^2}\right) \frac{\psi(\ell^{\alpha+\gamma})}{\psi(\ell^{\alpha+\gamma- \kappa})} \ell^{\beta-\kappa} D(t;\ell^{\beta+\gamma}).
}

\vspace{12 pt}
\noindent {\bf \emph{The case $\beta < \gamma$ and $\kappa < \beta-\alpha$}.}
\vspace{12 pt}
 
In this range of $\kappa$ we again have $\min(\kappa , \alpha+\gamma) = \kappa< \beta-\alpha$ since $\kappa < \beta-\alpha \leq \alpha+\gamma$.  The first case of Lemma \ref{weightscalc} (i) then applies and we have for $\kappa < \beta-\alpha$ that 
\es{\label{b<grange1} 
c_\kappa(t,q,d) =  \sum_{\substack{m^2 \mid \Delta \\ v_\ell(m) = \kappa}} h_w\left( \frac{\Delta}{m^2}\right) \ell^{\alpha+\kappa} D(t;\ell^{\alpha+\gamma+\kappa}).
}

We would like to apply Lemma \ref{cor728} to the class numbers occurring in (\ref{b<grange1}), but this requires a certain hypothesis to hold.  We verify this hypothesis with the following lemma.

\begin{lemma}\label{claim}
Suppose that $t \equiv dq+d^{-1} \pmod{\ell^{\alpha+\gamma+\kappa}},\ \beta < \gamma$ and $\kappa <\beta-\alpha$.  Then 
\est{v_\ell \left(\frac{\Delta}{\ell^{2\alpha}m^2}\right)\geq 2,} 
and 
\est{\frac{\Delta}{\ell^{2\alpha+2}m^2} \equiv 0,1 \pmod 4.}
\end{lemma} 
\begin{proof} 
Since $\kappa <\beta-\alpha$, if $\nu = 2\beta$ then $v_\ell(\Delta/(\ell^{2\alpha}m^2)) = 2\beta-2\alpha - 2\kappa \geq 2$.  
So suppose that $\nu \neq 2 \beta$.  Then by Lemma \ref{resclbound} (i) and (iii) we have that $\alpha+\gamma +\kappa \leq \min(\nu,2\beta)$.  Thus 
\est{
v_\ell\left( \frac{\Delta}{\ell^{2\alpha}m^2}\right) = \nu -2\alpha-2\kappa \geq 2 \gamma +\nu -2 \min (\nu,2\beta).}  If $2\beta <\nu$ then \est{ 2 \gamma +\nu -2\min(\nu,2\beta) = \nu + 2\gamma -4 \beta > 2(\gamma-\beta)\geq 2.} 
If $\nu < 2\beta$ then 
\est{2\gamma + \nu - 2\min(\nu,2\beta) = 2\gamma -\nu>2(\gamma-\beta)\geq 2.} 

If $\ell \neq 2$ then $\ell^2 \equiv 1 \pmod*4$ so the second statement of the lemma follows immediately from the first.  Now suppose $\ell=2$.  There are two cases, $0 \leq \kappa \leq \beta-\alpha-2$, and $\kappa = \beta - \alpha -1$. 

In the case $\kappa \leq \beta-\alpha-2$, if $\nu \geq 2\beta$ then \est{ \nu - 2\alpha-2\kappa \geq  2 \beta -2\alpha -2\kappa \geq 4} so that $\Delta / (\ell^{2\alpha}m^2) \equiv 0 \pmod*4$. If $\nu < 2\beta$ then by Lemma \ref{resclbound} (ii) we have $\alpha +\gamma + \kappa \leq \nu -1$, and so $\nu -2\alpha -2\kappa \geq 4$ so that $\Delta / (\ell^{2\alpha}m^2) \equiv 0 \pmod*4$ as well.

Now suppose that $\kappa = \beta-\alpha -1$.  In this case, by assumption $t\equiv dq+d^{-1} \pmod*{2^{\beta+\gamma-1}}$.  We calculate that 
\est{\Delta \equiv (dq-d^{-1})^2 \pmod{2^{\beta+\gamma + 1}}} 
so by definition of $\beta$ and the fact that $\beta<\gamma$ we have that $\Delta/2^{2\beta} = \Delta / \ell^{2(\alpha + \kappa +1)}$ is a square modulo $4$, as was to be shown.  \end{proof}

By the first part of Lemma \ref{claim}, we may apply Lemma \ref{cor728} to \eqref{b<grange1} with $f=\ell^\alpha$ and $d = \Delta/(\ell^{2\alpha}m^2)$, which is divisible by $\ell$.  Thus for $\beta<\gamma$ and $\kappa < \beta - \alpha$ we have \est{ c_\kappa(t,q,d) = \sum_{\substack{m^2 \mid \Delta \\ v_\ell(m) = \kappa}} h_w\left( \frac{\Delta}{\ell^{2\alpha}m^2}\right) \ell^{2\alpha+\kappa} D(t;\ell^{\alpha+\gamma+\kappa}).}

We now assemble the three cases for the ranges of $\kappa$.  The definition of the Hurwitz-Kronecker class number $H$ from \eqref{HKcn} implies that if $\nu \geq 2(\kappa+\alpha)$, then 
\begin{multline}\label{HWtrick} \sum_{\substack{m^2 \mid \Delta \\ v_\ell(m)=\kappa}} h_w \left( \frac{\Delta}{\ell^{2\alpha}m^2}\right) \\ = \begin{cases} H\left(\frac{\Delta}{\ell^{2(\alpha+\kappa)}}\right) - H\left( \frac{\Delta}{\ell^{2(\alpha+ \kappa+1)}}\right) & \text{ if } \nu \geq 2 (\alpha+\kappa+1)  \text{ and } \frac{\Delta}{\ell^{2(\alpha+\kappa+1)}}\equiv 0,1\pmod*{4} \\  
H\left(\frac{\Delta}{\ell^{2(\alpha+\kappa)}}\right) & \text{ otherwise.} 
\end{cases}
\end{multline}    
By the second part of Lemma \ref{claim} the conditions on the first line of \eqref{HWtrick} hold. So when $\beta < \gamma$ we have that 
\est{
C(t,q,d)=&  \sum_{\kappa \geq 0 } c_\kappa(t,q,d) \\ = & \sum_{\kappa=0}^{\beta-\alpha-1} \left(H\left( \frac{\Delta}{\ell^{2(\alpha+\kappa)}}\right) - H\left( \frac{\Delta}{\ell^{2(\alpha+\kappa+1)}}\right) \right) D(t;\ell^{\alpha+\gamma+\kappa}) \ell^{2\alpha+\kappa} \\ 
& + \delta(\nu, 2\beta) D(t;\ell^{\beta+\gamma} ) \ell^{\beta} \sum_{\kappa=\beta-\alpha}^{\beta} \sum_{\substack{m^2 \mid \Delta \\ v_\ell(m)=\kappa}} h_w\left(\frac{\Delta}{m^2}\right) .} 
Rearranging terms according to the argument of $H$ we have 
\es{\label{cfl}
C(t,q,d)
 = & H\left(\frac{\Delta}{\ell^{2\alpha}}\right)D(t;\ell^{\alpha+\gamma})\ell^{2\alpha} \\ 
& - \sum_{\kappa=1}^{\beta-\alpha-1} H\left( \frac{\Delta}{\ell^{2(\alpha+\kappa)}}\right) \left( \ell^{2\alpha+\kappa-1}D(t;\ell^{\alpha+\gamma+\kappa -1}) - \ell^{2\alpha+\kappa}D(t;\ell^{\alpha+\gamma+\kappa }) \right) \\ 
& - H\left(\frac{\Delta}{\ell^{2\beta}} \right) D(t;\ell^{\beta+\gamma-1})\ell^{\alpha+\beta-1} \\ 
& + \delta(\nu, 2\beta)D(t;\ell^{\beta+\gamma} ) \ell^{\beta} \sum_{\kappa=\beta-\alpha}^{\beta} \sum_{\substack{m^2 \mid \Delta \\ v_\ell(m)=\kappa}} h_w\left(\frac{\Delta}{m^2}\right).}

We claim that the last line of \eqref{cfl} is equal to
\[
\delta(\nu, 2\beta)D(t;\ell^{\beta+\gamma} ) \ell^{\beta + \alpha}H\left( \frac{\Delta}{\ell^{2\beta}}\right). 
\] 
Suppose that $\Delta/\ell^{2\beta} \in \Z$ and $t\equiv dq+d^{-1} \pmod{\ell^{\beta+\gamma}}$, since otherwise the last line of \eqref{cfl} vanishes.  Then 
\est{ \Delta \equiv 
 (dq-d^{-1})^2 \pmod*{\ell^{\beta+\gamma}},} 
 and since $\beta < \gamma$ we have that $\Delta/\ell^{2\beta}$ is a square modulo $\ell$.  If $\ell=2$ then since $\ell \nmid q$ we have 
 \est{ \Delta \equiv (dq-d^{-1})^2 \pmod*{2^{\beta+\gamma+2}}} 
 so that $\Delta/2^{2\beta}$ is a square modulo $8$.   Then Lemma \ref{cor728} applies with $f=\ell^{\beta-\kappa},$ and $d=\Delta/\ell^{2\beta}$, which is a square modulo $\ell$ (resp.~8).  So in the case $\beta -\alpha \leq \kappa \leq \beta$ we have 
 \es{ 
 \sum_{\substack{m^2 \mid \Delta \\ v_{\ell}(m) = \kappa}} h_w \left( \frac{\Delta}{m^2}\right)  = \sum_{\substack{ m^2 \mid \Delta \\ v_\ell(m) = \kappa}} h_w \left( \frac{ \Delta}{ \ell^{2(\beta-\kappa)m^2}}\right)\varphi(\ell^{\beta-\kappa}).} 
Therefore the last line of \eqref{cfl} simplifies as 
\begin{multline}\label{cfl2} 
\delta(\nu, 2\beta) D(t;\ell^{\beta+\gamma} ) \ell^{\beta} \sum_{\kappa=\beta-\alpha}^{\beta} \sum_{\substack{m^2 \mid \Delta \\ v_\ell(m)=\kappa}} h_w\left(\frac{\Delta}{m^2}\right) \\ = \delta(\nu, 2\beta) D(t;\ell^{\beta+\gamma} ) \ell^\beta \sum_{\kappa = \beta-\alpha}^\beta \sum_{\substack{m^2 \mid \Delta \\ v_\ell(m)=\kappa}}  h_w\left( \frac{\Delta}{\ell^{2(\beta-\kappa)}m^2}\right) \varphi(\ell^{\beta-\kappa}).\end{multline}  
If this term does not vanish we have $\nu = 2\beta$ and so 
 \est{ \sum_{\substack{m^2 \mid (\Delta / \ell^{2\beta}) \\ v_\ell(m) = 0}} h_w\left( \frac{\Delta}{\ell^{2\beta} m^2}\right) =H\left( \frac{\Delta}{\ell^{2\beta}}\right)
 } 
by the definition of Hurwitz-Kronecker class numbers \eqref{HKcn}.  Therefore the expression in \eqref{cfl2} equals 
\es{
\delta(\nu, 2\beta) D(t;\ell^{\beta+\gamma} ) \ell^\beta H\left( \frac{\Delta}{\ell^{2\beta}}\right) \sum_{\kappa  = \beta-\alpha }^\beta \varphi(\ell^{\beta-\kappa})  = \delta(\nu, 2\beta) D(t;\ell^{\beta+\gamma} ) \ell^{\beta + \alpha}H\left( \frac{\Delta}{\ell^{2\beta}}\right),
}  
proving the claim above.

Putting this back into \eqref{cfl}, when $\beta<\gamma$ we have 
\es{ 
C(t,q,d)= &\ell^{2\alpha}H\left( \frac{\Delta}{\ell^{2\alpha}}\right)\delta_{\ell^{\alpha+\gamma}}(t, dq+d^{-1}) \\ & - \ell^{2\alpha} \sum_{\kappa=1}^{\beta-\alpha} H\left( \frac{\Delta}{\ell^{2(\alpha+\kappa)}}\right) \left( \ell^{\kappa-1} \delta_{\ell^{\alpha+\gamma+\kappa-1}}(t, dq+d^{-1}) - \ell^{\kappa} \delta_{\ell^{\alpha+\gamma+\kappa}}(t, dq+d^{-1})\right).}  Finally, by applying Lemma \ref{H} we conclude that \est{C(t,q,d) = 2\ell^{2\alpha}\sum_{j=0}^{\beta-\alpha} \varphi(\ell^{j}) H_{\ell^{\gamma},\ell^{\alpha+j},d}(t,q),} 
which in the case $\beta<\gamma$ matches exactly the expression in Proposition \ref{MP}.

\vspace{12pt}
\noindent{\bf \emph{The case $\beta = \gamma$ and $\nu < 2(\alpha+\gamma)$.}}
\vspace{12pt}

Recall the definition of $c_\kappa(t,q,d)$ from \eqref{ckdef}, so that we have 
\est{C(t,q,d) = \sum_{\kappa \geq 0} c_\kappa(t,q,d).}  
We split into three cases according to the value of $\kappa$: each of the ranges $\kappa < \gamma-\alpha$, $\gamma - \alpha \leq \kappa < 2\lfloor \frac{\nu + 1}{2} \rfloor - (\gamma - \alpha)$, and $\kappa \geq 2\lfloor \frac{\nu + 1}{2} \rfloor - (\gamma - \alpha)$ will be treated differently.  Specifically we write \est{ {\rm I} \eqdef \sum_{\kappa < \gamma-\alpha} c_\kappa(t,q,d),\,\,\,\,\, {\rm II} \eqdef \sum_{\kappa = \gamma - \alpha }^{2\lfloor \frac{\nu + 1}{2} \rfloor - (\gamma - \alpha)-1} c_\kappa(t,q,d),\,\,\,\,\, {\rm III} \eqdef \sum_{\kappa \geq 2\lfloor \frac{\nu + 1}{2} \rfloor - (\gamma - \alpha)} c_{\kappa}(t,q,d)} so that $C(t,q,d) = {\rm I} + {\rm II} + {\rm III}.$  Note that the second range of $\kappa$ above has been chosen to have an even number of terms.  

\vspace{12pt}
\noindent{\bf \emph{The case $\kappa   \geq 2\lfloor \frac{\nu + 1}{2} \rfloor - (\gamma - \alpha)\geq \gamma - \alpha$, i.e. the sum {\rm III}.}}
\vspace{12pt}
  
We have $\kappa \leq \lfloor \nu / 2 \rfloor < \alpha+\gamma$ by hypothesis, so $\min(\kappa, \alpha+ \gamma) = \kappa$.  Therefore the second case of Lemma \ref{weightscalc} (ii) applies and we have for $\beta = \gamma$, $\nu < 2(\alpha+\gamma)$ and $\kappa   \geq 2\lfloor \frac{\nu + 1}{2} \rfloor - (\gamma - \alpha)$ that 
\es{ 
c_\kappa(t,q,d) = D(t;\ell^{2\gamma})\sum_{\substack{ m^2 \mid \Delta \\ v_\ell(m) = \kappa}} h_w\left(\frac{\Delta}{m^2}\right) S(\Delta^*,\ell^{\alpha+\gamma + \kappa}),
} 
where we note that if $\ell=2$ then $\Delta^* \in \Z$ otherwise $c_\kappa(t,q,d)=0$.  A repeated application of Lemma \ref{sqrts} shows that when $\kappa   \geq 2\lfloor \frac{\nu + 1}{2} \rfloor - (\gamma - \alpha)$ we have that 
\es{\label{Sevaluation2}
S(\Delta^*,\ell^{\alpha+\gamma+\kappa})=\begin{cases} \ell^{\nu/2} & \nu \text{ even, } \kappa=\nu-(\alpha+\gamma) \text{ and } \Delta/\ell^\nu \equiv 0,1\pmod 4 \\ \ell^{\nu/2} \left(1+\legen{\Delta/\ell^{\nu}}{\ell} \right) & \nu \text{ even, } \kappa>\nu-(\alpha+\gamma) \text{ and } \Delta/\ell^\nu \equiv 0,1\pmod 4 \\ 0 & \nu \text{ odd, or } \Delta/\ell^\nu \equiv 2,3\pmod 4.
\end{cases}
} 
If $t \equiv dq+d^{-1} \pmod {\ell^{\alpha+\gamma+\kappa}}$ then Lemma \ref{resclbound} (i) implies that $\nu\geq 2\gamma$.  Therefore for $\kappa  \geq 2\lfloor \frac{\nu + 1}{2} \rfloor - (\gamma - \alpha)$ we have that: 
\begin{enumerate}
\item If $ \nu \text{ even, } \kappa=\nu-(\alpha+\gamma), 2\gamma \leq \nu, \text{ and } \Delta/\ell^\nu \equiv 0,1\pmod*4 $ then \est{ c_\kappa(t,q,d)=
D(t;\ell^{2\gamma}) \ell^{\nu/2} \sum_{\substack{ m^2 \mid \Delta \\ v_\ell(m) = \kappa}} h_w\left(\frac{\Delta}{m^2}\right).  } 
\item If $  \nu \text{ even, } \kappa>\nu-(\alpha+\gamma), 2\gamma \leq \nu, \text{ and } \Delta/\ell^\nu \equiv 0,1\pmod*4 $ then \est{ c_\kappa(t,q,d)=
D(t;\ell^{2\gamma}) \ell^{\nu/2}   \left(1+\legen{\Delta/\ell^{\nu}}{\ell} \right) \sum_{\substack{ m^2 \mid \Delta \\ v_\ell(m) = \kappa}} h_w\left(\frac{\Delta}{m^2}\right). }  
\item If $ \nu \text{ odd, or }2\gamma < \nu,  \text{ or } \Delta/\ell^\nu \equiv 2,3\pmod*4 $ then $c_\kappa(t,q,d) = 0$.
\end{enumerate}

If $\nu$ is even, $\kappa = \nu - (\alpha+ \gamma)$, $2 \gamma \leq \nu$ and $\Delta/\ell^\nu \equiv 0,1\pmod*4$ then we may apply Lemma \ref{cor728} with $d = \Delta/ (\ell^{\nu + 2\kappa} m^2)$ and $f = \ell^{\nu/2-\kappa}$ to find 
\es{\label{b=grange31} c_\kappa(t,q,d) = & D(t;\ell^{2\gamma}) \ell^{\nu-\kappa} \sum_{\substack{ m^2 \mid \Delta \\ v_\ell(m) = \kappa}} h_w\left(\frac{\Delta}{\ell^{\nu - 2\kappa }m^2}\right) \left( 1- \legen{\frac{\Delta}{\ell^{\nu - 2\kappa }m^2}}{\ell} \frac{1}{\ell}\right) \\ 
 = &D(t;\ell^{2\gamma}) \ell^{\nu-\kappa} \sum_{\substack{ m^2 \mid (\Delta/\ell^\nu) \\ v_\ell(m) = 0}} h_w\left(\frac{\Delta}{\ell^{\nu  }m^2}\right) \left( 1- \legen{\frac{\Delta}{\ell^{\nu  }m^2}}{\ell} \frac{1}{\ell}\right)  \\
= & D(t;\ell^{2\gamma}) \ell^{\nu-\kappa}\left( 1- \legen{\Delta/\ell^{\nu  }}{\ell} \frac{1}{\ell}\right)  \sum_{\substack{ m^2 \mid (\Delta/\ell^\nu) \\ v_\ell(m) = 0}} h_w\left(\frac{\Delta}{\ell^{\nu  }m^2}\right), } 
where the second equality follows from the change of variables $m \to m\ell^\kappa$ and the third equality follows from the fact that $m \not\equiv 0 \pmod{\ell}$, so $\legen{m^2}{\ell}=1$.  

Similarly, in the case that $\nu$ is even, $ \nu - (\alpha+ \gamma) < \kappa < \nu/2$, $2 \gamma \leq \nu$ and $\Delta/\ell^\nu \equiv 0,1\pmod*4$ we have that 
\es{\label{b=grange32}
c_\kappa(t,q,d) =D(t;\ell^{2\gamma}) \ell^{\nu-\kappa}\left( 1- \legen{\Delta/\ell^{\nu  }}{\ell} \frac{1}{\ell}\right) \left(1+\legen{\Delta/\ell^{\nu}}{\ell} \right) \sum_{\substack{ m^2 \mid (\Delta/\ell^\nu) \\ v_\ell(m) = 0}} h_w\left(\frac{\Delta}{\ell^{\nu  }m^2}\right). 
} 
If $\nu$ is even, $\kappa = \nu /2$, $2 \gamma \leq \nu$ and $\Delta/\ell^\nu \equiv 0,1\pmod*4$ then we do not apply Lemma \ref{cor728} to $c_\kappa(t,q,d)$, but change variables $m \to m \ell^\kappa$.  We have 
\es{\label{b=grange33} 
c_\kappa(t,q,d) =D(t;\ell^{2\gamma}) \ell^{\nu/2}   \left(1+\legen{\Delta/\ell^{\nu}}{\ell} \right) \sum_{\substack{ m^2 \mid (\Delta/\ell^\nu) \\ v_\ell(m) = 0}} h_w\left(\frac{\Delta}{\ell^\nu m^2}\right) .
}

Observe that the following is a telescoping sum: \begin{multline}\label{telescope} \psi(\ell^{\alpha+ \gamma}) = \ell^{\alpha+ \gamma}\left( 1- \legen{\Delta/\ell^{\nu  }}{\ell} \frac{1}{\ell}\right)  + \sum_{\kappa = \nu-(\alpha+\gamma)+1}^{\nu/2-1} \ell^{\nu-\kappa}\left( 1- \legen{\Delta/\ell^{\nu  }}{\ell} \frac{1}{\ell}\right) \left(1+\legen{\Delta/\ell^{\nu}}{\ell} \right) \\ + \ell^{\nu/2}   \left(1+\legen{\Delta/\ell^{\nu}}{\ell} \right). \end{multline} Thus, taking the sum over $\kappa \geq 2\lfloor \frac{\nu + 1}{2} \rfloor - (\gamma - \alpha)$ of the expressions \eqref{b=grange31}, \eqref{b=grange32}, and \eqref{b=grange33}, we have by \eqref{telescope} that if $\beta = \gamma$ and $\nu <2(\alpha+\gamma)$, then 
\est{ {\rm III} = 
\begin{cases}D(t;\ell^{2\gamma}) \psi(\ell^{\alpha+ \gamma}) \sum_{\substack{ m^2 \mid (\Delta/\ell^\nu) \\ v_\ell(m) = 0}} h_w\left(\frac{\Delta}{\ell^\nu m^2}\right) & \text{ if } \nu \text{ is even and }\nu \geq 2\gamma \\
0 & \text{ if } \nu \text{ is odd or } \nu < 2\gamma. \end{cases}
} 
Note that the condition $\Delta/\ell^\nu\equiv 0,1 \pmod*4$ is now implicit in the definition of $h_w$.   

\vspace{12pt}
\noindent{\bf \emph{The case $\gamma - \alpha \leq \kappa < 2 \lfloor \frac{\nu + 1}{2}\rfloor - (\alpha + \gamma)$, i.e. the sum {\rm II}.}}
\vspace{12pt}
  
We have $\kappa \leq \lfloor \nu / 2 \rfloor < \alpha+\gamma$ by hypothesis, so $\min(\kappa, \alpha+ \gamma) = \kappa$.  Therefore the second case of Lemma \ref{weightscalc} (ii) applies, so for $\beta = \gamma$, $\nu < 2(\alpha+\gamma)$ and $\gamma - \alpha \leq \kappa < 2 \lfloor \frac{\nu + 1}{2}\rfloor - (\alpha + \gamma)$ we have 
\es{ 
c_\kappa(t,q,d) = D(t;\ell^{2\gamma})\sum_{\substack{ m^2 \mid \Delta \\ v_\ell(m) = \kappa}} h_w\left(\frac{\Delta}{m^2}\right) S(\Delta^*,\ell^{\alpha+\gamma + \kappa}).
}  
Now, for $\gamma - \alpha \leq \kappa < 2 \lfloor \frac{\nu + 1}{2}\rfloor - (\alpha + \gamma)$ we have by Lemma \ref{sqrts} that 
\es{\label{Sevaluation1}S(\Delta^*,\ell^{\alpha+\gamma+\kappa}) = \begin{cases}\ell^{\lfloor \frac{\alpha+\gamma+\kappa}{2}\rfloor} & \text{ if }\Delta/\ell^{2\lfloor \frac{\alpha+\gamma+\kappa}{2}\rfloor}\equiv 0,1 \pmod*4 \\ 0 & \text{ if }\Delta/\ell^{2\lfloor \frac{\alpha+\gamma+\kappa}{2}\rfloor}\equiv 2,3 \pmod*4. 
\end{cases} 
} 
If $t \equiv dq+d^{-1} \pmod {\ell^{\alpha+\gamma+\kappa}}$ then Lemma \ref{resclbound} (i) implies that $\nu\geq 2\gamma$.  Therefore if $\Delta/\ell^{2\lfloor \frac{\alpha+\gamma+\kappa}{2}\rfloor}\equiv 0,1 \pmod*4$ and $\nu \geq 2\gamma$ we have for $\gamma - \alpha \leq \kappa < 2 \lfloor \frac{\nu + 1}{2}\rfloor - (\alpha + \gamma)$ that \es{c_\kappa(t,q,d) = D(t;\ell^{2\gamma})\sum_{\substack{ m^2 \mid \Delta \\ v_\ell(m) = \kappa}} h_w\left(\frac{\Delta}{m^2}\right) \ell^{\lfloor \frac{\alpha+\gamma+\kappa}{2}\rfloor},} and if $\nu < 2\gamma$ or $\Delta/\ell^{2\lfloor \frac{\alpha+\gamma+\kappa}{2}\rfloor}\equiv 2,3 \pmod*4$ then $c_\kappa(t,q,d)=0$.

We simplify $c_\kappa(t,q,d)$ in two cases according to whether $\alpha+\gamma + \kappa$ is even or odd.  First, assume that $\alpha+\gamma + \kappa$ is even.  Then the assumption $\kappa < 2 \lfloor \frac{\nu + 1}{2}\rfloor - (\alpha + \gamma)$ implies that $\nu - (\alpha + \gamma +\kappa)>0$.  If $\Delta/\ell^{2\lfloor \frac{\alpha+\gamma+\kappa}{2}\rfloor}\equiv 0,1 \pmod*4$ we apply Lemma \ref{cor728} with $d = \frac{\Delta}{\ell^{\alpha+ \gamma -\kappa} m^2}$ and $f = \ell^{\alpha +\gamma - (\frac{\alpha+\gamma+\kappa}{2})}.$ Therefore \es{\label{cktd1}c_\kappa(t,q,d) = D(t;\ell^{2\gamma})\ell^{\alpha+\gamma} \sum_{\substack{ m^2 \mid \Delta \\ v_\ell(m) = \kappa}} h_w\left(\frac{\Delta}{\ell^{\alpha+\gamma - \kappa}m^2}\right).}  

Now suppose $\alpha+\gamma+\kappa$ is odd.  The assumption $\kappa < 2 \lfloor \frac{\nu + 1}{2}\rfloor - (\alpha + \gamma)$ implies that $\nu - (\alpha + \gamma +\kappa-1)>0$. If $\Delta/\ell^{2\lfloor \frac{\alpha+\gamma+\kappa}{2}\rfloor}\equiv 0,1 \pmod*4$ we apply Lemma \ref{cor728} with $d = \frac{\Delta}{\ell^{\alpha+ \gamma -\kappa-1} m^2}$ and $f = \ell^{\alpha +\gamma - (\frac{\alpha+\gamma+\kappa+1}{2})}.$ Thus 
\es{\label{cktd2}c_\kappa(t,q,d) = D(t;\ell^{2\gamma})\ell^{\alpha+\gamma-1} \sum_{\substack{ m^2 \mid \Delta \\ v_\ell(m) = \kappa}} h_w\left(\frac{\Delta}{\ell^{\alpha+\gamma - \kappa-1}m^2}\right).}
Putting the even and odd cases for $\alpha+\gamma +\kappa$ together, we see that if $\alpha+ \gamma+\kappa \equiv 0 \pmod*2$ and $2\gamma \leq \nu$, then 
\es{\label{cktd3} c_\kappa(t,q,d) + c_{\kappa+1}(t,q,d)  = D(t;\ell^{2\gamma}) \psi(\ell^{\alpha+\gamma})  \sum_{\substack{ m^2 \mid \Delta \\ v_\ell(m) = \kappa}} h_w\left(\frac{\Delta}{\ell^{\alpha+\gamma - \kappa}m^2}\right),}
and if $\nu < 2\gamma$ then  $c_\kappa(t,q,d)=0$.  

Pairing up the terms in the sum ${\rm II}$ two at a time, we have that if $2 \gamma \leq \nu$
\es{\label{cktd4} {\rm II} & = \sum_{\kappa = \gamma - \alpha }^{2\lfloor \frac{\nu + 1}{2} \rfloor - (\gamma - \alpha)-1} c_\kappa(t,q,d) \\ 
& = \sum_{j=0}^{\lfloor \frac{\nu+ 1}{2}\rfloor  - \gamma - 1} \left( c_{\gamma-\alpha + 2j} (t,q,d) + c_{\gamma - \alpha +2j +1}(t,q,d)\right) \\ 
& = D(t;\ell^{2\gamma}) \psi(\ell^{\alpha+\gamma}) \sum_{j=0}^{\lfloor \frac{\nu+ 1}{2}\rfloor  - \gamma - 1}     \sum_{\substack{ m^2 \mid \Delta \\ v_\ell(m) = \gamma - \alpha + 2j}} h_w\left(\frac{\Delta}{\ell^{\alpha+\gamma - (\gamma - \alpha + 2j)}m^2}\right) \\
& = D(t;\ell^{2\gamma}) \psi(\ell^{\alpha+\gamma}) \sum_{j=0}^{\lfloor \frac{\nu+ 1}{2}\rfloor  - \gamma - 1}     \sum_{\substack{ m^2 \mid (\Delta/\ell^{2\gamma}) \\ v_\ell(m) = j}} h_w\left(\frac{\Delta}{\ell^{2\gamma}m^2}\right), }
and if $2\gamma>\nu$ then ${\rm II} = 0$.  The definition of $H$ from \eqref{HKcn} implies that if $\nu <2(\alpha+ \gamma)$, then 
\es{\label{IIandIII}{\rm II } + {\rm III} = 
\begin{cases} D(t;\ell^{2\gamma}) \psi(\ell^{\alpha+\gamma}) H\left(\frac{\Delta}{\ell^{2\gamma}}\right) & \text{ if } 2\gamma \leq \nu  \\ 0 & \text{ if } \nu < 2\gamma .
\end{cases}}

\vspace{12pt}
\noindent{\bf \emph{The case $\kappa < \gamma - \alpha$, i.e. the sum {\rm I}.}}
\vspace{12pt}
  
These terms are treated similarly to the ($\beta<\gamma$ and $\kappa < \beta-\alpha$) case from before but there are several details that change, so we repeat the argument.  

We have that $\kappa = \min(\kappa, \alpha + \gamma)$ because $\kappa < \gamma-\alpha$.  The first case of Lemma \ref{weightscalc} (ii) holds, which implies
\es{c_\kappa(t,q,d) =D(t;\ell^{\alpha+\gamma+ \kappa}) \ell^{\alpha+\kappa} \sum_{\substack{m^2 \mid \Delta \\ v_\ell(m) = \kappa}} h_w\left( \frac{\Delta}{m^2}\right).}  

We want to apply Lemma \ref{cor728} to this expression, but need to check a certain hypothesis.

\begin{lemma}\label{claim2}
\begin{enumerate} 
\item If $t\equiv dq+d^{-1} \pmod*{\ell^{\alpha+\gamma+\kappa}}$ and $\kappa <\gamma-\alpha$ then $v_\ell(\Delta/(\ell^{2\alpha}m^2)) \geq 1$ and $\Delta/(\ell^{2\alpha}m^2) \equiv 0,1 \pmod*4$.  

\item If $\kappa < \gamma-\alpha -1$ then $v_\ell(\Delta/(\ell^{2\alpha}m^2)) \geq 2$ and $\Delta/(\ell^{2\alpha+2}m^2) \equiv 0,1 \pmod*4$.
\end{enumerate}
\end{lemma}
\begin{proof}
\begin{enumerate}
\item Suppose $\nu \geq 2\gamma$.  Then we have directly that 
\est{ \nu - 2\alpha -2\kappa \geq \nu - 2\gamma +2 \geq 2.} 
Now suppose $\nu < 2\gamma$.  By Lemma \ref{resclbound} (i) we have $\alpha+ \gamma +\kappa \leq \nu$, from which it follows that 
\est{ \nu - 2\alpha - 2\kappa \geq \gamma -\alpha - \kappa \geq 1.}  
If $\ell\neq 2$ then $\ell^2 \equiv 1 \pmod*4$, and so $\Delta/(\ell^{2\alpha} m^2) \equiv 0,1\pmod*4$.  

Suppose $\ell=2$.  If $\nu \geq 2\gamma$ we already have $\nu -2\gamma -2\kappa \geq 2$ so $\Delta / (2^{2\alpha} m^2) \equiv 0 \pmod*4$.  If $\nu < 2\gamma$ then by Lemma \ref{resclbound} (ii) and $\kappa <\gamma-\alpha$ we have that $\nu-2\alpha -2\kappa\geq 2$, and so $\Delta / (2^{2\alpha} m^2) \equiv 0 \pmod*4$ as well.

\item We now consider the stronger condition that $\kappa \leq \gamma - \alpha- 2$.  This set of $\kappa$ is non-empty only if $\gamma \geq 2$ so we assume this.  If $\nu\geq 2\gamma$ we have directly \est{\nu -2\alpha -2\kappa \geq \nu -2\gamma + 2 \geq 2} as above.  If $\nu< 2\gamma$ then we have 
\est{\nu -2\alpha -2\kappa \geq \gamma - \alpha - \kappa \geq 2} 
by Lemma \ref{resclbound} (i).  If $\ell \neq 2$ then $\ell^2 \equiv 1 \pmod*4$ so we have $\Delta / (\ell^{2(\alpha + 1)} m^2) \equiv 0,1\pmod*4$.  

If $\ell =2$ then we write $t\equiv dq+d^{-1} + \epsilon 2^{\alpha + \gamma +\kappa} \pmod*{2^{\alpha + \gamma + \kappa +2}}$ with $\epsilon \in \Z/4\Z$.  Then because $\gamma \geq 2$ and $q$ is odd we have that 
\est{ \Delta \equiv (dq-d^{-1})^2 \pmod*{2^{\alpha + \gamma + \kappa +2}}.} Since \est{\alpha + \gamma + \kappa + 2 -2(\alpha+ \kappa +1) \geq \gamma - \alpha -\kappa \geq 2} 
we have that $\Delta / 2^{2(\alpha + \kappa +1 )} $ is a square modulo $4$, as was to be shown.
\end{enumerate}
\end{proof}

By Lemma \ref{claim2} we may apply Lemma \ref{cor728} to $c_\kappa(t,q,d)$ with $d = \frac{\Delta}{\ell^{2\alpha}m^2}$ and $f = \ell^\alpha$.  For $k < \gamma-\alpha$ this implies
\es{\label{b=gksmall}
c_\kappa(t,q,d) = D(t;\ell^{\alpha+\gamma+ \kappa}) \ell^{2\alpha+\kappa} \sum_{\substack{m^2 \mid \Delta \\ v_\ell(m) = \kappa}} h_w\left( \frac{\Delta}{\ell^{2\alpha} m^2}\right).}    
We apply the expression \eqref{HWtrick} to \eqref{b=gksmall}.  If $\kappa < \gamma-\alpha-1$ then the second part of Lemma \ref{claim2} implies that 
\es{\label{ck1}c_\kappa(t,q,d) = D(t;\ell^{\alpha+\gamma+ \kappa}) \ell^{2\alpha+\kappa} \left( H \left( \frac{\Delta}{\ell^{2\alpha+2\kappa}}\right) - H\left(\frac{\Delta}{\ell^{2\alpha+2\kappa+2}}\right)\right).}
It remains for us to consider the case $\kappa = \gamma-\alpha-1$. 

\begin{lemma}\label{claim3}
Suppose that $\beta = \gamma$.   
We have that $t\equiv dq+d^{-1} \pmod*{\ell^{2\gamma}}$ if and only if $t \equiv dq+d^{-1}\pmod*{\ell^{2\gamma-1}}$, $\nu \geq 2\gamma$, and $\Delta/\ell^{2\gamma} \equiv 0,1 \pmod*4$.  
\end{lemma}
\begin{proof}
We first prove the ``only if'' direction.  Suppose that $t\equiv dq+d^{-1} \pmod*{\ell^{2\gamma}}$.  Clearly $t\equiv dq+d^{-1} \pmod*{\ell^{2\gamma-1}}$, so Lemma \ref{resclbound} (i) directly shows that $\nu \geq 2\gamma$.  If $\ell \neq 2$ then $\Delta/\ell^{2(\alpha+\kappa+1)} \equiv 0,1 \pmod*4$ is automatically satisfied.    
If $\ell=2$ one may calculate from $t\equiv dq+d^{-1} \pmod*{2^{2\gamma}}$ and the fact that $q$ is odd that \est{\Delta \equiv (dq-d^{-1})^2 \pmod*{2^{2\gamma + 2}},} which by the assumption $\beta=\gamma$ implies that $\Delta/2^{2\gamma}$ is a square modulo $4$.

We now prove the ``if'' direction.  We have that $t \equiv dq+d^{-1}\pmod*{\ell^{2\gamma-1}}$, which we write as $t \equiv dq+d^{-1} +\epsilon \ell^{2\gamma-1} \pmod*{\ell^{2\gamma}}$ with $\epsilon \in \Z/\ell\Z$.  Our goal is to show that $\epsilon =0$.  

We have that 
\est{ \Delta \equiv (dq-d^{-1})^2 + 2 \epsilon(dq+d^{-1})\ell^{2\gamma-1} \equiv 0 \pmod*{\ell^{2\gamma}}.}  
Then $\beta=\gamma$ implies that $2 \epsilon(dq+d^{-1}) \equiv 0 \pmod*\ell$.  We have $dq+d^{-1} \equiv 2d^{-1} \pmod*\ell$, so if $\ell \neq 2$ we must have that $\epsilon =0$ as was to be shown.  

If $\ell= 2$ then assume $t \equiv dq+d^{-1} +\epsilon 2^{2\gamma-1} \pmod*{2^{2\gamma+2}}$ where $\epsilon \in \Z/8\Z$.  Our goal is to show $\epsilon$ is even. If $\gamma\geq 2$ we have by the hypothesis $\nu \geq 2\gamma$ that 
\est{\Delta / 2^{2\gamma} \equiv \epsilon (dq+d^{-1}) \equiv 0,1 \pmod*4.} Since $q$ is odd $dq+d^{-1} \equiv 2 \pmod*4,\ \epsilon$ must be even and hence $t \equiv dq+d^{-1} \pmod*{2^{2\gamma}}.$  If $\gamma=1$ then we have that \est{\Delta/2^{2\gamma} \equiv \epsilon (dq+d^{-1} +\epsilon)\equiv 0,1\pmod*4,} 
from which one checks using $dq+d^{-1} \equiv 2 \pmod*4$ that $\epsilon$ must be even.
\end{proof}

By Lemma \ref{claim3} and \eqref{HWtrick}, when $\kappa = \gamma -\alpha-1$ we have that 
\es{\label{ck2}c_\kappa(t,q,d) = D(t;\ell^{\alpha+\gamma+ \kappa}) \ell^{2\alpha+\kappa}  H \left( \frac{\Delta}{\ell^{2\alpha+2\kappa}}\right) - D(t;\ell^{\alpha+\gamma+ \kappa+1}) \ell^{2\alpha+\kappa} H\left(\frac{\Delta}{\ell^{2\alpha+2\kappa+2}}\right).}

Now putting together the formulas \eqref{ck1} and \eqref{ck2} and rearranging according to the argument of $H$ we have that 
\est{{\rm I} = & \sum_{\kappa < \gamma-\alpha} c_\kappa(t,q,d) \\ 
= & H\left(\frac{\Delta}{\ell^{2\alpha}}\right) D(t;\ell^{\alpha+\gamma})\ell^{2\alpha} \\ & - \sum_{\kappa=1}^{\gamma-\alpha-1} H\left( \frac{\Delta}{\ell^{2(\alpha+\kappa)}}\right) \left( \ell^{2\alpha+\kappa-1}D(t;\ell^{\alpha+\gamma+\kappa-1}) - \ell^{2\alpha+\kappa}D(t;\ell^{\alpha+\gamma+\kappa}) \right) \\ & - H\left(\frac{\Delta}{\ell^{2\gamma}}\right) D(t;\ell^{2\gamma})\ell^{\gamma+\alpha-1}
.}
Adding this together with the result for ${\rm II} + {\rm III}$ from \eqref{IIandIII} we find that 
\es{\label{bgctd} C(t,q,d) = & {\rm I} + {\rm II} + {\rm III} \\ 
= & H\left(\frac{\Delta}{\ell^{2\alpha}}\right)D(t;\ell^{\alpha+\gamma})\ell^{2\alpha} \\ & - \sum_{\kappa=1}^{\gamma-\alpha-1} H\left( \frac{\Delta}{\ell^{2(\alpha+\kappa)}}\right) \left( \ell^{2\alpha+\kappa-1}D(t;\ell^{\alpha+\gamma+\kappa-1}) - \ell^{2\alpha+\kappa}D(t;\ell^{\alpha+\gamma+\kappa}) \right) \\ & + H\left(\frac{\Delta}{\ell^{2\gamma}}\right) D(t;\ell^{2\gamma})\ell^{\gamma+\alpha}
.}
By Lemma \ref{H} this matches exactly the claimed formula from Proposition \ref{MP}.

\vspace{12pt}
\noindent {\bf \emph{The case $\beta = \gamma$ and $\nu \geq 2(\alpha+\gamma)$.  }}
\vspace{12pt}

Recall the definition of $c_\kappa(t,q,d)$ from \eqref{ckdef}.  We split into three cases according to the value of $\kappa$: each of the ranges $\kappa < \gamma-\alpha$, $\gamma - \alpha \leq \kappa < \gamma + \alpha$, and $\kappa \geq \gamma + \alpha$ will be treated differently.  Specifically we write \est{ {\rm I} \eqdef \sum_{\kappa < \gamma-\alpha} c_\kappa(t,q,d),\,\,\,\,\, {\rm II} \eqdef \sum_{\kappa = \gamma - \alpha }^{\gamma+ \alpha-1} c_\kappa(t,q,d),\,\,\,\,\, {\rm III} \eqdef \sum_{\kappa \geq \gamma +\alpha} c_{\kappa}(t,d)} so that $C(t,q,d) = {\rm I} + {\rm II} + {\rm III}.$  Note that the second range of $\kappa$ above has been chosen to have an even number of terms.

\vspace{12pt}
\noindent{\bf \emph{The case $\kappa \geq \gamma + \alpha$, i.e. the sum {\rm III}.}}
\vspace{12pt}
  
In this case we have $\gamma -\alpha \leq \gamma + \alpha = \min(\kappa,\alpha + \gamma)$ so the second case of Lemma \ref{weightscalc} (ii) applies. We have that
\es{ 
c_\kappa(t,q,d) = D(t;\ell^{2\gamma}) \sum_{\substack{m^2 \mid \Delta \\ v_\ell(m) = \kappa}} h_w \left(\frac{\Delta}{m^2}\right) \frac{\psi(\ell^{\alpha + \gamma})}{\ell^{\alpha+\gamma}} S\left(\Delta^*,\ell^{2(\alpha+\gamma)}\right).
} 
By assumption $\nu \geq 2(\alpha+ \gamma)$, which implies that that $\Delta^* \equiv 0 \pmod*{\ell^{2(\alpha+\gamma)}}$.  For $\kappa \geq \gamma + \alpha$, Lemma \ref{sqrts} implies that
\est{
S\left(\Delta^*,\ell^{2(\alpha+\gamma)}\right) = S\left(0,\ell^{2(\alpha+\gamma)}\right) =\ell^{\alpha+\gamma}.
}   
Thus for $\kappa \geq \gamma+\alpha,\ \beta=\gamma$, and $\nu \geq 2(\alpha+\gamma)$ we have 
\es{
c_\kappa(t,q,d) =  D(t;\ell^{2\gamma}) \psi(\ell^{\alpha+\gamma}) \sum_{\substack{m^2 \mid \Delta \\ v_\ell(m) = \kappa}} h_w \left(\frac{\Delta}{m^2}\right),
} 
and therefore,
\es{{\rm III} = \sum_{\kappa \geq \alpha + \gamma} c_\kappa(t,q,d) = \sum_{j \geq \alpha} c_{\gamma+j}(t,q,d) =D(t;\ell^{2\gamma}) \psi(\ell^{\alpha+\gamma}) \sum_{j \geq \alpha} \sum_{\substack{m^2 \mid (\Delta/\ell^{2\gamma}) \\ v_\ell(m) = j}} h_w \left(\frac{\Delta}{\ell^{2\gamma} m^2}\right) .
} 

\vspace{12pt}
\noindent{\bf \emph{The case $\gamma -\alpha \leq \kappa < \gamma + \alpha$ i.e. the sum {\rm II}.}}
\vspace{12pt}
  
In this case we have $\kappa = \min (\kappa, \alpha + \gamma)$ so the second case of Lemma \ref{weightscalc} (ii) applies. This implies
\es{ c_\kappa(t,q,d) = D(t;\ell^{2\gamma})\sum_{\substack{ m^2 \mid \Delta \\ v_\ell(m) = \kappa}} h_w\left(\frac{\Delta}{m^2}\right) S(\Delta^*,\ell^{\alpha+\gamma + \kappa}).}  
By our assumption that $\nu\geq 2(\alpha+\gamma)$, we have $ 2\gamma \leq \alpha + \gamma + \kappa < 2(\alpha+ \gamma) \leq \nu$ so exactly as in \eqref{Sevaluation1} we have 
\es{S(\Delta^*,\ell^{\alpha+\gamma+\kappa}) = \begin{cases}\ell^{\lfloor \frac{\alpha+\gamma+\kappa}{2}\rfloor} & \text{ if }\Delta/\ell^{2\lfloor \frac{\alpha+\gamma+\kappa}{2}\rfloor}\equiv 0,1 \pmod*4 \\ 0 & \text{ if }\Delta/\ell^{2\lfloor \frac{\alpha+\gamma+\kappa}{2}\rfloor}\equiv 2,3 \pmod*4 .\end{cases} } 
If $\Delta/\ell^{2\lfloor \frac{\alpha+\gamma+\kappa}{2}\rfloor}\equiv 0,1 \pmod*4$  then we have 
\es{c_\kappa(t,q,d) = D(t;\ell^{2\gamma})\sum_{\substack{ m^2 \mid \Delta \\ v_\ell(m) = \kappa}} h_w\left(\frac{\Delta}{m^2}\right) \ell^{\lfloor \frac{\alpha+\gamma+\kappa}{2}\rfloor},} 
and again $\nu \geq 2(\alpha+\gamma)$ and $\kappa < \alpha +\gamma$ give $\nu -(\alpha + \gamma + \kappa) > 0 $. Therefore, if $\Delta/\ell^{2\lfloor \frac{\alpha+\gamma+\kappa}{2}\rfloor}\equiv 0,1 \pmod*4$, then we may apply Lemma \ref{cor728} and pair the terms exactly as in \eqref{cktd1},  \eqref{cktd2},  \eqref{cktd3}, and  \eqref{cktd4}. In this case we have
\est{{\rm II} =  \sum_{\kappa = \gamma-\alpha }^{\gamma+\alpha -1} c_\kappa(t,q,d) = D(t;\ell^{2\gamma}) \psi(\ell^{\alpha + \gamma}) \sum_{j\leq \alpha -1}\sum_{\substack{ m^2 \mid (\Delta/ \ell^{2\gamma}) \\ v_\ell(m) = j}} h_w\left(\frac{\Delta}{\ell^{2\gamma}m^2}\right).}  
Therefore by the definition of $H$ from \eqref{HKcn}, in the case ($\beta = \gamma$ and $\nu \geq 2(\alpha+ \gamma)$) we have  
\est{ {\rm II} + {\rm III} =D(t;\ell^{2\gamma}) \psi(\ell^{\alpha + \gamma}) H\left( \frac{\Delta}{\ell^{2\gamma}}\right).}

\vspace{12pt}
\noindent{\bf \emph{The case $\kappa < \gamma - \alpha$, i.e. the sum {\rm I}.}}
\vspace{12pt}
  
When we treated the case ($\beta = \gamma$, $\nu < 2(\alpha + \gamma)$ and $\kappa < \gamma-\alpha$) we did not use the assumption $\nu <2(\alpha + \gamma)$ at all.  Therefore the same proof goes through verbatim in the present case ($\beta = \gamma$, $\nu \geq 2(\alpha + \gamma)$ and $\kappa < \gamma-\alpha$).  We take $C(t,q,d) = {\rm I} + {\rm II} + {\rm III}$ and conclude that \eqref{bgctd} holds whenever $\beta = \gamma$. Thus, by Lemma \ref{H} we conclude Proposition \ref{MP} in all cases. 

\subsection{The General Case of Proposition \ref{MP}}\label{MPGeneral}

We now discuss the computation of $T^{(e)}(t)$ in the case of general levels $M$ and $N$ without giving full details.  Recall that $(N,q)=1$, $d\in (\Z/N\Z)^\times$, and $L=(d^2q-1,N)$. Let $\ell \mid MN$ be a prime, $\alpha_\ell = v_\ell(M),\ \beta_\ell= v_\ell(L)$ and $\gamma_\ell = v_\ell(N)$.   Let 
\est{W(d) \eqdef \ssum_{c \in (\Z/MN\Z)^\times} \delta_N(c,d^{-1})} 
and for positive integers $K$ such that $K^2 \mid \Delta$, let 
\es{C_{K,N,M}(t,q,d) \eqdef \sum_{m^2 \mid (\Delta/K^2)} h_w \left( \frac{\Delta}{K^2 m^2}\right) \frac{\psi(MN)}{\psi(MN/(MN,m))} W(d).} 
Note that with this definition we have 
\est{ C_{K,1,1}(t,q,d) = H\left( \frac{\Delta}{K^2}\right)} and \est{C_{1,\ell^\gamma,\ell^\alpha }(t,q,d) = C(t,q,d).} 
Following the same steps as \eqref{Tet1} through \eqref{orthogonalityresult} we find that 
\est{T^{(e)}(t) = \frac{1}{2} \left( C_{1,N,M}(t,q,d) + (-1)^k C_{1,N,M}(t,q,-d)\right),} 
so it suffices to compute $ C_{1,N,M}(t,q,d)$.  
\begin{proposition}\label{MP2}
We have that \begin{multline*} C_{K,N,M}(t,q,d)\\
 = \frac{\psi(\ell^{2\gamma_\ell})}{\psi(\ell^{2(\gamma_\ell-\alpha_\ell)})} \sum_{j=0}^{\beta_\ell -\alpha_\ell} \frac{\varphi(\ell^{2j})\varphi(\ell^{\gamma_\ell-\alpha_\ell-j})}{\varphi(\ell^{\gamma_\ell-\alpha_\ell})} \Bigg( C_{\ell^{\alpha_\ell}K,N/\ell^{\gamma_\ell},M/\ell^{\alpha_\ell}} (t,q,d) \delta_{\ell^{\alpha_\ell}}(d^2q,1) D(t;\ell^{\alpha_\ell+\gamma_\ell}) \\  - \sum_{k=1}^{\gamma_\ell-\alpha_\ell -1}  C_{\ell^{\alpha_\ell+k}K,N/\ell^{\gamma_\ell},M/\ell^{\alpha_\ell}} (t,q,d)  \delta_{\ell^{\alpha_\ell+k}}(d^2q,1) \left(D(t;\ell^{\alpha_\ell+\gamma_\ell+k-1}) - D(t;\ell^{\alpha_\ell+\gamma_\ell+k}) \right) \Bigg).\end{multline*}
\end{proposition}

Proposition \ref{MP} is the special case of Proposition \ref{MP2} with $K=1,$ $N= \ell^\gamma$ and $M = \ell^\alpha$.  The proof of Proposition \ref{MP2} is nearly the same as the proof of Proposition \ref{MP} but notationally more cumbersome.  It suffices to replace within the proof of Proposition \ref{MP} the several instances of $H(\Delta/K^2)$ with $C_{K,N/\ell^{\gamma_\ell},M/\ell^{\alpha_\ell}} (t,q,d)$ to pass to the proof of Proposition \ref{MP2}. 

Applying Proposition \ref{MP2} recursively, one checks that the result matches the definition of the class numbers $H_{N,M}(t,q,d)$ in \eqref{cndef}. Therefore  
\es{ C_{1,N,M}(t,q,d) = 2\frac{\psi(N^2)}{\psi(N^2/M^2)}\sum_{\Lambda \mid (L/M)} \frac{\varphi(\Lambda^2)\varphi(N/(M\Lambda))}{\varphi(N/M)}H_{N,\Lambda M}(t,q,d)} and \est{T^{(e)}(t) = \frac{\psi(N^2)}{\psi(N^2/M^2)}\sum_{\Lambda \mid (L/M)} \frac{\varphi(\Lambda^2)\varphi(N/(M\Lambda))}{\varphi(N/M)}\left( H_{N,\Lambda M}(t,q,d) + (-1)^k H_{N,\Lambda M}(t,q,-d) \right).} 
Note from the definition that $H_{N,M}(t,q,-d)=H_{N,M}(-t,q,d)$ and also that $U_{k-2}(t,q)$ is an even (resp.~odd) function of $t$ when $k$ is even (resp.~odd).  Therefore 
\es{\frac{1}{2} \sum_{t^2<4q } U_{k-2}(t,q)\left( H_{N,\Lambda M}(t,q,d) + (-1)^k H_{N,\Lambda M}(t,q,-d) \right)  =  \sum_{t^2<4q } U_{k-2}(t,q)H_{N,\Lambda M}(t,q,d)} and \est{ T^{(e)} = \frac{\psi(N^2)}{\psi(N^2/M^2)}\sum_{\Lambda \mid (L/M)} \frac{\varphi(\Lambda^2)\varphi(N/(M\Lambda))}{\varphi(N/M)}\sum_{t^2<4q } U_{k-2}(t,q) H_{N,\Lambda M}(t,q,d),} 
as claimed in the statement of Theorem \ref{ESTF}.

\section{Comparison of Class Number Sums}\label{assembly}
Recall the definitions of $\omega_{A}(q,d)$ and $\omega_{A}^*(q,d)$ from \eqref{omegadef} and \eqref{omega*def}.  The following expression was the main result of Section \ref{curves2classnumbers}, expressing elliptic curve counts in terms of class numbers.  See Proposition \ref{s3mp}.  We have
\begin{multline}\label{section3result}\E_q(U_{k-2}(t,q) \Phi_A) = \frac{1}{q}\omega_{A}(q,d) + \frac{1}{q}\omega^*_{A}(q,1) \\ +  q^{k/2-1}\frac{(p-1)(k-1)}{24q} \left( \delta_{n_1(A)} (\sqrt{q}, 1) + (-1)^k \delta_{n_1(A)} (\sqrt{q}, -1) \right).
\end{multline} 

The main result of Section \ref{traces}, Theorem \ref{ESTF}, expresses $\tr(T_q \langle d \rangle | S_k(n_1,n_2))$ as a sum of similar class numbers.  Specifically, for $n_2 \mid n_1$ let 
\es{
\Sigma_{n_1,n_2}(q,d)\eqdef \sum_{t^2<4q } U_{k-2}(t,q)H_{n_1,n_2}(t,q,d),
} 
and if $q=p^{-1}$ we set $\Sigma_{n_1,n_2}(q,d)=0$.  Recall the definition of $\phi(n)$ as the Dirichlet convolution inverse to $\varphi(n^2)$ and the definition of  $T_{n_1,\lambda}(q,d)$ in terms of traces of Hecke operators.  Suppose $d^2q\equiv 1 \pmod*{n_2}$.  The trace formula, Theorem \ref{ESTF}, and the definition of $\phi$ imply that 
\es{\label{tfresult}
\Sigma_{n_1,n_2}(q,d) =  \frac{1}{\varphi(n_1/n_2)} \sum_{\nu \mid \frac{(d^2q-1,n_1)}{n_2}} \phi(\nu) T_{n_1,n_2 \nu}\left(q,d\right).
}  
The goal of this section is the comparison of the expressions \eqref{section3result} and \eqref{tfresult} given by Proposition \ref{AnsAlpha} below.  The main result of the paper, Theorem \ref{MT}, follows directly from \eqref{section3result}, \eqref{tfresult}, and Proposition \ref{AnsAlpha}.
\begin{proposition}\label{AnsAlpha}
Let $A$ a finite abelian group of rank at most $2$ and $q$ be a power of a prime $p$ such that $(q,|A|)=1$. Then we have 
\[
\omega_{A}(q,1)+ \omega^*_{A}(q,1) = \Sigma_{n_1(A),n_2(A)}(q,1) -p^{k-1}\Sigma_{n_1(A),n_2(A)}(q/p^2,p).
\]
\end{proposition}
\begin{proof} 
To lessen the notational burden, within this proof we write $U(t,q)=U_{k-2}(t,q)$ for the Chebyshev polynomials, $n_1=n_1(A)$ and $n_2=n_2(A)$, and take the $d$ implicit in the $D(t;n)$ notation to be $d=1$.  Let $q=p^v$ with $v \in \Z_{\geq 0 }$.  To prove Proposition \ref{AnsAlpha} it suffices to show that 
\begin{multline}\label{lem12e1}
\Sigma_{n_1,n_2}(q,1) = \sum_{0 \leq 2i<v} (p^i)^{k-1}
\left( \omega_{A}(q/p^{2i},p^i) + \omega^*_{A}(q/p^{2i},p^i)\right) \\ 
+ \frac{1}{2}p^{v/2}\delta(n_2,1)\delta_2(v,0)\left(\frac{1}{2}U(0,q)D(0;n_1)  + \frac{1}{3}U(q^{1/2},q)D(q^{1/2},n_1) \right. \\ 
\left.+ \frac{1}{3}U(-q^{1/2},q)D(-q^{1/2},n_1)\right)
\end{multline}
and 
\begin{multline}\label{lem12e2}
p^{k-1}\Sigma_{n_1,n_2}(q/p^2,p) = \sum_{0 < 2i<v} (p^i)^{k-1}\left( \omega_{A}(q/p^{2i},p^i) + \omega^*_{A}(q/p^{2i},p^i)\right) \\ 
+ \frac{1}{2}p^{v/2}\delta(n_2,1)\delta_2(v,0)\left(\frac{1}{2}U(0,q)D(0;n_1) + \frac{1}{3}U(q^{1/2},q)D(q^{1/2},n_1) \right. \\ 
\left. + \frac{1}{3}U(-q^{1/2},q)D(-q^{1/2},n_1)\right).
\end{multline}

\begin{lemma}\label{omegalem}
Let $q=p^v$.  For any $i$ satisfying $0\leq 2i<v$, we have that 
\es{\label{omegaiformula} (p^i)^{k-1} \omega_{A}(q/p^{2i},p^id) = \sum_{\substack{ v_p (t)=i \\ t^2<4q }}U(t,q) H_{n_1,n_2}(t,q,d).}  
\end{lemma}
\begin{proof}
Rearranging terms in the definition of $U_{k-2}(t,q)$ shows that
\es{\label{omegalemproofeq1} (p^i)^{k-1} U_{k-2}(t,q/p^{2i}) = p^i U_{k-2}(p^it,q).} 
We claim that if $v_p(q/p^{2i})\geq 1$ and $p\nmid t$ then
\es{\label{omegalemproofeq2}
p^iH_{n_1,n_2}(t,q/p^{2i},p^i) = H_{n_1,n_2}(p^i t,q,1).
}  
The lemma follows immediately after the change of variables $p^it \to t$.  

We now prove the claim. For any $(n,q)=1$ such that $n^2 \mid ((p^it)^2-4q)$ we have 
\begin{align*} 
H\left( \frac{(p^it)^2 -4q}{n^2} \right) = & \sum_{\delta^2 \mid \frac{(p^it)^2 -4q}{n^2} } h_w \left(  \frac{(p^it)^2 -4q}{n^2\delta^2}\right) \\ 
=&  \sum_{\kappa= 0}^{i} \sum_{\substack{\delta^2 \mid \frac{(p^{i-\kappa}t)^2 -4q/p^{2\kappa}}{n^2}\\ v_p(\delta) = 0} } h_w \left(  \frac{(p^{i-\kappa}t)^2 -4q/p^{2\kappa}}{n^2\delta^2}\right).
\end{align*}
For any $\kappa$ satisfying $0\leq \kappa \leq i-1$ we apply Lemma \ref{cor728} with $d = \frac{t^2-4q/p^{2i}}{n^2}$ and $f=p^{i-\kappa}$.  Since $(n,p)=1$ and $p\nmid t$ we have that $\frac{t^2-4q/p^{2i}}{n^2}$ is a non-zero square modulo $p$.  This implies
\begin{multline*}
H\left( \frac{(p^it)^2 -4q}{n^2} \right)  =  \sum_{\substack{\delta^2 \mid \frac{t^2 -4q/p^{2i}}{n^2}\\ v_p(\delta) = 0} } h_w \left(  \frac{t^2 -4q/p^{2i}}{n^2\delta^2}\right)\\ + \sum_{\kappa= 0}^{i-1}p^{i-\kappa} \left(1-\frac{1}{p}\right) \sum_{\substack{\delta^2 \mid \frac{t^2 -4q/p^{2i}}{n^2}\\ v_p(\delta) = 0} } h_w \left(  \frac{t^2 -4q/p^{2i}}{n^2\delta^2}\right) \\ 
=   p^i H\left( \frac{t^2-4q/p^{2i}}{n^2}\right).
\end{multline*}
We also have that $n \mid (p^{-i}d q - p^{-i}d^{-1} -t)$ if and only $n\mid (dq - d^{-1} -p^it)$ since $(n,p)=1$.  This completes the proof of the claim.
\end{proof}

\vspace{12pt}
\noindent {\bf \emph{Verifying \eqref{lem12e1}.}}
\vspace{12 pt}

We prove that equation \eqref{lem12e1} holds by considering several cases.  Recall that $q = p^v$.

If $q=p^v$ with $v$ even, then taking the sum of \eqref{omegaiformula} over $i$ gives
\begin{multline}\label{veven}
\Sigma_{n_1,n_2}(q,1) = \sum_{0 \leq 2i<v} (p^i)^{k-1}\omega_{A}(q/p^{2i},p^i) + U(-q^{1/2},q)H_{n_1,n_2}(-q^{1/2},q,1)\\ + U(0,q)H_{n_1,n_2}(0,q,1)  + U(q^{1/2},q)H_{n_1,n_2}(q^{1/2},q,1).
\end{multline}
If $q = p^v$ with $v$ odd, then taking the sum of \eqref{omegaiformula} over $i$ gives
\begin{multline}\label{vodd}
\Sigma_{n_1,n_2}(q,1) = \sum_{0 \leq 2i<v} (p^i)^{k-1}\omega_{A}(q/p^{2i},p^i)   + U(0,q) H_{n_1,n_2}(0,q,1)
\\
 + \begin{cases} 0& \text{ if } p \neq 2,3 \\  
 U(-2^{\frac{v+1}{2}},q) H_{n_1,n_2}(-2^{\frac{v+1}{2}},q,1) + U(2^{\frac{v+1}{2}},q) H_{n_1,n_2}(2^{\frac{v+1}{2}},q,1) & \text{ if } p =2 \\ 
 U(-3^{\frac{v+1}{2}},q) H_{n_1,n_2}(-3^{\frac{v+1}{2}},q,1)   + U(3^{\frac{v+1}{2}},q) H_{n_1,n_2}(3^{\frac{v+1}{2}},q,1) & \text{ if } p = 3. 
 \end{cases}
 \end{multline}
Similarly, it follows from Lemma \ref{omegalem} when $v$ is even  that 
\begin{multline}\label{veven2}
p^{k-1}\Sigma_{n_1,n_2}(q/p^2,p) = \sum_{0 < 2i<v} (p^i)^{k-1}\omega_{A}(q/p^{2i},p^i) \\ 
+ p^{k-1}\left( U(-q^{1/2}/p,q/p^2)H_{n_1,n_2}(-q^{1/2}/p,q/p^2,p) + U(0,q/p^2)H_{n_1,n_2}(0,q/p^2,p) \right. \\ 
\left. + U(q^{1/2}/p,q/p^2)H_{n_1,n_2}(q^{1/2}/p,q/p^2,p)\right),
\end{multline}
and when $v$ is odd that
\begin{multline}\label{vodd2}p^{k-1}
\Sigma_{n_1,n_2}(q/p^2,p) = \sum_{0 < 2i<v} (p^i)^{k-1}\omega_{A}(q/p^{2i},p^i)   + p^{k-1}U(0,q/p^2) H_{n_1,n_2}(0,q/p^2,p)
\\
 + p^{k-1}\begin{cases} 0& \text{ if } p \neq 2,3 \\  U(-2^{\frac{v-1}{2}},2^{v-2}) H_{n_1,n_2}(-2^{\frac{v-1}{2}},2^{v-2},2) \\ \hspace{48pt} + U(2^{\frac{v-1}{2}},2^{v-2}) H_{n_1,n_2}(2^{\frac{v-1}{2}},2^{v-2},2) & \text{ if } p =2 \\ 
 U(-3^{\frac{v-1}{2}},3^{v-2}) H_{n_1,n_2}(-3^{\frac{v-1}{2}},3^{v-2},3)  \\ \hspace{48pt} + U(3^{\frac{v-1}{2}},3^{v-2}) H_{n_1,n_2}(3^{\frac{v-1}{2}},3^{v-2},3) & \text{ if } p = 3. 
 \end{cases}
 \end{multline}
 
 Returning to \eqref{veven} and  \eqref{vodd} we calculate from the definition of $H_{n_1, n_2}(t,q,d)$ that:
 \begin{enumerate}
 \item We have
 \begin{multline}\label{extreme1} 
 H_{n_1,n_2}(0,q,1) = \begin{cases} 
 0 & \text{ if } n_2>2 \\ 
 \frac{1}{2} H(-q) D(0,2n_1) & \text{ if } n_2=2 \\
 \frac{1}{2} H(-4q) D(0,n_1) \\ \hspace{16pt} - \delta_4(n_1,0) \frac{1}{2} H(-q) \left( D(0,n_1) - D(0,2n_1)\right) & \text{ if } n_2 =1.
 \end{cases} 
 \end{multline}
\item If $v$ even, then 
 \begin{multline}\label{extreme2} 
 H_{n_1,n_2}(\pm q^{1/2} , q,1) = 
 \begin{cases}
 0 & \text{ if } n_2>1 \\ 
 \frac{1}{2} H(-3q) D(\pm q^{1/2},n_1) & \text{ if } n_2 =1. 
 \end{cases} \end{multline}
\item If $v$ odd and $p=2$, then 
\begin{multline}\label{extreme3} H_{n_1,n_2}(\pm 2^{\frac{v+1}{2}}, q,1) = \begin{cases}
0 & \text{ if } n_2>1 \\ 
\frac{1}{2} H(-2q) D(\pm 2^{\frac{v+1}{2}},n_1) & \text{ if } n_2 =1.
\end{cases} \end{multline}
\item If $v$ odd and $p=3$, then 
 \begin{multline}\label{extreme4} 
 H_{n_1,n_2}(\pm 3^{\frac{v+1}{2}}, q,1) = \begin{cases}
0 & \text{ if } n_2>1 \\ 
\frac{1}{2} H(-q)D(\pm 3^{\frac{v+1}{2}},n_1) & \text{ if } n_2 =1.
\end{cases} 
\end{multline}
\end{enumerate}

We check that \eqref{lem12e1} is true by separating into cases $n_2>2, n_2=2,$ and $n_2=1$.

\vspace{12pt}
\noindent {\bf \emph{The case $n_2 > 2$.}}
\vspace{12 pt}

If $n_2 > 2$ then \eqref{extreme1}, \eqref{extreme2}, \eqref{extreme3}, and \eqref{extreme4} all vanish, but so does $\omega^*_{A}(q,1)$ by definition of $H^*_{n_1,n_2}(t,q,d)$.  Therefore we have verified \eqref{lem12e1} in the case that $n_2>2$.  
 
 \vspace{12pt}
\noindent {\bf \emph{The case $n_2 = 2$.}}
\vspace{12 pt}
 
 If $n_2=2$ then of \eqref{extreme1}, \eqref{extreme2}, \eqref{extreme3}, and \eqref{extreme4} only $H_{n_1,n_2}(0,q,1)=\frac{1}{2} H(-q) \delta_{2n_1}(q+1,0)$ can be non-zero.  If 
$\delta_{2n_1}(q+1,0) = 1$, then $q\equiv 3 \pmod*4$, and $v$ must be odd.   Then  Lemma \ref{cor728}, \eqref{omegalemproofeq1}, and $(n,p)=1$ imply that
 \begin{align*} 
 U(0,q)H_{n_1,n_2}(0,q,1)= & \frac{1}{2} h_w(-p) \sigma(p^{\frac{v-1}{2}}) U(0,q) \delta_{2n_1}(q+1,0) \\ 
 = &  \frac{1}{2} h_w(-p) \sum_{0\leq 2i <v } (p^i)^{k-1} U(0,q/p^{2i})  \delta_{2n_1}(p^iq/p^{2i}+p^{-i},0). 
 \end{align*}
The definition \eqref{H*2} of $H^*_{n_1,n_2}$ and the definition of $\omega^*_{A}$ imply that 
\begin{align} \label{n2calc}
U(0,q)H_{n_1,n_2}(0,q,1) = & \sum_{0\leq 2i <v } (p^i)^{k-1} \sum_{t^2<4q} U(t,q/p^{2i}) H_{n_1,n_2}^*(t,q/p^{2i},p^i) \nonumber \\
 = & \sum_{0\leq 2i <v } (p^i)^{k-1} \omega^*_{A}(q/p^{2i},p^i).
 \end{align}
Therefore we have verified \eqref{lem12e1} in the case that $n_2=2$.
 
 \vspace{12pt}
\noindent {\bf \emph{The case $n_2 = 1$ and $v$ even.}}
\vspace{12 pt}
 
 In this case, $q\equiv 1 \pmod*4$ and we have by Lemma \ref{cor728} and \eqref{extreme1} that 
 \est{ 
 H_{n_1,1}(0,q,1) =  \frac{1}{4}\left( \sigma(p^{v/2-1})  \left(1-\legen{-4}{p}\right) + p^{v/2}\right) D(0,n_1).
 } 
 This expression together with \eqref{H*1even} implies that 
 \es{\label{UH1}U(0,q) H_{n_1,1}(0,q,1) 
 =  \sum_{0\leq 2i < v} (p^i)^{k-1} U(0,q/p^{2i}) H^*_{n_1,1}(0,q/p^{2i},p^i) + \frac{1}{4} p^{v/2} U(0,q) D(0,n_1).
 }  
Likewise, if $t = \pm q^{1/2}$ and $v$ is even, we apply Lemma \ref{cor728} to \eqref{extreme2} to find that 
\est{ 
H_{n_1,1}(\pm q^{1/2},q,1) = \frac{1}{6}\left( \sigma(p^{v/2-1}) \left(1-\legen{-3}{p}\right) + p^{v/2}\right) D(\pm q^{1/2},n_1).
} 
By \eqref{H*1even} we have 
\begin{multline}\label{UH2} 
U(\pm q^{1/2},q) H_{n_1,1}(\pm q^{1/2},q,1) 
 =  \sum_{0\leq 2i < v} (p^i)^{k-1} U(\pm q^{1/2},q/p^{2i}) H^*_{n_1,1}(\pm q^{1/2},q/p^{2i},p^i) \\  
 + \frac{1}{6} p^{v/2} U(\pm q^{1/2},q) D(\pm q^{1/2},n_1).
 \end{multline} 
 Putting together \eqref{veven}, \eqref{UH1}, and \eqref{UH2} we verify \eqref{lem12e1} in the case that $v$ is even and $n_2=1$.  
 
\vspace{12pt}
\noindent {\bf \emph{The case $n_2 = 1$ and $v$ odd.}}
\vspace{12 pt}
 
In this case, we have by \eqref{extreme1} and Lemma \ref{cor728} that 
\est{
H_{n_1,1}(0,q,1) = \frac{1}{2}\sigma(p^{\frac{v-1}{2}}) \left( H(-4p)D(0,n_1) - \delta_4(n_1,0)  h_w(-p) \left(D(0,n_1) - D(0,2n_1)\right)\right).
} 
It follows from \eqref{H*1odd} and \eqref{omegalemproofeq1} that 
\begin{align}\label{UH3} 
U(0,q) H_{n_1,1}(0,q,1) 
 = & \sum_{0 \leq 2i <v } (p^i)^{k-1} U(0,q/p^{2i}) H_{n_1,1}^*(0,q/p^{2i},p^i).\end{align}
If $n_2=1$, $v$ is odd, and $p=2$ then by \eqref{extreme3} and Lemma \ref{cor728}, we have that 
 \est{  H_{n_1,n_2}(\pm 2^{\frac{v+1}{2}}, q,1) = \frac{1}{4} \sigma (2^{\frac{v-1}{2}}) D(\pm 2^{\frac{v+1}{2}},n_1). } 
This expression together with \eqref{H*1odd} and \eqref{omegalemproofeq1} imply that 
 \begin{align}\label{UH42}
 U(\pm 2^{\frac{v+1}{2}},q) H_{n_1,n_2}(\pm 2^{\frac{v+1}{2}}, q,1) 
 = &\sum_{0 \leq 2i <v } (p^i)^{k-1} U(\pm 2^{\frac{v+1}{2}},q/p^{2i}) H_{n_1,1}^*(\pm 2^{\frac{v+1}{2}},q/p^{2i},p^i).
 \end{align}
 
 Similarly, if $n_2=1$, $v$ is odd, and $p=3$ we have that 
 \es{\label{UH43} 
 U(\pm 3^{\frac{v+1}{2}},q) H_{n_1,n_2}(\pm 3^{\frac{v+1}{2}}, q,1) =\sum_{0 \leq 2i <v } (p^i)^{k-1} U(\pm 3^{\frac{v+1}{2}},q/p^{2i}) H_{n_1,1}^*(\pm 3^{\frac{v+1}{2}},q/p^{2i},p^i).
 }  
 Putting together \eqref{vodd} and \eqref{UH3} we verify \eqref{lem12e1} in the case that $p \neq 2,3$, $v$ is odd, and $n_2=1$.  If $p=2$, $v$ is odd, and $n_2=1$ we apply \eqref{vodd}, \eqref{UH3}, and \eqref{UH42} to verify \eqref{lem12e1}.  In the case $p=3$, $v$ is odd, and $n_2=1$ we apply \eqref{vodd}, \eqref{UH3}, and \eqref{UH43} to verify \eqref{lem12e1}.  We have now verified \eqref{lem12e1} in all cases.
 
 \vspace{12pt}
\noindent {\bf \emph{Verifying \eqref{lem12e2}.}}
\vspace{12 pt}
 
To verify this equation we return to \eqref{veven2} and \eqref{vodd2} and perform similar calculations to those given above.  

\vspace{12pt}
\noindent {\bf \emph{The case $n_2 > 2$.}}
\vspace{12 pt}
 
In this case, each of the terms \eqref{extreme1}, \eqref{extreme2}, \eqref{extreme3}, and \eqref{extreme4} vanish, but so do $H^*_{n_1,n_2}(t,q/p^2,p)$ and $\omega_{A}^*(q/p^2,p)$, so we have verified \eqref{lem12e2} in this case.

\vspace{12pt}
\noindent {\bf \emph{The case $n_2 = 2$.}}
\vspace{12 pt}
 
In this case, by \eqref{omegalemproofeq1} and \eqref{extreme1} 
\est{ 
p^{k-1} U(0,q/p^2) H_{n_1,n_2}(0,q/p^2,p) 
 =  p U(0,q) \frac{1}{2} H(-q/p^2) D(0,2n_1) 
 } 
 and by Lemma \ref{cor728} we have 
 \est{ 
 p\frac{1}{2} H(-q/p^2) = \frac{1}{2}p \sigma(p^{\frac{v-3}{2}}) h_w(-p) = \frac{1}{2} \left(\sigma(p^{\frac{v-1}{2}})-1\right) h_w(-p)  = \frac{1}{2}H(-q) - \frac{1}{2}h_w(-p).
 } 
 A similar argument to the one that gave \eqref{n2calc} implies that  
 \begin{align*} & \,pU(0,q)H(0,q/p^2,p) \\
 = & \, U(0,q) \frac{1}{2}H(-q) D(0,2n_1) - U(0,q) \frac{1}{2}h_q(-p)D(0,2n_1) \\ 
 = &\,  \sum_{0 \leq 2i<v} (p^i)^{k-1} \sum_{t^2<4q} U(t,q/p^{2i}) H_{n_1,2}^*(t,q/p^{2i},p^i) - \frac{1}{2}h_w(-p) U(0,q)  D(0,2n_1) \\
 = &\,  \sum_{0 < 2i<v} (p^i)^{k-1} \sum_{t^2<4q} U(t,q/p^{2i}) H_{n_1,2}^*(t,q/p^{2i},p^i),\end{align*}
 which verifies \eqref{lem12e2} in the case $n_2=2$.
 
 \vspace{12pt}
\noindent {\bf \emph{The case $n_2 = 1$.}}
\vspace{12 pt}

By \eqref{extreme1} and \eqref{omegalemproofeq1} we have 
\begin{multline*}
p^{k-1} U(0,q/p^2)H_{n_1,1}(0,q/p^2,p)  = pU(0,q) \left( \frac{1}{2} H(-4q/p^2)D(0,n_1) \right. \\ \left.- \delta_4(n_1,0) \frac{1}{2}H(-q/p^2) \left( D(0,n_1) -D(0,2n_1) \right)\right).
\end{multline*} 
Lemma \ref{cor728} implies that 
\est{p\frac{1}{2}H(-q/p^2) = \frac{1}{2}H(-q)-\frac{1}{2} h_w(-p) } 
and 
\est{ p\frac{1}{2} H(-4q/p^2) = \frac{1}{2}H(-4q)-
\begin{cases} 
\frac{1}{4} \left(1-\legen{-4}{p}\right)  & \text{ if } v \text{ even} \\ 
\frac{1}{2}H(-4p) & \text{ if } v \text{ odd.}
\end{cases}} 
Therefore, by \eqref{UH1} and \eqref{UH3} we have that  
\begin{multline*} 
p^{k-1} U(0,q/p^2)H_{n_1,1}(0,q/p^2,p) = \sum_{0\leq 2i < v} (p^i)^{k-1} U(0,q/p^{2i}) H^*_{n_1,1}(0,q/p^{2i},p^i) \\ 
+ \frac{1}{4} p^{v/2} U(0,q) D(0,n_1)   - U(0,q) D(0,n_1) 
\begin{cases} 
\frac{1}{4} \left(1-\legen{-4}{p}\right)  & \text{ if } v \text{ even } \\ \frac{1}{2}H(-4p) & \text{ if } v \text{ odd }\end{cases}  \\ + \delta_4(n_1,0) \frac{1}{2} h_w(-p) \left( D(0,n_1)  - D(0,2n_1)  \right).
\end{multline*} 
Applying \eqref{H*1even} or \eqref{H*1odd} to this expression gives
\begin{multline}\label{t0/p} 
p^{k-1} U(0,q/p^2)H_{n_1,1}(0,q/p^2,p) = \sum_{0< 2i < v} (p^i)^{k-1} U(0,q/p^{2i}) H^*_{n_1,1}(0,q/p^{2i},p^i) \\ 
+ \frac{1}{4} p^{v/2} U(0,q) D(0,n_1) .
\end{multline} 
 
 \vspace{12pt}
\noindent {\bf \emph{The case $n_2 = 1$ and $v$ even.}}
\vspace{12 pt}
 
By \eqref{omegalemproofeq1}, \eqref{extreme2}, and $(n_1,p)=1$ we have 
 \est{p^{k-1} 
 U(\pm q^{1/2}/p) H_{n_1,n_2}(\pm q^{1/2}/q,q/p^2,p) = pU(\pm q^{1/2},q) \frac{1}{2} H(-3q/p^2)D(\pm q^{1/2},n_1).} 
 By Lemma \ref{cor728} we have 
 \begin{align*} 
 p\frac{1}{2} H(-3q/p^2) = & \frac{1}{6} \left( p\sigma (p^{\frac{v}{2}-2})  \left(1-\legen{-3}{p}\right) +p^{v/2}\right) \\
 = & \frac{1}{6} \left( \left(\sigma(p^{\frac{v}{2}-1})-1\right) \left(1-\legen{-3}{p}\right) + p^{v/2}\right) \\ 
 = & \frac{1}{2} H(-3q) -\frac{1}{6} \left(1-\legen{-3}{p}\right).
 \end{align*}
 Using these last two equations we find by \eqref{UH2} that 
 \begin{multline*} 
 p^{k-1} U(\pm q^{1/2}/p) H_{n_1,n_2}(\pm q^{1/2}/q,q/p^2,p)\\  = \sum_{0 \leq 2i < v } (p^i)^{k-1} U(\pm q^{1/2},q/p^{2i}) H_{n_1,1}^*(\pm q^{1/2} , q/p^{2i},p^i)  
 + \frac{1}{6} p^{v/2} U( \pm q^{1/2},q) D(\pm q^{1/2},n_1) \\ - U(\pm q^{1/2},q) D(\pm  q^{1/2},n_1) \frac{1}{6} \left(1-\legen{-3}{p}\right).
 \end{multline*}
 By the definition \eqref{H*1even} of $H^*_{n_1,1}$ this simplifies to 
 \begin{multline}\label{tsqrtq/p} 
 p^{k-1} U(\pm q^{1/2}/p) H_{n_1,n_2}(\pm q^{1/2}/q,q/p^2,p) \\ = \sum_{0 < 2i < v } (p^i)^{k-1} U(\pm q^{1/2},q/p^{2i}) H_{n_1,1}^*(\pm q^{1/2} , q/p^{2i},p^i)  
 + \frac{1}{6} p^{v/2} U( \pm q^{1/2},q) D(\pm q^{1/2},n_1).
 \end{multline} 
 Putting together \eqref{vodd2}, \eqref{t0/p}, and \eqref{tsqrtq/p} verifies \eqref{lem12e2} when $v$ is even.
 
  \vspace{12pt}
\noindent {\bf \emph{The case $n_2 = 1$ and $v$ odd.}}
\vspace{12 pt}
 
 If $v$ is odd and $p=2$ then by $(n_1,2)=1$, \eqref{omegalemproofeq1}, and \eqref{extreme3} we have 
 \est{ 
 2^{k-1} U(\pm 2^{\frac{v-1}{2}},2^{v-2}) H_{n_1,1}(\pm 2^{\frac{v-1}{2}},2^{v-2},2) = 2U(\pm 2^{\frac{v+1}{2}},2^v) \frac{1}{2} H(-2q/2^2) D(\pm 2^{\frac{v+1}{2}},n_1).
 }  
We have by Lemma \ref{cor728} that 
\est{ 2\frac{1}{2}H(-2q/2^{2}) = \frac{1}{4} \left(\sigma(2^{\frac{v-1}{2}}) -1\right) = \frac{1}{2}H(-2q) -\frac{1}{4}.
}  
Therefore by \eqref{UH42} we have 
\begin{multline*}
2^{k-1} U(\pm 2^{\frac{v-1}{2}},2^{v-2}) H_{n_1,1}(\pm 2^{\frac{v-1}{2}},2^{v-2},2) \\ = \sum_{0\leq 2i < v} (p^i)^{k-1} U(\pm 2^{\frac{v+1}{2}},q/p^{2i}) H_{n_1,1}^*(\pm 2^{\frac{v+1}{2}},q/p^{2i},p^i) 
- \frac{1}{4} U(\pm 2^{\frac{v+1}{2}},2^{v}) D(2^{\frac{v+1}{2}},n_1).
\end{multline*} 
By the definition \eqref{H*1odd} of $H_{n_1,1}^*(t,q,d)$ this simplifies to 
\begin{multline}\label{final2}2^{k-1} 
U(\pm 2^{\frac{v-1}{2}},2^{v-2}) H_{n_1,1}(\pm 2^{\frac{v-1}{2}},2^{v-2},2) \\ = \sum_{0< 2i < v} (p^i)^{k-1} U(\pm 2^{\frac{v+1}{2}},q/p^{2i}) H_{n_1,1}^*(\pm 2^{\frac{v+1}{2}},q/p^{2i},p^i).
\end{multline}  
Similarly, if $v$ is odd and $p=3$ we have by \eqref{extreme3} and \eqref{UH43} that 
\begin{multline}\label{final3}3^{k-1} 
U(\pm 3^{\frac{v-1}{2}},3^{v-2}) H_{n_1,1}(\pm 3^{\frac{v-1}{2}},3^{v-2},3)\\ =  \sum_{0< 2i < v} (p^i)^{k-1} U(\pm 3^{\frac{v+1}{2}},q/p^{2i}) H_{n_1,1}^*(\pm 3^{\frac{v+1}{2}},q/p^{2i},p^i).
\end{multline}  
Combining \eqref{vodd2}, \eqref{t0/p}, \eqref{final2}, and \eqref{final3} verifies \eqref{lem12e2} when $v$ is odd.
\end{proof}

\section{Acknowledgements}
The authors thank Alina Cojocaru, Noam Elkies, Jordan Ellenberg, and Corentin Perret-Gentil for helpful conversations related to this project, and also the anonymous referee whose detailed report greatly improved this paper.

\end{document}